\documentclass[showinstructions,faculty=fw,department=wis,phddegree=mat,final]
{adsphd}

\title{\textcolor{black}{{ Groupoids and singular foliations}}}

\author{Alfonso Gadir}{Garmendia Gonz\'alez}

\supervisor{{Prof. Marco Zambon}}{}
\president{Prof. Paul Igodt}
\jury{Prof. Karel Dekimpe\\
	Prof. Gabor Szabo\\
	Prof. Joeri Van der Veken
}
\externaljurymember{Prof. Camille Laurent-Gengoux}{Univ. de Lorraine}

\researchgroup{Geometry Section}
\website{http://wis.kuleuven.be/meetkunde} 
\email{alfonso.garmendia@kuleuven.be} 

\address{Celestijnenlaan 200B box 2400}

\date{July 2019}
\copyyear{2019}

\usepackage{amsthm}
\usepackage{txfonts}
\usepackage{mathtools}
\usepackage{color}
\usepackage{hyperref} 
\usepackage{tikz-cd}
\usetikzlibrary{babel}
\usepackage[all,poly,web,knot]{xy}
\newtheorem{theorem}{Theorem}[section]
\newtheorem{lemma}[theorem]{Lemma}
\newtheorem{proposition}[theorem]{Proposition}
\newtheorem{corollary}[theorem]{Corollary}
\newtheorem{cor}[theorem]{Corollary}
\newtheorem{definition}[theorem]{Definition}
\newtheorem{defi}[theorem]{Definition}
\newtheorem{ex}[theorem]{Example}
\newtheorem{con}[theorem]{Convention}
\newtheorem{rem}[theorem]{Remark}
\newtheorem{prop}[theorem]{Proposition}
\newtheorem{lem}[theorem]{Lemma}
\newtheorem{thm}[theorem]{Theorem}

\newcommand{\NN}{\mathbb{N}}
\newcommand{\ZZ}{\mathbb{Z}}

\newcommand{\RR}{\mathbb{R}}

\newcommand{\vX}{\mathfrak{X}}

\newcommand{\bt}{\mathbf{t}}                  
\newcommand{\bs}{\mathbf{s}}  

\newcommand{\CR}{\mathcal{R}}
\newcommand{\CL}{\mathcal{L}}

\newcommand{\CV}{\mathcal{V}}

\newcommand{\CK}{\mathcal{K}}
\newcommand{\CC}{\mathcal{C}}
\newcommand{\CSH}{\mathcal{SH}}

\newcommand{\CX}{\mathfrak{X}}
\newcommand{\CS}{\mathcal{S}}
\newcommand{\CF}{\mathcal{F}}
\newcommand{\cF}{\mathcal{F}}
\newcommand{\CU}{\mathcal{U}}
\newcommand{\CH}{\mathcal{H}}
\newcommand{\CG}{\mathcal{G}}
\newcommand{\CCI}{\mathcal{I}}
\newcommand{\CP}{\mathcal{P}}
\newcommand{\g}{\mathfrak{g}}
\newcommand{\h}{\mathfrak{h}}

\newcommand{\CI}{C^\infty}

\newcommand{\SEC}{\Gamma}

\usepackage{scalerel,stackengine}
\stackMath
\newcommand\reallywidehat[1]{%
	\savestack{\tmpbox}{\stretchto{%
			\scaleto{%
				\scalerel*[\widthof{\ensuremath{#1}}]{\kern.1pt\mathchar"0362\kern.1pt}%
				{\rule{0ex}{\textheight}}
			}{\textheight}%
		}{2.4ex}}%
	\stackon[-6.9pt]{#1}{\tmpbox}%
}

\newcommand{\smallwhitestar}{\star}
\newcommand{\cev}[1]{\reflectbox{\ensuremath{\vec{\reflectbox{\ensuremath{#1}}}}}}

\newcommand{\de}{\partial}
\newcommand{\fto}{\rightarrow}
\newcommand{\soutar}{\rightrightarrows}

\newcommand{\st}{\hspace{.05in}:\hspace{.05in}}
\newcommand{\y}{\hspace{.05in}\text{and}\hspace{.05in}}
\newcommand{\Fr}{\colon}
\newcommand{\img}{\mathrm{Img}}

\begin{document}
	

	\maketitle
	
	\frontmatter 
	
	\chapter*{Acknowledgements}                                  \label{ch:preface}
	
	To those who have contributed directly to the development of this thesis:{\it
		\begin{itemize}
			\item[]
			First of all, I want to thank my advisor Marco Zambon who believed in me, a wanderer that passed by his office one random day in Madrid. His guidance has been a major contribution to the way I see math.
			\item[] The members of the jury, Prof. Paul Igodt, Karel Dekimpe, Gabor Szabo, Joeri Van der Veken and Camille Laurent-Gengoux for asking interesting questions and helping me become a better mathematician. I admire all of you in a different way each.
			\item[]
			The professors Iakovos Androulidakis and Georges Skandalis for  their research and key comments on this work.
			\item[]
			My colleague Stephane Geudens who read my whole thesis and helped me use the correct English expressions. I also want to thank him for his substantial comments on my jokes, which made them even funnier than before.
			\item[]
			My friend Ori Yudilevich with whom I discussed a large part of my research and with whom I wrote an article. I also want to thank him for being a source of jokes that I could follow up on, most of the time making them less funny than before.
	\end{itemize}}
	
	To those who helped me not to become crazy in the process:
	{\it
		\begin{itemize}
			\item[]
			My parents Carmen Gonzalez and Gadir Garmendia for all the moral support throughout my entire life.
			
			\item[]
			My family, especially Nathalia Garmendia, Adela Gonzalez, Carlos, Isabel and Lucia Ortiz for being on this side of the pond and making my life happier each time I visit you.
			
			\item[]
			My officemate Leyli Mammadova for all the grammar and math questions, pizza nights and for the motivating cat pictures. 
			
			\item[]
			My next door neighbor at work Marco Usula for all the math and life discussions and for introducing me to pasta bottarga and broccoli Parmesan (not inclusive). 
			
			\item[]
			My study partner  Robin Van der Veer; We followed three master courses together and he helped me understand the evaluation system and the subjects in class.
			
			\item[]
			My PhD companion Marcel Rubio, we started and are ending together. It was nice to have a friend through all these 4 years. You did it first and that made it easier for me.	
			
			\item[]
			My travel companion Marta Cobo, with whom I have visited many places in Belgium and with whom I learned how to make my life more active.	
			
			\item[]
			The société de jeux de société by Charlotte Kirchhoff-Lukat, Marilena Moruz, Joel Villatoro and Anne Wijffels. We enjoyed many excellent evenings together.
			
			\item[]
			My namesake Alfonso Tortorella, for his support during my job search, and his constant company for coffee.
			
			\item[]
			My Sparty group Sarah Nissen, Annelies van Vijver, Yamid Gomez and Rodolfo Mar\'in. For all the interesting conversations and activities in Spanish. Each one of you changed me one way or another. 	
			
			\item[] My "Bowling" team of Elisa Ledesma and Selyna Smekens for all those evenings together that helped me forget the stress of the thesis. 
			
			\item[] My old friend Rodrigo Anzola for supporting me all the time and making English learning more interesting. 
	\end{itemize}}
	
	I am lucky. Normally I should know less than half of you half as well as I do now; but life let me know all of you half as well as you deserve (maybe I lied about the not becoming crazy part).

	
	
	\chapter*{Abstract} \label{ch:abstract}
	
	This thesis has two objectives. The first objective is to introduce a notion of equivalence for singular foliations  that preserves their transverse geometry and is compatible with the notions of Morita equivalence of the holonomy groupoids and the transverse equivalence for regular foliations that appeared in the 1980's.

	The second one is to describe the structures behind quotients of singular foliations and to connect these results with their associated holonomy groupoids.

	We also want to give an introduction to the notion of singular foliations as given in \cite{AndrSk}, as well as to their relation with Lie groupoids and Lie algebroids.

	\instructionsabstract

	

	\chapter*{Beknopte samenvatting}
	
	Deze thesis heeft twee doelstellingen. De eerste bestaat erin een notie van equivalentie voor singuliere foliaties in te voeren, die hun
	transversale meetkunde bewaart. Deze equivalentierelatie wordt zodanig gedefinieerd dat de holonomiegroepoïden geassocieerd met equivalente singuliere foliaties Morita-equivalent zijn. Tevens veralgemeent ze noties van transversale equivalentie voor reguliere foliaties die in de jaren '80 verschenen.
	
	Het twede doel van deze thesis is het beschrijven van de achterliggende structuren van quotiënten van singuliere foliaties.
	
	We willen ook een inleiding geven op het begrip "singuliere foliaties", dat verscheen in \cite{AndrSk}, alsook de relatie tussen singuliere foliaties enerzijds, en Lie groepoïden en Lie algebroïden anderzijds bespreken.

	\cleardoublepage
	
	


	
	
	\chapter*{Contents}
	\makeatletter
	\@starttoc{toc}
	\makeatother
	
	
	\mainmatter  
	

	\chapter*{Introduction}\label{ch:introduction}
	
	I started my PhD with the project of quotients of singular foliation, a subject suggested by my adviser Marco Zambon. At that time, he gave me to study an article by Prof. Androulidakis and Prof. Skandalis \cite{AndrSk}, which I started reading right away. Trying to understand the proofs of every theorem we discovered step by step new properties and relations. After many discussions with him and my colleagues, this thesis was born. 
	
	This thesis studies properties of singular foliations, focusing on their associated holonomy groupoids. It tries to be a self contained work and to give an introduction to the notions of singular foliations, Lie algebroids and Lie groupoids, as well as their relations.
	
	\section*{The structure}
	
	The first three chapters provide an introduction to many important notions for this thesis, such as singular foliations, bisubmersions, Lie groupoids, Lie algebroids and Morita equivalence. The results of this thesis lie in the last two chapters, the fourth and the fifth ones.
	
	The first chapter of this thesis is an introduction to singular foliations in the sense of the article \cite{AndrSk}. Here we also explain in detail many of the notions from different points of view. The sections with new results are \S \ref{sec:trans.maps}, \S \ref{sec:prod.fol} and \S \ref{sec:fol.sheaf}.
	
	In the second chapter we give an introduction to bisubmersions, which first appeared in \cite{AndrSk} and inspired our notion of Hausdorff Morita equivalence for singular foliations, which we also present in the same chapter. The sections with new results are \S \ref{sec:aut.fol} \S \ref{sec:ME1}, \S \ref{subsec:firstinv}  and \S \ref{sec:exam}. In \S \ref{sec:aut.fol} we give a new proof of an old result. In  \S\ref{sec:ME1}, \S\ref{subsec:firstinv} and \S\ref{sec:exam} we  give an introduction of Hausdorff Morita equivalence for singular foliations, as in \cite{ME2018} by Marco Zambon and myself.
	
	The third chapter gives the preliminaries for Lie groupoids, Lie algebroids and Morita equivalence. We also connect these objects with singular foliations. There are no new results but \S\ref{sec:pull.grpd.algd} and \S \ref{sec:ME.gpd} are particularly important.
	
	The fourth chapter gives an introduction to the holonomy groupoid of a singular foliation, as in article \cite{AndrSk}. It also relates our notion of Hausdorff Morita equivalence for singular foliation with the classical notion of Morita equivalence for the holonomy groupoids. The sections with new results are \S \ref{sec:pullb.grpd}, \S \ref{sec:hol.trans} and \S \ref{sec:me.hol.grpd}, the last one being one of the most important sections of this thesis.
	
	In the fifth chapter we show the results we obtained for the quotients of singular foliations. The sections with new results are \S \ref{sec:Qfolman} and \S \ref{sec:general}.
	
	\section*{The motivation}
	
	This work is motivated mainly by the theory of foliations and Poisson geometry.
	
	\subsubsection*{Foliations} 
	
	The notion of regular foliation is widely studied in differential geometry, to the extent that it regularly appears in general master courses in mathematics. A regular foliation can be seen as  either one of two equivalent notions: integrable distributions or regular smooth partitions.
	
	\textbf{Definition.} \textit{An \textbf{integrable distribution} $D$ on a manifold $M$ is a vector subbundle of the tangent bundle $TM$, such that it is involutive i.e. $[\Gamma(D),\Gamma(D)]\subset \Gamma(D)$.}
	
	\textbf{Definition.} \textit{A \textbf{smooth partition} on $M$ is a partition $\{L_i\}_{i\in I}$ of $M$, where every class $L_i$ is a connected immersed\footnote{In this thesis any immersion is considered to be injective.} submanifold of $M$ called a leaf; and such that for each $p\in M$ and $v\in T_p M$ tangent to the leaves, there exists a vector field $X\in \CX(M)$ with $X(p)=v$ and tangent to every leaf.} \textit{A smooth partition is \textbf{regular} if every leaf has the same dimension.}	
	
	The notions of integrable distribution and regular smooth partition are equivalent because of the Frobenius theorem. Given a regular smooth partition, one can get an integrable distribution using the tangent spaces of the leaves. Moreover given an integrable distribution $D$ one can follow the flows of the vector fields in $\Gamma(D)\subset \CX(M)$ and get a regular smooth partition of $M$.
	
	An important notion for this thesis is the holonomy groupoid of a foliation. In the case of regular foliations, this groupoid was introduced by Ehresmann \cite{EhrGrpd} and Winkelnkemper \cite{WinkGrpd}. Its construction can be seen in section \ref{sec:hol.grpd} of this thesis.
	
	On the one hand, the holonomy groupoid encodes the information of the leaf space (which  is typically not smooth) and the holonomy groups, so this feature gives a geometric motivation for its study. On the other hand, in non commutative geometry, it gives a geometric interpretation for certain $C^*$-algebras. Every Lie groupoid has an associated $C^*$-algebra \cite{GrpdCAlg} therefore every regular foliation has an associated $C^*$-algebra given by its holonomy groupoid. This construction was given by A. Connes in \cite{Connes3}.
	
	A well known fact is that any commutative $C^*$-algebra is isomorphic to the space of functions on a topological space. This gives a geometric interpretation of an, in principle, algebraic object. The $C^*$-algebra defined by A. Connes to a regular foliation is usually non commutative, giving a geometric interpretation for a more general kind of $C^*$-algebras. Moreover, using this construction, in the articles \cite{Connes1}, \cite{Connes2} and \cite{Connes3} the authors developed a longitudinal psedo-differential calculus and an index theory for regular foliations. 
	
	One can also generalize these constructions for singular foliations. Here the term ``singular'' comes in when we allow smooth partitions whose leaves vary in dimension. Similarly to the Frobenius theorem, the Stefan-Sussmann theorem states that any smooth partition of $M$ gives a unique ``singular integrable distribution'' on $M$ and vice-versa. We will not discuss this characterization. The reader can check \cite[\S 2.3]{GrudMSch} for more details.
	
	In this thesis we define singular foliations as Androulidakis and Skandalis in \cite{AndrSk}. By the Stefan-Sussmann theorem, one can check that this definition gives indeed a smooth partition of $M$. However, this process is not reversible. In chapter \ref{ch:introduction} we will give a proper introduction to this object.
	
	\textbf{Definition.}
	\textit{A \textbf{singular foliation} on a manifold $M$ is a $\CI(M)$-submodule $\CF$
		{of the compactly supported vector fields}
		$\CX_c(M)$, closed under the Lie bracket and locally finitely generated.
		A \textbf{foliated manifold} is a manifold with a singular foliation.}
	
	Note that, the above definition generalizes canonically the notion of integrable distribution. Indeed, for any integrable distribution $D\subset TM$ there exists a singular foliation $\CF:=\SEC_c(D)\subset \CX_c(M)$, and for any singular foliation $\CF\subset \CX(M)$, satisfying some regularity condition, one can get an integrable distribution $D\subset TM$.
	
	Although foliated manifolds as in the above definition are recent objects, they have been studied by many people, as the reader can see by the following articles: \cite{AndrSk}, \cite{AZ1}, \cite{AZ2}, \cite{AMsheaf}, \cite{HenT}, \cite{DebordJDG}, \cite{Debord2013},  \cite{SylvainArticle}, \cite{Kwang}. 
	
	In \cite{AndrSk} the authors generalized the holonomy groupoid to {any foliated manifold}. We will review this construction and guide the reader step by step in section \ref{sec:holconstr}. Even though in many cases this groupoid fails to be smooth (i.e, Lie), it has a partial smooth structure \cite{Debord2013}. By using this partial smooth structure, Androulidakis and Skandalis defined a  $C^*$-algebra of a foliated manifold. This was an important step towards the generalization of the work of A. Connes to the singular case.
	
	\subsubsection*{Poisson geometry, $C^*$-algebras and Morita Equivalence}
	
	Modern Poisson geometry, which takes its roots in the classical work of Poisson in the 19th century on classical mechanics, was established following   the pioneering work of Weinstein~\cite{WeinPois83} in 1983, and it has become an active research field with connections to a variety of fields, such as algebra, differential geometry and mathematical physics.
	
	Any Poisson manifold $(M,\pi)$ has an { associated Lie algebroid} \cite{WAPois}, namely the cotangent bundle of $M$ with anchor $\pi^\sharp\colon T^*M\fto TM; \alpha\mapsto \pi(\alpha,-)$. This Lie algebroid { induces a foliated manifold $(M,\pi^\sharp(\Omega^1_c(M)))$} where each leaf has a symplectic structure. Moreover, this Lie algebroid can sometimes be integrated into a source simply connected { symplectic Lie groupoid} (\cite{CFPois}) and the orbits of this Lie groupoid coincide with the symplectic leaves of the singular foliation.
	
	Poisson geometry is also related to the study of {quantization} in physics (\cite{Dirac1}, \cite{kostant}, \cite{DefQ1} and \cite{DefQ2}). The goal of quantization is to relate classical mechanics, described in mathematical terms by {Poisson geometry} (\cite{AnaSympl} and \cite{WAPois}), and quantum mechanics, described by {non-commutative $C^*$-algebras}.
	
	A notion which lies at the intersection of the theories of $C^*$-algebras, Poisson manifolds and Lie groupoids is the notion of {\bf Morita equivalence}. This notion first appeared in algebra as a weak version of isomorphisms, in the sense that two $C^*$-algebras are Morita equivalent if their categories of modules are equivalent. This is an important notion because for some cases one does not care for the algebraic object, but for the possible actions (or representations) of it.
	
	Morita equivalence was later extended to Lie groupoids and Poisson manifolds. In \cite{WAPois} it is explained in more detail how the concepts of Morita equivalence for these three different objects are related: for example, if two Lie groupoids are Morita equivalent, then their corresponding $C^*$-algebras are also Morita equivalent. In this thesis we do not address $C^*$-algebras. We will focus on groupoids and then many results for their $C^*$-algebras follow immediately.
	
	The notion of Morita equivalence for Lie groupoids has a geometric interpretation. Namely, if two Lie groupoids $\CG$ and $\CH$ are Morita equivalent, their orbitspaces (i.e. their spaces of orbits) are homeomorphic to the same topological space $S$ and their isotropy Lie groups are isomorphic. Hence one can regard $\CG$ and $\CH$ as equivalent smooth structures (or atlases) for $S$. This has consequences for foliated manifolds, namely if $\CG\soutar M$ is a source connected Lie groupoid with associated foliated manifold $(M,\CF)$, then the orbitspace of $\CG$ coincides with the leafspace of $(M,\CF)$, and therefore $\CG$ can be seen as a ``smooth structure'' for the leafspace of $(M,\CF)$. 
	
	One of the goals of this thesis is to generalize the notion of Morita equivalence from Poisson manifolds, Lie groupoids and $C^*$-algebras to the setting of foliated manifolds, and to relate it with the notion of transverse equivalence given by Molino in \cite{MolME}.\\

	\section*{The results}
	
	This thesis has two main results. The first one is about Morita equivalence of singular foliations. The second one is about quotients of foliated manifolds and of their associated holonomy groupoids. 
	
	\subsubsection*{Hausdorff Morita equivalence of singular foliations}
	
	The transverse geometry of a regular foliation was studied in the 1980's and 1990's. Haefliger  stated that a property of a regular foliation is {transverse} if it can be described in terms of the Morita equivalence class of its holonomy groupoid (see the first paragraph  of \cite[\S 1.5]{HaefligerHolonomieClassifiants}).
	
	Molino introduced various notions of {transverse equivalence} of regular foliations (see \cite[\S2.7]{MolME} and \cite[\S 2.2 d]{MolinoOrbit-LikeFol}). His notion of transverse equivalence requires that the pullbacks of the foliations to suitable spaces agree, and does not make any reference to the holonomy groupoid.
	
	In the same spirit as Molino, M. Zambon and I introduced in \cite{ME2018} a definition of  Morita equivalence  of singular foliations. In section \ref{sec:ME1} we give an introduction to this notion.
	
	One of the results in this thesis is that Morita equivalence  of singular foliations have many invariants, which should be regarded as constituents of the ``transverse geometry'' of a singular foliation:
	
	\textbf{Theorem.} 
	\textit{If two singular foliations are  Hausdorff Morita equivalent, then:
		\begin{itemize}
			\item[a)] the leaf spaces are homeomorphic,
			\item[b)] the isotropy Lie groups (and isotropy Lie algebras) of corresponding leaves are isomorphic,
			\item[c)] the representations of corresponding isotropy Lie groups on   
			normal spaces to the leaves are isomorphic.
	\end{itemize}}

	Notice that this is analogous to the Morita equivalence of Lie groupoids, for which the space of orbits, the isotropy Lie groups and their normal representations are a complete set of invariants  \cite[theorem.  4.3.1]{MatiasME}.
	
	Several geometric objects have singular foliations associated to them. In section \ref{sec:me.hol.grpd} we will prove the following statements:
	
	\textbf{Theorem.} \textit{If two source connected Hausdorff Lie groupoids,
		two Lie algebroids \cite[\S 6.2]{GinzburgGrot} or two Poisson manifolds \cite{xuME} are Morita equivalent, then their underlying singular foliations are  Hausdorff Morita equivalent. }
	
	\textbf{Theorem.} \textit{If two singular foliations are Hausdorff Morita equivalent then their holonomy groupoids   are Morita equivalent (as open topological groupoids).}
	
	\noindent For regular foliations -- and more generally for projective foliations -- the holonomy groupoid is a source connected Lie groupoid. Therefore, in this case, Hausdorff Morita equivalence of singular foliations coincides with Morita equivalence of their associated holonomy groupoids (under Hausdorffness assumptions).
	
	From the above it is clear that the notions of Morita equivalence for singular foliations and Morita equivalence for the associated holonomy groupoids are closely connected. We emphasize that our definition of Hausdorff Morita equivalence is expressed in terms of the singular foliation alone, without making any reference to the associated holonomy groupoid, and as such, it has the advantage of being easy to handle.
	
	\subsubsection*{Quotients of singular foliations}
	
	Let $(P,\CF)$ be a foliated manifold and $\pi \colon P \to M$ a {surjective} submersion with connected fibers (in particular, $M$ is the quotient of $P$ by the regular foliation given by the fibers of $\pi$).
	Under mild conditions, $\cF$ can be ``pushed forward'' along $\pi$ to a singular foliation $\cF_M$ on $M$. 
	
	As recalled earlier, every singular foliation has a canonically associated
	topological groupoid, called holonomy groupoid \cite{AndrSk}. Since the singular foliation $\cF_M$ is obtained from $\cF$ by a quotient procedure, it is natural to wonder whether the holonomy groupoid $\CH(\cF_M)$ of $\cF_M$ is also a quotient of the holonomy groupoid $\CH(\cF)$. In Chapter \ref{ch:conclusion}, Theorem \ref{thm:sur.hol} we show that this is always the case: there is a canonical surjective morphism 
	
	$$\Xi\colon \CH(\cF)\to \CH(\cF_M).$$

	We then refine the above result in the case of Lie group actions. That is, we assume that a Lie group $G$ acts freely and properly on $P$ preserving the singular foliation $\cF$, and $\pi$ is the projection to the orbit space $M=P/G$. The action lifts naturally to a $G$-action on $\CH(\cF)$, but a simple dimension count shows that the quotient can not be isomorphic to $\CH(\cF_M)$ in general. In \S\ref{sec:quotgrouppull} we show that 
	when  $\cF$ contains the infinitesimal generators of the $G$-action,
	there is a natural action of a semidirect product Lie group $G \rtimes G$ on $\CH(\cF)$ -- not by Lie groupoid automorphisms -- with quotient $\CH(\cF_M)$.
	Remarkably, this is a group action in the category of Lie groupoids (see Theorem
	\ref{thm:act.q.fol}).
	
	{In the general case when $\cF$ does not necessarily contain the infinitesimal generators of the $G$-action, it is no longer true that the $\Xi$-fibers are given by orbits of a Lie 2-group action. In \S\ref{sec:groidorbits} we describe the fibers of $\Xi$ as orbits of a certain \emph{groupoid} action on $\CH(\cF)$ (see Prop. \ref{prop:fib.xi}). Under additional assumptions, this can be promoted to a Lie group action.
		In the end of \S\ref{sec:general} we show that there is a canonical Lie ideal $\h$ of $\g$ 
		which gives rise to a Lie 2-group $H \rtimes G$ and a Lie 2-group action on $\CH(\cF)$
		whose orbits are contained  in the $\Xi$-fibers (see Prop. \ref{prop:Lie.2.grpACTION}). The concrete form of this Lie 2-group action is inspired by a special case, in which $\cF$ contains the infinitesimal generators of the $G$-action (hence $H=G$). Indeed in that case we recover  the Lie 2-group action given in \S\ref{sec:quotgrouppull}, 
		as we explain in Prop. \ref{prop:sameaction}.
	}

	\section*{Funding}
	
	This thesis was partially supported by IAP Dygest and by the FWO under EOS project G0H4518N.

	\cleardoublepage

	\chapter{Singular foliations}\label{ch:sing foliation}
	
	The objective of this chapter is to introduce singular foliation in the sense of the article \cite{AndrSk}. In the first section we introduce this definition, its properties and give some examples.
	
	In the second and third sections we introduce the definitions of transversal maps to a foliation, the pullback of a foliation under a transversal map and under a submersion, and the local description of any singular foliation. In the fourth section we construct the product of foliations and show some of their properties with the pullbacks.
	
	In the last section of this chapter we give a different but equivalent characterization for foliations, using sheaves, and motivate its advantages. 
	
	\section{Definition of singular foliations}\label{sec:def.singfol}
	
	This subsection is an introduction to the definition of singular foliation given by I. Androulidakis and G. Skandalis in \cite{AndrSk} therefore most of the result can be found there. We will make clear when there are new results.

	\begin{defi} A $\CI_c(M)$-submodule $\mathcal{F}\subset \mathfrak{X}_c(M)$ is \textbf{finitely generated} if there exists a finite set of vector fields $Y^1,...,Y^r\in\mathfrak{X}(M)$, called \textit{generators}, such that
		\begin{equation*}
			\mathcal{F} = \langle Y^1,...,Y^r \rangle_{C^\infty_c(M)},
		\end{equation*}
		i.e. $\mathcal{F}$ is the $C^\infty_c(M)$-linear span of the generators.
	\end{defi}
	
	Note that the generators of $\mathcal{F}$ are not required to be in $\mathcal{F}$ and their support is not necessarily compact.
	
	\begin{defi}\label{def:loc.gen} A submodule $\mathcal{F}\subset \mathfrak{X}_c(M)$ is \textbf{locally finitely generated} if every point $x\in M$ has a neighborhood $U\subset M$ such that $$\iota_U^{-1}\CF:=\{ X|_U \st X\in \CF \y \text{supp}(X)\subset U\},$$ is finitely generated as a $\CI_c(U)$-module. 
	\end{defi}
	
	\begin{rem}
		The notation ``$\iota_U^{-1}\CF$'' comes from the pullback under the inclusion map $\iota_U: U\fto M$ given in definition \ref{def:pullback}.
	\end{rem}
	
	The following lemma will be useful later. In particular it states that finitely generated modules can be restricted to smaller open sets and glued finitely many times.
	
	\begin{lem}
		\label{lemma:unionoffinite} 
		Let $\mathcal{F}\subset \mathfrak{X}_c(M)$ be a submodule.
		\begin{enumerate}
			\item[(a)] If $\mathcal{F}$ is finitely generated, then $\iota_U^{-1}\CF$ is finitely generated for all open subsets $U\subset M$.
			
			\item[(b)] If $U,V\subset M$ are open subsets such that $\iota_U^{-1}\CF$ and $\iota_V^{-1}\CF$ are finitely generated, then $\iota_{U\cup V}^{-1}\CF$ is finitely generated. 
			
			\item[(c)]  $\mathcal{F}$ is locally finitely generated if and only if, for all $\rho\in C^\infty_c(M)$, the submodule 
			\begin{equation*}
				\rho\mathcal{F}:=\{ \rho X \;|\; X\in\mathcal{F} \}\subset \mathfrak{X}_c(M)
			\end{equation*}
			is finitely generated. 
		\end{enumerate}
	\end{lem}
	
	\begin{proof}
		(a) Using that $\iota_U^{-1}\CF=\left<\mathcal{F}\right>_{C^\infty_c(U)}$, one easily sees that that the restrictions of generators of $\mathcal{F}$ to $U$ are generators of $\iota_U^{-1}\CF$. 
		
		(b) Let $X^1,...,X^r\in \mathfrak{X}(U)$ be generators of $\iota_U^{-1}\CF$ and $Y^1,...,Y^s\in \mathfrak{X}(V)$ of $\iota_V^{-1}\CF$, and choose a partition of unity $\{\rho_U,\rho_V\}$ subordinate to the open cover $\{U,V\}$ of $U\cup V$. We show that $\rho_U X^1,...,\rho_U X^r,\rho_V Y^1,...,\rho_V Y^s\in\mathfrak{X}(U\cup V)$ (or rather their extension by zero to $U\cup V$) are generators of $\iota_{U\cup V}^{-1}\CF$, i.e. that $\iota_{U\cup V}^{-1}\CF$ is equal to the $C^\infty_c(U\cup V)$-linear span of the generators. 
		
		On the one hand, if $f\in C^\infty_c(U\cup V)$ then $f\rho_U\in C^\infty_c(U)$, since $\mathrm{supp}(f\rho_U) = \mathrm{supp}(f)\cap\mathrm{supp}(\rho_U)$ and a closed subset of a compact set is compact. Similarly, $f\rho_V\in C^\infty_c(V)$. Thus, using that $\iota_{U}^{-1}\CF \subset \iota_{U\cup V}^{-1}\CF$ and $\iota_{V}^{-1}\CF \subset \iota_{U\cup V}^{-1}\CF$, we see that any $C^\infty_c(U\cup V)$-linear combination of the generators is an element of $\iota_{U\cup V}^{-1}\CF$. 
		
		On the other hand, let $Z\in \iota_{U\cup V}^{-1}\CF$ and write $Z = \rho_U Z + \rho_V Z$. For the reason explained above, $\rho_U Z \in \mathfrak{X}_c(U)$ and $\rho_V Z \in \mathfrak{X}_c(V)$. Choose functions $\lambda_U\in C^\infty_c(U)$ and $\lambda_V\in C^\infty_c(V)$ such that $\lambda_U|_{\mathrm{supp}(\rho_U Z)}=1$ and $\lambda_V|_{\mathrm{supp}(\rho_V Z)}=1$, we have that $Z = \rho_U \lambda_U Z + \rho_V \lambda_V Z$. Finally, $\lambda_U Z = \sum_{i=1}^r f_i X^i$ for some $f^i\in C^\infty_c(U)$ and $\lambda_V Z = \sum_{i=1}^s g_i \textcolor{blue}{Y}^i$ for some $g^i\in C^\infty_c(V)$, and so
		\begin{equation*}
			Z = \sum_{i=1}^r f_i \, \rho_U X^i + \sum_{i=1}^s g_i \, \rho_U Y^i. 
		\end{equation*}
		
		(c) We begin with the forward implication. Let $\rho\in C^\infty_c(M)$. There exists an open subset $U\subset M$ containing $\mathrm{supp}(\rho)$ such that $\iota_{U}^{-1}\CF$ is finitely generated, namely we may choose a finite covering of $\mathrm{supp}(\rho)$ by open subsets $U_1,...,U_r$ such that all restrictions $\iota_{U_i}^{-1}\CF$ are finitely generated and then apply part (b) a finite number of times to conclude that $\iota_{U_1\cup...\cup U_r}^{-1}\CF$ is finitely generated. It is now easy to see that $\rho \mathcal{F}= \rho \mathcal{F}_U$ and hence $\rho\mathcal{F}$ is finitely generated.
		
		For the reverse implication, given $x\in M$, we may construct an open neighborhood $U$ of $x$ and a $\rho\in C^\infty_c(M)$ such that $\rho|_U =1$, and we immediately see that $\iota_{U}^{-1}\CF=\rho (\iota_{U}^{-1}\CF)=\iota_{U}^{-1}(\rho\CF)$, which is finitely generated by part (a).
	\end{proof}
	
	\begin{definition}\label{defi:sing.fol}
		A \textbf{singular foliation} on a manifold $M$ (in the sense of Androulidakis and Skandalis \cite{AndrSk}) is a locally finitely generated submodule $\mathcal{F}\subset \CX_c(M)$ that is involutive, i.e. $[\mathcal{F},\mathcal{F}]\subset \mathcal{F}$. A manifold with a singular foliation is called a {\bf foliated manifold}.
	\end{definition}
	
	\begin{ex}\label{ex:reg.fol} By the Frobenius theorem, any regular foliation has an associated integrable distribution $D\subset TM$. Then $\CF:=\SEC_c(D)$ is a singular foliation.	
	\end{ex}
	
	Given a set of vector fields $X_1,\dots X_n\in \CX(M)$, its span $\CF=\left< X_1,\dots, X_n\right>_{\CI_c(M)}$ is a singular foliation if and only if it is involutive. A particular example where this happens is the following:
	
	\begin{ex} Let $\g$ be a Lie algebra acting infinitesimally on $M$. This action consists of a vector bundle map $\rho\colon A_\g\fto TM$ which preserves the Lie bracket, where $A_\g:=\g\times M$. Given a basis $v_1,\dots,v_k$ for $\g$ the vectorfields $X_i (p):=\rho(v_i,p)\in \CX(M)$ span a singular foliation $\CF_{\g}:=\left<X_1,\dots,X_k\right>_{\CI_c(M)}$.
	\end{ex}
	
	\begin{ex}\label{ex:alg.fol} In section \S \ref{sec:Lie.Algd}, particularly in \ref{prop:Folie.Algd}, we will see that any Lie algebroid canonically defines a singular foliation. Moreover in example \ref{ex:Algd,Gprd} we will see that any Lie groupoid has an associated Lie algebroid and therefore a singular foliation.
	\end{ex}
	
	One of the motivations for definition \ref{defi:sing.fol} is that it gives an integrable singular distribution and, by the Stefan-Sussmann theorem \cite[\S 2.3]{GrudMSch}, a smooth partition of $M$. Another motivation for definition \ref{defi:sing.fol} is that it also allows us to define the holonomy groupoid for a singular foliation, as we will show in chapter \ref{ch:conclusion}. A drawback is that not any singular integrable distribution is a singular foliation, nevertheless singular foliations are general enough to contain many key examples, for instance: \ref{ex:reg.fol} and \ref{ex:alg.fol}.
	
	For a singular foliation there are certain interesting spaces given point wise. These spaces measure the regularity of the foliation around a point.
	
	\begin{defi} Let $(M,\CF)$ be a foliated manifold and $x\in M$. Denote $I_x:=\{f\in \CI(M) \st f(x)=0\}$. We denote:
		
		The {\bf tangent of $\CF$} at $x$ by: $F_x:=\{X(x) \st X\in \CF\}\subset T_x M$.
		
		The {\bf fibre of $\CF$} at $x$ by: $\CF_x:= \CF/I_x\CF$.
	\end{defi}
	
	Note that the evaluation map $\mathrm{ev}_x\colon \CF_x\fto F_x$ at the point $x$ induces a short exact sequence of vector spaces:
	$$0\to \ker(\mathrm{ev}_x)\to \CF_x \to F_x\to 0.$$
	Moreover, the Lie bracket on $M$ descends canonically to a Lie bracket on $\ker(\mathrm{ev}_x)$, giving a Lie algebra structure on it. 
	
	\begin{defi} The space $\g_x^{\CF}:=\ker(ev_x)$ is called the {\bf isotropy Lie algebra} of $\CF$ at $x$.
	\end{defi}
	
	The tangent space $F_x$ gives a way to determine if $\CF$ is a regular foliation, as it is shown in the following corollary, which is a direct consequence of the Frobenius theorem:
	
	\begin{cor}\label{cor:tan.reg.fol} $\CF$ is a regular foliation if and only if the space $F_x$ has constant dimension  for all $x\in M$. This also implies $F_x\simeq\CF_x$ and $\g_x^{\CF}=0$.
	\end{cor}
	
	\begin{definition} For the rest of this thesis a \textbf{regular foliation} on a manifold $M$ is a singular foliation $\CF$ on $M$ such that the spaces $F_x$ have constant dimension for all $x\in M$.
	\end{definition}
	
	The fiber $\CF_x$ is related to a minimal amount of generators for $\CF$ near $x$, as the following lemma  (proved in \cite{AndrSk}) shows:
	
	\begin{lem} Let $\{X_1,\dots X_n\}\in \CF$ be a finite subset and fix $x_0\in M$. If the class of elements in $\{X_1,\dots X_n\}$ are a basis for $\CF_{x_0}$ then there exists a neighborhood $U\subset M$ of $x_0$ such that $\{X_1,\dots X_n\}$ generates $\iota_U^{-1}\CF$.
	\end{lem}
	\begin{proof} Because $\CF$ is locally finitely generated, there exists a neighborhood $V$ of $x_0$ and vector fields $Y_1,\dots, Y_N$ spanning $\iota_V^{-1}\CF$. Because the classes of $X_1,\dots,X_n$ in $\CF_{x_0}$ are a basis there exists a unique matrix $A\in M_{n\times N}(\RR)$ such that:
		$$\left[\begin{matrix}
			(Y_1)_{x_0} \\
			\vdots\\
			(Y_N)_{x_0}
		\end{matrix}\right] =
		A\cdot \left[\begin{matrix}
			(X_1)_{x_0} \\
			\vdots\\
			(X_n)_{x_0} 
		\end{matrix}\right],$$
		where $(-)_{x_0}$ denotes the class in $\CF_{x_0}$. Then using that $Y_1,\dots, Y_N$ are generators near $x_0$, there exists $B\in C(M_{N\times N}(\RR))$ such that $B(x_0)=0$ and
		$$\left[\begin{matrix}
			Y_1 \\
			\vdots\\
			Y_N
		\end{matrix}\right] -
		A\cdot \left[\begin{matrix}
			X_1 \\
			\vdots\\
			X_n
		\end{matrix}\right]=
		B \left[\begin{matrix}
			Y_1 \\
			\vdots\\
			Y_N
		\end{matrix}\right].$$
		Therefore,
		$$
		A\cdot \left[\begin{matrix}
			X_1 \\
			\vdots\\
			X_n
		\end{matrix}\right]=
		(Id-B) \left[\begin{matrix}
			Y_1 \\
			\vdots\\
			Y_N
		\end{matrix}\right].$$
		Using the smoothness of $B$ and that $B(x_0)=0$ there must exists a neighborhood $U\subset V$ of $x_0$ such that $Id-B$ is invertible. Finally on $U$:
		$$
		\left[\begin{matrix}
			Y_1 \\
			\vdots\\
			Y_N
		\end{matrix}\right]=
		(Id-B)^{-1} A\cdot \left[\begin{matrix}
			X_1 \\
			\vdots\\
			X_n
		\end{matrix}\right].$$
		which proves that $X_1,\dots,X_n$ are generators of $\iota_U^{-1}\CF$.	
	\end{proof}
	
	This fiber $\CF_x$ also determines whether or not $\CF$ is a projective module. Recall that, an $\CI_c(M)$-module $\CF$ is \textbf{projective} if there is another $\CI_c(M)$-module $Q$ such that $\CF\oplus Q$ is a free $\CI_c(M)$-module. The Serre-Swan theorem says that any $\CI_c(M)$-module is projective if and only if there exists a vector bundle $A\to M$ such that $\CF \cong \Gamma_c(A)$ as $\CI_c(M)$-modules (the free $\CI_c(M)$-modules are given by trivial vector bundles over $M$).
	
	This fact says that a projective foliation $\CF$ behaves quite well (see \cite{DebordJDG}, where they are called almost regular foliations). If you see the definition of Lie algebroids \ref{def:Lie.algd}, the vector bundle $A$ associated to $\CF$ acquires canonically a Lie algebroid structure for which the anchor map is injective on an open dense set.
	
	\begin{lem}\label{lem:fiber.proy} $\CF_x$ has constant dimension for all $x\in M$ if and only if $\CF$ is a projective module.
	\end{lem}
	\begin{proof} (Sketch)
		If $\CF$ is projective, there exists a vector bundle $A\fto M$ whose module of sections are isomorphic to $\CF$. Then $\dim(\CF_x)=\dim(A_x)$ which is constant.
		
		For the converse suppose $\dim(\CF_x)$ is constant. Then one can create a vector bundle $A$ with the condition $A_x\cong\CF_x$. Finally $\Gamma_c(A)\simeq \CF$.
	\end{proof}
	
	An important characteristic of singular foliations is that they give a singular distribution on and a smooth partition of their underlying manifolds. Nevertheless, the singular foliation contains more information than the smooth partition, for example:
	
	In $\RR^2$ we can choose
	$$\CF^0:=\left< x\de_x +y\de_y, y\de_x -x\de_y \right>_{\CI_c(M)},$$
	$$\CF^1:=\left<x\de_x, y\de_y, y\de_x, x\de_y \right>_{\CI_c(M)}.$$
	
	Both foliations give the smooth partition $L_0=\{(0,0)\}$ and $L_1=(\RR^2-\{(0,0)\})$ but they do not define the same submodule. In fact, looking at the dimensions of the fibers at the origin we see that  $dim(\CF^0_{(0,0)})=2$ and $dim(\CF^1_{(0,0)})=4$, which implies that $\CF^0$ is a projective module whereas $\CF^1$ is not.
	
	\section{Transversal maps}\label{sec:trans.maps}
	Let $P,M$ be manifold and $\pi:P\fto M$ a smooth map. Consider the following maps of sections:
	\begin{align*}
		d\pi\colon  \CX(P) &\fto \Gamma(P,\pi^*TM), Y\mapsto d\pi Y,\\
		\pi^*\colon\CX(M)&\fto \SEC(P,\pi^*TM), X\mapsto X\circ \pi.
	\end{align*}
	
	A definition of transversality given in \cite{AndrSk} is the following:
	
	\begin{defi}
		A map $\pi\Fr P\fto M$ is transversal to a foliation $\CF$ on $M$ if the canonical map:
		\begin{equation}\label{defi:canmap}\CX_c(P)\oplus \left< \pi^*\CF \right>_{\CI_c(P)} \fto \SEC_c(\pi^*TM);Y\oplus \xi\mapsto (d\pi Y)+\xi
		\end{equation}
		is onto.
	\end{defi}
	
	This condition is a point wise property. This means that it is equivalent to the following statement, which was suggested by Henrique Bursztyn in his article \cite{HenT}:
	
	\begin{prop}
		A map $\pi\Fr P\fto M$ is transversal to a foliation $\CF$ on $M$ if and only if for all $p\in P$ the following condition is satisfied:
		\begin{equation}\label{eq:trans}
			\img(d_p\pi)+F_p=T_{\pi(p)}M,\end{equation}
		where $F_p:=\{X(p)\st X\in \CF\}$.
	\end{prop}
	\begin{proof}
		{\bf Transversal $\Rightarrow$ condition (\ref{eq:trans}):} Take arbitrary $p\in P$ and $z\in T_{\pi(p)} M$. There is a section $Z\in \SEC_c(\pi^*TM)$ such that $Z(p)=z$. Using tranversality in equation (\ref{defi:canmap}), there exist a section $Y\in \CX_c(P)$, elements $X_i\in \CF$ and $g_i\in \CI_c(P)$ such that 
		\[d\pi(Y)+\sum_i g_i\pi^*X_i=Z.\]
		Applying this formula to the point $p$ we get:
		$$d_p\pi Y+\sum_i g_i(p)X_i(\pi(p))=Z(p)=z.$$
		
		Denote $X:=\Sigma_i g_i(p)X_i\in \CF$, then $d_p\pi Y+X(\pi(p))=Z(p)=z$. Because $p$ and $z$ are arbitrary we get condition (\ref{eq:trans}).
		
		{\bf Condition (\ref{eq:trans}) $\Rightarrow$ transversal:} Take an arbitrary $p\in P$, there exists a local frame $\{Y_i\}$ for $TP$ around $p$ and some local generators $\{X_j\}$'s for $\CF$ nearby $\pi(p)$. Using the condition \ref{eq:trans} we get that the $\{d_p\pi Y_i\}$ and the $\{X_j(\pi(p))\}$ span $T_{\pi(p)}M$. It is possible to choose among the $d_p\pi Y_i$'s and the $X_j(\pi(p))$'s a basis for $T_{\pi(p)}M$. Because of the regularity of $\pi^*TM$ as a vector bundle over $P$ the chosen $d\pi Y_i$'s and $\pi^*X_j$'s form a local frame of $\pi^*TM$ nearby $p$. Therefore for each $p$ there exists a neighborhood $U_p$, some vectorfields $Y_i^p\in \CX(P)$ and $X_j^p\in \CF$ such that the $d\pi Y_i^p$ and the $\pi^*X_j^p$ form a basis for $\pi^*TM|_{U_p}$.
		
		Take $Z\in \SEC_c(\pi^*TM)$. For each $p\in \text{supp}(Z)$ we find a $U_p$, $d\pi Y_i^p$ and $\pi^*X_j^p$ as in the last paragraph. There exist unique smooth functions $g^p_i,g^p_j$ on $U_p$ such that:
		\begin{equation}\label{equ:trans.z}
			Z|_{U_p}=\sum_i g_i^p (d\pi  Y^p_i)+\sum_j g_j^p (\pi^*X^p_j).
		\end{equation}
		
		Because $\text{supp}(Z)$ is compact, there exist finitely many $U_p$ that cover $\text{supp}(Z)$, take $U_k$'s as this finite cover for $\text{supp}(Z)$ and $g_i^k, g_j^k, d\pi Y^k_i,  \pi^* X^k_j$ the correspondent functions and vector fields satisfying equation (\ref{equ:trans.z}) for each $U_k$. Take a partition of unity $f_k$'s subordinate to the $U_k$'s. Therefore, $Z$ can be written as the following finite sum of elements:
		
		\begin{align*}
			Z &=  \sum_{i,k} f_k g_i^k (d\pi Y^k_i )+\sum_{j,k} f_k g_j^k (\pi^* X^k_j),\\
			&= d\pi\left(\sum_{i,k} f_k g_i^k  Y^k_i \right)+\sum_{j,k} f_k g_j^k (\pi^* X^k_j),
		\end{align*}
		where $(\Sigma  f_k g_i^k  Y^k_i)\in \CX_c(P)$ and $(\Sigma f_k g_j^k (\pi^* X^k_j))\in \left<\pi^*\CF\right>_{\CI_c(P)}$. Then $Z$ is in the image of the canonical map (\ref{defi:canmap}). Finally, because $Z$ is arbitrary, we have transversality.
	\end{proof}
	
	\section{Pullbacks and push forwards}\label{sec:pull.push}
	
	In this section we review the pullback foliation for transversal maps, as it was first appeared in \cite{AndrSk}. We will also give an easier description of this pullback foliation when the map is a submersion and explain a local picture for singular foliations.
	
	We start with the definition of pullback foliation.
	
	\begin{defi}\label{def:pullback}
		Let $P,M$ be manifolds, $\CF$ a foliation on $M$  and  $\pi:P\fto M$ a smooth map. Recall the maps of sections:
		\begin{align*}
			d\pi\colon  \CX(P) &\fto \SEC(P,\pi^*TM), Y\mapsto d\pi Y,\\
			\pi^*\colon\CX(M)&\fto \SEC(P,\pi^*TM), X\mapsto X\circ \pi.
		\end{align*}
		The {\bf pullback foliation} of $\CF$ under $\pi$ \cite[proposition 1.10]{AndrSk} is the submodule of $\CX(P)$ given by:
		\[\pi^{-1}(\CF):= \left( d\pi^{-1}\left(\left<\pi^*\CF\right>_{\CI_c(P)}\right)\right)_c.\]
	\end{defi}
	Let us explain the construction of the pullback foliation step by step: take the image $\pi^*(\CF)$, then make it a $\CI_c(P)$-module $\left<\pi^*\CF\right>_{\CI_c(P)}$, take the preimage by the module map $d\pi$ and finally the compactly supported elements of it.
	
	\begin{prop}Let $P,M$ be manifolds, $\CF$ a foliation on $M$  and $\pi\Fr P\fto M$ be a smooth map.
		\begin{enumerate}
			\item $\pi^{-1}(\CF)$ is closed under the Lie bracket.
			\item If $\pi$ is transverse to $\CF$ then $\pi^{-1}(\CF)$ is locally finitely generated.
			\item Let $S$ be a manifold and $\varphi\Fr S\fto P$ be a smooth map. The map $\varphi$ is transverse to $\pi^{-1}\CF$ if and only if $\pi\circ \varphi$ is transverse to to $\CF$ and we have $(\pi\circ \varphi)^{-1}\CF=\varphi^{-1}(\pi^{-1}\CF)$.
		\end{enumerate}
	\end{prop}
	\begin{proof} Prop 1.10 and 1.11 in \cite{AndrSk}
	\end{proof}
	
	The following lemma contains a simplified description for the pullback foliation when the transversal map is replaced by a submersion:
	
	\begin{lem}\label{lem:pullback}
		Let $(M,\CF)$ be a foliated manifold and $\pi\colon P\rightarrow M$ a submersion, then
		\[\pi^{-1}(\CF)= \left<(d\pi)^{-1} \pi^* (\CF)\right>_{\CI_c(P)}.\]
	\end{lem}
	
	Note that the set $(d\pi)^{-1} \pi^* (\CF)$ consists of the projectable vector fields in $\CX(P)$ going to $\CF$. Therefore $\pi^{-1}(\CF)$ is generated by projectable vector fields going to $\CF$
	
	\begin{proof} Call $H:=\pi^*\CF$. We will prove that:
		$$\left( d\pi^{-1}\left(\left< H\right>_{\CI_c(P)}\right)\right)_c=\left<(d\pi)^{-1} H\right>_{\CI_c(P)}.$$
		To prove this fact we will use double inclusion and the following facts:
		\begin{itemize}
			\item  $\pi$ is a submersion and therefore $H\subset \img(d\pi)$.
			\item It is clear that $0\in H$, which implies that $\ker(d\pi)\subset (d\pi)^{-1}H$ and $\left(\ker(d\pi)\right)_c\subset \left<(d\pi)^{-1}H\right>_{\CI_c(P)}$.
		\end{itemize}
		Now we will start the proof:\\

		($\supseteq$):  Note that every element of $\left< \left((d\pi)^{-1}H\right)\right>_{\CI_c(P)}$ has compact support, therefore we only need to prove: \[(d\pi)^{-1}\left(\left< H\right>_{\CI_c(P)}\right)\supseteq  \left<(d\pi)^{-1}H\right>_{\CI_c(P)}.\]
		
		To do so, note that $$d\pi\left(\left< d\pi^{-1}(H)\right>_{\CI_c(P)}\right)=\left<d\pi( d\pi^{-1}(H))\right>_{\CI_c(P)}=\left< H\right>_{\CI_c(P)},$$
		using in the last equality that $H\subset \img(d\pi)$ implies $d\pi (d\pi^{-1}(H))= H$.
		
		Finally:
		\[ \left<(d\pi)^{-1}H\right>_{\CI_c(P)}\subseteq (d\pi)^{-1}\left(\left< H\right>_{\CI_c(P)}\right).\]
		
		($\subseteq$): Take $X\in \left((d\pi)^{-1}\left( \left<H\right>_{\CI_c(P)}\right)\right)_c$, we have that $X\in \CX_c(P)$ and:\\
		\[d\pi(X)=\sum_i f_i Y_i, \hspace{.8in} \text{for} \hspace{.1in} f_i\in \CI_c(P) \hspace{.1in}\text{and} \hspace{.1in} Y_i\in H.\]
		
		Using again that $H\in Im(d\pi)$, there are $X_i\in (d\pi)^{-1}H$ such that $d\pi(X_i)=Y_i$. These elements are not necessarily compactly supported, but the elements $f_i X_i\in \CX_c(P)$ have compact support. Then $X-\Sigma f_i X_i\in \CX_c(P)$ has compact support and therefore:
		\[\left(X-\Sigma f_i X_i\right)\in Ker(d\pi)_c\subset \left<(d\pi)^{-1}H\right>_{\CI_c(P)}.\]
		Finally, because $\Sigma f_i X_i\in \left<(d\pi)^{-1}H\right>_{\CI_c(P)}$, we get $X\in \left<(d\pi)^{-1}H\right>_{\CI_c(P)}$, concluding the statement.
	\end{proof}
	
	\begin{rem} Note that in the proof of \ref{lem:pullback}, the hypothesis  on $\pi$ being a submersion is only used to justify that $H\subset \img(d\pi)$ where $H=\pi^*\CF$.
	\end{rem}
	
	\begin{ex} Let $(M,\CF)$ a foliated manifold and $U\subset M$ an open set. Consider the inclusion $\iota_U\colon U\fto M$, which is clearly a submersion. The set $(d\iota_U)^{-1} \iota_U^* (\CF)$ is equal to $\CF|_U$, therefore by Lemma \ref{lem:pullback} we have:
		$$\iota_U^{-1}\CF=\left<\CF|_U\right>_{\CI_c(U)}=\{X|_U \st X\in \CF \y \mathrm{supp}(X)\subset U\}.$$	
	\end{ex}
	
	A result of I. Androulidakis and M. Zambon in \cite[lemma 3.2]{AZ1} shows us that there is a way to push forward foliations via surjective submersions with connected fibers, as the following proposition:
	
	\begin{lem}\label{lem:submfol} Let $\pi\colon P \to M$ be a {surjective} submersion with connected fibres. Let $\cF$ be a singular foliation on $P$, such that $\Gamma_c(\ker d\pi) \subset \cF$. Then there is a unique singular foliation  $\cF_M$ on $M$
		with $\pi^{-1}(\cF_M) = \cF$. 
	\end{lem}
	
	A version of the ``splitting theorem'' it is shown in \cite{AndrSk}, and we will recall it in the following proposition:
	
	\begin{prop}\label{prop:loc.pic.fol} Let $(M,\CF)$ be a foliated manifold and $x\in M$. Denote $k:=\dim(F_x)$ and $q:=\dim(M)-k$. There exists a neighborhood $U$ of $x$, a foliated manifold $(S,\CF_S)$ of dimension $q$ and a surjective submersion with connected fibers $\pi\colon U\fto S$ such that $\CF|_U=\pi^{-1}\CF_S$, $(F_S)_{\pi(x)}=0$ and $F_x=\ker(d_x \pi)$.	
	\end{prop}
	
	\begin{proof}
		This is Proposition 1.12 of \cite{AndrSk}, but we can make a sketch of the proof following Lemma \ref{lem:submfol}. First, there is a small enough transversal $S$ of $F_x$ with dimension $q$, an open neighborhood $U\subset M$ of $x$ and a surjective submersion with connected fibers $\pi\colon U \fto S$ satisfying $\Gamma_c\ker(d\pi)\subset \CF$. Then using Lemma \ref{lem:submfol} we get the desired foliation $\CF_S$ satisfying the conditions of this proposition.
	\end{proof}
	
	Consider the setting of proposition \ref{prop:loc.pic.fol}, i.e. let  $(M,\CF)$ be a foliated manifold, $x\in U\subset M$, $(S,\CF_S)$ and $\pi\colon U\fto S$. Using the constant rank theorem for $\pi$ one can see $U$ as an open set of $\RR^k\times \RR^{n-k}$, $S$ as an open set of $\RR^{n-k}$ and $\pi$ as the projection map. Then $\CF|_U$ is equal to $\pi^{-1}\CF_S$,  where $\CF_S$ is a foliation in $\RR^{n-k}$  that vanishes at the origin. 
	
	Therefore using Lemma \ref{lem:pullback} the local picture $\CF|_U=\pi^{-1}\CF_S$ is the span $\left<\de_{x_1},\dots , \de_{x_k}, \CF_S\right>_{\CI_c(U)}$. In particular, the local picture of a regular foliation of rank $k$ is given as the span $\left<\de_{x_1},\dots, \de_{x_k}\right>_{\CI_c(U)}$.
	
	\section{Product foliations}\label{sec:prod.fol}
	
	Let $(M, \CF_M)$ and $(N,\CF_N)$ be two foliated manifolds. In this section we use the definition of pullback foliation to get a foliation on $M\times N$ which we call the product foliation. This section is an original part of this thesis.
	
	\begin{defi}
		Given two foliated manifolds $(M,\CF_M)$ and $(N,\CF_N)$ we define the following submodule $\CF_M\times \CF_N$ of $\CX_c(M\times N)$:
		\[\CF_M\times \CF_N:=\pi_M^{-1}\CF_M\cap \pi_N^{-1} \CF_N.\]
		Here $\pi_M\Fr M\times N\fto M$ and $\pi_N\Fr M\times N\fto N$ are the projections.
	\end{defi}
	
	In order to simplify notation, for a foliated manifold $(M,\CF_M)$ and a manifold $N$, we denote:
	\begin{equation}\label{eq:prod.fol}
		\CF_M^{M\times N}:=d\pi_M^{-1}(\pi_M^*\CF_M)=\{(X,Y)\in \CX(M\times N) \st X\in \CF_M\},
	\end{equation}
	
	which consists of the projectable vector fields in $M\times N$ going to $\CF_M$.
	
	\begin{lem}\label{lem:product}
		$$\CF_M\times \CF_N=\left<\{(X_M ,X_N) \st X_M \in\CF_N \y X_N\in \CF_N\}\right>_{\CI_c (M\times N)}$$
	\end{lem}
	
	\begin{proof}
		Note that $\{(X_M,X_N) \st X_M\in\CF_N \y X_N\in \CF_N\}=\CF_M^{M\times N}\cap \CF_N^{M\times N}$. Therefore this lemma asks us to prove that:
		\[ \left(\left< \CF_M^{M\times N}\right>_{\CI_c(M\times N)}\right)\cap \left( \left< \CF_N^{M\times N}\right>_{\CI_c(M\times N)}\right)= \left< \CF_M^{M\times N}\cap \CF_N^{M\times N}\right>_{\CI_c(M\times N)}.\]

		We continue the proof by double inclusion:
		
		It is clear that
		\[ \left(\left< \CF_M^{M\times N}\right>_{\CI_c(M\times N)}\right)\cap \left( \left< \CF_N^{M\times N}\right>_{\CI_c(M\times N)}\right)\supset \left< \CF_M^{M\times N}\cap \CF_N^{M\times N}\right>_{\CI_c(M\times N)}.\]

		For the other inclusion take
		\[X\in \CF_M\times \CF_N=\left(\left< \CF_M^{M\times N}\right>_{\CI_c(M\times N)}\right)\cap \left( \left< \CF_N^{M\times N}\right>_{\CI_c(M\times N)}\right).\]
		There are functions $f_1\cdots f_n, g_1,\cdots g_k\in \CI_c(M\times N)$ and vector fields $X_1,\cdots ,X_n\in \CF_M^{M\times N}$ and $Y_1,\cdots ,Y_k\in \CF_N^{M\times N}$ such that:
		\begin{equation}\label{eq1:lem:product}
			\begin{matrix}
				X= f_1 X_1+\cdots + f_n X_n,\\
				X= g_1 Y_1+\cdots + g_k Y_k.\\
			\end{matrix}
		\end{equation}
		We can write $X_i=(x_i,x'_i)$ and $Y_i=(y_i,y'_i)$ for $x_i,y_i\in \SEC(\pi_M^* TM)$ and $x_i',y_i'\in \SEC(\pi_N^* TN)$. Because $X_i\in \CF_M^{M\times N}$ then $x_i\in \CF_M$ and using a similar argument $y_i'\in \CF_N$.\\
		
		Now using equation \ref{eq1:lem:product} we have that $\Sigma_i f_i (0,x'_i)=\Sigma_j g_j (0,y'_j)$ and $\Sigma_i f_i (x_i,0)=\Sigma_j g_j (y_j,0)$. Finally: 
		\[X=\sum_i f_i\cdot(x_i,0)+\sum_j g_j\cdot(0,y'_j)\in\left< \CF_M^{M\times N}\cap \CF_N^{M\times N}\right>_{\CI_c(M\times N)} .\]
	\end{proof}
	
	\begin{prop}
		$\CF_M\times \CF_N$ is a singular foliation. 
	\end{prop}
	\begin{proof}
		The set $\{(X_M,X_N) \st X_M\in\CF_N \y X_N\in \CF_N\}$ is involutive and locally finitely generated on square open sets. Then $\CF_M\times \CF_N$ is involutive and locally finitely generated.
	\end{proof}
	
	Now we prove that the definition of product foliation is well behaved under pullbacks:
	
	\begin{prop}\label{prop:product}
		Let $(M,\CF_M)$ and $(N,\CF_N)$ be foliated manifolds and $P, S$ manifolds with surjective submersions $\tau: P\fto M$, $\sigma: S\fto N$. Then 
		\begin{equation}\label{prop:product:eq}
			\tau^{-1}\CF_M\times \sigma^{-1} \CF_N=(\tau\times \sigma)^{-1}(\CF_M\times \CF_N).
		\end{equation}
	\end{prop}
	\begin{proof}
		Denote $W=P\times S$. Consider the following commutative diagram:
		\[\begin{tikzcd}
			P \arrow[d, "\tau"]& \arrow[l,swap,"\pi_P"] W \arrow[r,"\pi_S"] \arrow[d, "\tau\times\sigma"] & S \arrow[d,"\sigma"]\\
			M & M\times N \arrow[l,swap,"\pi_M"] \arrow[r,"\pi_N"]& N \\
		\end{tikzcd}\]
		
		Similar to equation (\ref{eq:prod.fol}), for any submodule $\CS\subset\CX(M\times N)$ we will denote:
		$$(\CS)^{W}:=d(\tau\times \sigma)^{-1}((\tau\times\sigma)^*(\CS)),$$
		which are the projectable vector fields on $W$ going to $\CS$.
		
		Note that:
		\[\left(\pi_M^{-1}\CF_M \cap \pi_N^{-1} \CF_N\right)^{W}=(\pi_M^{-1}\CF_M)^{W}\cap (\pi_N^{-1} \CF_N)^{W}.\]
		
		{\bf (1)} The right hand side of equation \eqref{prop:product:eq} is
		\begin{align*}(\tau\times \sigma)^{-1}(\CF_M\times \CF_N) &= (\tau\times\sigma)^{-1}\left(\pi_M^{-1}\CF_M \cap \pi_N^{-1} \CF_N\right)\\
			&= \left<\left(\pi_M^{-1}\CF_M \cap \pi_N^{-1} \CF_N\right)^{W}\right>_{\CI_c(W)}\\
			&= \left<(\pi_M^{-1}\CF_M)^{W}\cap (\pi_N^{-1} \CF_N)^{W}\right>_{\CI_c(W)}.
		\end{align*}	
		
		{\bf (2)} The left hand side of equation \eqref{prop:product:eq} is
		\begin{align*}
			\tau^{-1}\CF_M\times \sigma^{-1} \CF_N &= \left\{\pi_P^{-1}\tau^{-1}\CF_M\right\} \cap \left\{\pi_S^{-1}\sigma^{-1} \CF_N\right\}\\
			& = \left\{(\tau\times\sigma)^{-1}\pi_M^{-1}\CF_M\right\} \cap \left\{(\tau\times\sigma)^{-1}\pi_N^{-1} \CF_N\right\}\\
			&=\left\{\left< (\pi_M^{-1}\CF_M)^{W}\right>_{\CI_c(W)} \right\}\cap \left\{\left<(\pi_N^{-1} \CF_N)^{W}\right>_{\CI_c(W)}\right\}.
		\end{align*}
		
		Due to (1) and (2) above we only need to prove that
		{\footnotesize
			\[\left<(\pi_M^{-1}\CF_M)^{W}\cap (\pi_N^{-1} \CF_N)^{W}\right>_{\CI_c(W)}= \left\{\left< (\pi_M^{-1}\CF_M)^{W}\right>_{\CI_c(W)}\right\}\cap \left\{\left< (\pi_N^{-1} \CF_N)^{W}\right>_{\CI_c(W)}\right\}.\]}
		
		\noindent By double inclusion, using a similar argument as in lemma \ref{lem:product}, we get the desired result.\\
	\end{proof}
	
	\section{Singular foliations as sheaves}\label{sec:fol.sheaf}
	
	The definition \ref{defi:sing.fol} of foliations as submodules becomes very restrictive in the analytic and holomorphic setting or for non Hausdorff manifolds. Some examples of this fact are:
	
	\begin{itemize}
		\item For $M$ a non compact analytic or holomorphic manifold there is no compactly supported analytic or holomorphic vector field on $M$ different from zero.
		\item Consider $\RR\times \NN$ and the equivalent relation given by $(x,n)\sim (y,k)$ iff $x=y$ and $y\neq 0$. Take $M=(\RR\times \NN)/\sim$, which is a non-hausdorff manifold. Note that $M$ can be thought as the real line $\RR$ with infinitely many origins. Every compactly supported smooth vector field on $M$ must vanish in all the origins $(0,n)$.
	\end{itemize}

	Therefore, it appeared in the articles \cite{SylvainArticle} and \cite{ME2018} a different characterization of singular foliations in terms of sheaves:
	\begin{center}
		\textit{A singular foliation is an involutive and locally finitely generated subsheaf of the vectorfields sheaf $\CX$}.
	\end{center}
	
	Before we start, we give a brief introduction to sheaves. Recall that a presheaf of groups on a topological space $M$ consists of a choice of a group $\CS(U)$ for every open set $U\subset M$ and of morphisms $(-)|_V\colon \CS(U)\fto \CS(V);x\mapsto x|_V$ for every open set $V\subset U$. 
	
	This choice must satisfy some conditions, namely it must act as a contravariant functor $\CS\colon \mathrm{Op}(M)\fto \CC$, where $\mathrm{Op}(M)$ is the category of open sets on $M$ and $\CC$ is the category of groups (it is possible to have $\CC$ as the category of algebras or modules).
	
	A presheaf is a sheaf if for every cover $\{U_i\}_{i\in I}$ of an open set $U$ the following axioms are satisfied:
	
	\begin{itemize}
		\item \textbf{Locality:} If $x,y\in \CS(U)$ are such that $x|_{U_i} = y|_{U_i}$ for each $i$ then $x = y$.
		\item \textbf{Gluing:} For every family $x_i\in \CS(U_i)$ satisfying $x_i|_{U_i\cap U_j} = x_j|_{U_i\cap U_j}$ for all $i,j\in I$, there is a section $x \in \CS(U)$ such that $x|_{U_i} = x_i$.
	\end{itemize}  
	
	Examples of sheaves on a manifold $M$ are the presheaf of smooth functions $\CI$ and of vector fields $\CX$. The presheaf of compactly supported functions $\CI_c$ fails to satisfy the gluing axiom for infinite covers, therefore is not a sheaf.
	
	In this section we prove that any singular foliation (as submodule of $\CX_c(M)$) on a smooth (Hausdorff) manifold can be considered as an involutive and locally finitely generated smooth subsheaf of the vector fields sheaf $\CX$, as the following theorem shows.
	\begin{theorem}\label{thm:sheafol} For any smooth (Hausdorff) manifold $M$, we have the following:
		\begin{itemize}
			\item There is a bijection between submodules of $\CX_c(M)$ and subsheaves of $\CX$, described in proposition \ref{prop:sheafol}.
			\item The condition of being locally finitely generated is invariant under this bijection.
			\item The condition of involutivity is invariant under this bijection.
		\end{itemize}
	\end{theorem}
	\begin{proof} This theorem is a direct consequence of propositions \ref{prop:sheafol}, \ref{prop:sheafingen} and \ref{prop:sheafinv}, that will appear later in this section.
	\end{proof}
	
	Theorem \ref{thm:sheafol} is a logical continuation to the work in article \cite{AMsheaf} by I. Androulidakis and M. Zambon. Moreover, its content can be easily generalized. Given a Lie algebroid $A$, replacing $\CX_c(M)$ by $\Gamma_c(A)$ and $\CX$ by $\Gamma_A:=\Gamma(-,A)$, one can follow the proofs given here and get the same results.
	
	It is important to mention that these results are equivalent to some statements appeared earlier in \cite{RoyThesis} by Roy Wang. We claim that our work differs from \cite{RoyThesis} in a matter of presentation, being more direct. In what follows we will explain our work.
	
	We use the notion of global hull to ``forget'' the compactly supported condition:
	
	\begin{defi}\label{defi:globalhull} Given $\CF$ a submodule of $\CX_c(M)$, the {\bf global hull} of $\CF$ is given by:
		$$\widehat{\CF}:=\{X\in \CX(M) \st fX\in \CF \hspace{.1in} \forall f\in \CI_c(M)\}.$$
		Given $\CS$ a submodule of $\CX(M)$ one can also define its compact elements:
		$$(\CS)_c:=\{X\in \CS \st \mathrm{supp}(X) \text{ is compact}\}=\left<\CS\right>_{\CI_c(M)}.$$
	\end{defi}

	An important property of global hulls is the following:
	
	\begin{lemma} For $\CF\subset \CX_c(M)$ a submodule, $\CS$ a subsheaf of $\CX$ and $U\subset M$ open we get the following equalities:
		$$\left(\widehat{\CF}\right)_c=\CF$$
		$$\reallywidehat{(\left(\CS(U)\right)_c)}=\CS(U)$$
	\end{lemma}
	\begin{proof}
		The first equality and the inclusion $\CS(U)\subset \reallywidehat{(\left(\CS(U)\right)_c)}$ are clear, we only need to prove $\reallywidehat{(\left(\CS(U)\right)_c)}\subset \CS(U)$.
		
		Take $X\in \reallywidehat{(\left(\CS(U)\right)_c)}$ and $\{\varphi_i\}_{i\in I}$ a partition of unity for $U$ with functions with compact support. Then there exists a cover $\{U_j\}_{j\in J}$ of $U$ such that the sum $\Sigma_i \varphi_i X$ is finite in each $U_j$. Moreover, $\varphi_i X\in \CS(U)$ therefore $X |_{U_j}=\Sigma_i \varphi_i X |_{U_j}\in \CS(U_j)$. By the gluing axiom of sheaves, there exists $Y\in \CS(U)$ such that $Y|_{U_i}=X|_{U_j}$ and by the locality axiom of sheaves $X=Y\in \CS(U)$.
	\end{proof}
	
	Given a subsheaf $\CS$ of $\CX$ it is easy to define a submodule of $\CX_c(M)$, just take $\CF:=(\CS(M))_c$. A reasonable question is if we can recover $\CS$ from $\CF$, which is answered in the following proposition:
	
	\begin{lem}\label{lem:glob.eq} Let $\CS$ be a subsheaf of $\CX$. Denote $\CF=(\CS(M))_c$ then for any $U\subset M$ open we have $(\CS(U))_c=\iota_U^{-1} \CF$ and therefore
		\[\CS(U)=\widehat{\iota_U^{-1} \CF}.\]
	\end{lem}
	\begin{proof} We will prove by double inclusion that for every open set $U\subset M$ we get the following equality: $(\CS(U))_c=\iota_U^{-1}\CF$, which also implies that $\CS(U)=\widehat{\iota_U^{-1}\CF}$.
		
		Take $X\in (\CS(U))_c$, then $X\in \CS(U)$ and the two sets $U$ and $M/\text{supp}(X)$ are a cover of $M$. Using the gluing property of $\CS$ as sheaf, there exists $Y\in \CS(M)= \widehat{\CF}$ such that $Y|_U= X$ and $Y_{M/\text{supp}(X)}= 0$. Note that $Y$ has compact support, then $Y\in \CF$ and this support is in $U$, therefore by definition $X= Y|_U\in \iota_U^{-1}\CF$.
		
		For the converse take $X\in \iota_U^{-1} \CF$. By definition there exists $Y\in \CF=(\CS(M))_c\subset \CS(M)$ such that $Y|_U= X$ and $\text{supp}(Y)\subset U$. Therefore $X=Y|_U\in (\CS(U))_c$. 
	\end{proof}

	Therefore a natural way to define a presheaf for a singular foliation $\CF$ is as follows:
	\[\CS^\CF(U):=\widehat{\iota^{-1}_U \CF}.\]
	
	An easy consequence of lemma \ref{lem:glob.eq} is that starting from a sheaf $\CS$ we get the module $\CF$ and $\CS^\CF=\CS$, i.e. the process of getting a module from a sheaf is invertible.
	
	Now we just need to prove that, starting from a module $\CF\subset \CX(M)$, the presheaf $\CS^\CF$ is indeed a sheaf.
	
	\begin{lem}\label{lem:sheaf.fol} Given $\CF\subset \CX_c(M)$ a $\CI(M)$-submodule, $\CS^\CF$ is a sheaf.
	\end{lem}
	\begin{proof} Note that $\CS^\CF$ satisfies the locality axiom because it is a sub presheaf of $\CX$, which is a sheaf.
		
		We will prove the gluing condition. Take $\{U_i\}_{i\in I}$ a family of open sets contained in and covering $U$, and $X_i\in \CS^\CF(U_i)$ such that $X_i |_{U_i\cap U_j} = X_j |_{U_i\cap U_j}$. Then there exists $X\in \CX(U)$ such that $X|_{U_i}=X_i$. We will show that $X\in \CS^\CF(U)$.
		
		Take $f\in \CI_c(U)$. Then there exists finitely many $U_i$ in the family that cover $\text{supp}(f)$. Without loss of generality call them $U_1,\dots, U_k$ and denote $U_0=U-\mathrm{supp}(f)$. There exists a partition of unity $\varphi_0,\varphi_1,\dots,\varphi_k\subset \CI_c(U)$  subordinated to $U_0,U_1,\dots,U_k$. For all $j>0$ the functions $\varphi_j$ have compact support on $U_j$, then $\varphi_j f X = \varphi_j f X_j \in \iota_U^{-1}\CF$, therefore:
		$$fX= \sum_{j>0} \varphi_j f X\in \iota_U^{-1}\CF.$$
		
		Using that $f$ is arbitrary, we get $X\in \widehat{\iota_U^{-1} \CF}=\CS^\CF(U)$.
	\end{proof}
	
	\begin{prop}\label{prop:sheafol} Let $\CC$ denote the collection of submodules of $\CX_c(M)$ and $\CSH$ the subsheaves of $\CX$.
		
		The map $\CC\fto\CSH;\CF\mapsto \CS^\CF$ is a bijection. The inverse map is $\CS\mapsto \CF^\CS:=(\CS(M))_c$.
	\end{prop}
	\begin{proof} Lemma \ref{lem:sheaf.fol} gives a map $\CC\fto\CSH;\CF\mapsto \CS^\CF$. Lemma \ref{lem:glob.eq} shows that this map is the inverse of the map $\CSH\fto \CC;\CS\mapsto \CF^\CS:=(\CS(M))_c$. 
	\end{proof}
	
	Now we want to prove the following:
	
	\begin{prop}\label{prop:sheafingen} $\CF$ is locally finitely generated as in the sense of definition \ref{def:loc.gen} if and only if $\CS^\CF$ is locally finitely generated as a sheaf.
	\end{prop}
	For a manifold $M$ remember than any open set $U\subset M$ is a manifold. Therefore Proposition \ref{prop:sheafingen} is a consequence of the following lemma:
	
	\begin{lem}\label{lem:loc.gen}
		Let $U$ be a manifold, $X_1,\cdots, X_n\in \CX(U)$, and $\CF:=\left< X_1,\cdots,X_n \right>_{\CI_c(U)}$, then $\widehat{\CF}=\left< X_1,\cdots,X_n \right>_{\CI(U)}$.
	\end{lem}
	\begin{proof}
		Let $H=\left< X_1,\cdots,X_n \right>_{\CI(U)}$, we want to show that $\widehat{\CF}=H$. It is clear that $H\subset \widehat{\CF}$. Now we prove the other inclusion. Take $X\in \widehat{\CF}$ and $\{\varphi_i\}_{i\in I}$ a partition of unity for $U$ with compact support. Then for every $i$ we have that $\varphi_i X\in \CF$. Therefore there exist $\alpha^1_i,\dots,\alpha^n_i\in \CI_c(U)$ such that:
		\[\varphi_i X= \alpha^1_i X_1 +\dots +\alpha^n_i X_n\]
		
		with $\text{supp}(\alpha^j_i)\subset \text{supp}(\varphi_i)$, therefore $\Sigma_i \alpha^j_i \in \CI(M)$ for all $j$. Finally:
		\begin{eqnarray*}
			X=\sum_{i\in I} \varphi_i X &=&  \sum_{i\in I} \left(\alpha^1_i X_1 +\dots +\alpha^n_i X_n\right)\\
			&=&  \left(\sum_{i\in I} \alpha^1_i\right) X_1 +\dots +\left(\sum_{i\in I} \alpha^n_i\right) X_n\in H\\
		\end{eqnarray*}
	\end{proof}
	
	Lemma \ref{lem:loc.gen} shows that if $\iota^{-1}_U \CF$ is generated by some elements in $\widehat{\iota^{-1}_U \CF}$ then these elements also generate the global hull $\CS^\CF(U)=\widehat{\iota^{-1}_U \CF}$ and viceversa.\\
	
	Finally we need to prove the involutivity condition:
	
	\begin{prop}\label{prop:sheafinv} A submodule $\CF$ of $\CX_c(M)$ is involutive if and only if $\CS^\CF$ is involutive.
	\end{prop}
	\begin{proof}
		($\Leftarrow$) Suppose $\CS^\CF$ is involutive, then $\CS^\CF(M)$ is involutive and so $\CF=(\CS^\CF(M))_c$.
		
		($\Rightarrow$) Suppose $\CF$ is involutive, then for all $U\subset M$ open we get that $\iota^{-1}_U\CF=(\CS^\CF(U))_c$ is involutive. We just need to prove that if a submodule $H\subset \CX(U)$ is involutive then $\widehat{H}$ is also. Then take $X,Y\in \widehat{H}$ and $f\in \CI_c(U)$ we will prove that $f[X,Y]\in H$. Take $g\in \CI_c(M)$ such that $g|_{\text{supp}(f)}=1$. Then using Leibniz formula we obtain $f[X,Y]=[fX,gY]-(gY(f)) X \in H$.
	\end{proof}
	
	As a consequence of the equivalent description between subsheaves and submodules given in theorem \ref{thm:sheafol}, get the following proposition:
	
	\begin{prop}\label{prop:fol.loc.pro}
		Let $\CF_1$ and $\CF_2$ be two singular foliations on a manifold $M$ and $\CU:=\{U\subset M\}_{i\in I}$ an open cover. Then $\CF_1=\CF_2$ if and only if $\iota_{U_i}^{-1}\CF_1=\iota_{U_i}^{-1}\CF_2$ for all $i\in I$.
	\end{prop}
	\begin{proof} It is clear that $\CF_1=\CF_2$ implies $\iota_{U_i}^{-1}\CF_1=\iota_{U_i}^{-1}\CF_2$ for all $i\in I$.
		
		For the converse take $\CS^{\CF_1}$ and $\CS^{\CF_2}$ the corresponding sheaves of $\CF_1$ and $\CF_2$. By hypothesis we have: $$\CS^{\CF_1}(U_i)=\widehat{\iota_{U_i}^{-1}\CF_1}=\widehat{\iota_{U_i}^{-1}\CF_2}=\CS^{\CF_1}(U_i).$$
		Using elemental properties of sheaves (the restriction maps and the gluing axiom) it is possible to prove by double inclusion that $\CS^{\CF_1}(M)=\CS^{\CF_2}(M)$ and therefore $\CF_1=\CF_2$ (also $\CS^{\CF_1}=\CS^{\CF_2}$).
	\end{proof}
	
	This proposition says that two foliations are equal if and only if they are locally equal.
	
	\cleardoublepage

	
	\chapter{Hausdorff Morita equivalence for singular foliations}\label{ch:2}
	
	The first two sections of this chapter are a summary of sections 1.2.2 and 2 of the article \cite{AndrSk} by I. Androulidakis and G. Skandalis. There are no new results but we give a new proof for proposition \ref{prop:exp.in.aut} and we changed the order and presentation of the theorems. The intention for the rest of this chapter is to introduce Hausdorff Morita equivalence for singular foliations as in the article \cite{ME2018} by M. Zambon and myself. We show the definition, display some easy invariants, and present several classes of examples.

	\section{The automorphism group of a foliated manifold}\label{sec:aut.fol}
	
	We start this section with the group of automorphisms for a singular foliation. This is a quite interesting object by its own and it serves as a motivation for many other notions that  we will introduce later on:
	
	\begin{defi}
		Let $\mathcal{F}$ be a singular foliation on a manifold $M$. The 
		\textbf{automorphism group} of $\mathcal{F}$ is the group
		\begin{equation*}
			\mathrm{Aut}(\mathcal{F}) := \{\; \varphi\in\mathrm{Diff}(M)\;|\; \varphi^{-1}(\mathcal{F})\subset \mathcal{F} \;\}.
		\end{equation*} 
		
		Also due to the compact support for any element $X\in\mathcal{F}$ we may define the \textbf{exponential map} $\mathrm{exp}:\mathcal{F}\to \mathrm{Diff}(M)$ by $\mathrm{exp}(X):=\varphi_X^{1}$, where $\varphi_X^{1}$ is the flow of $X$ at time $1$. The exponential map allows us to define the {\bf exponential group}:
		\begin{equation*}
			\mathrm{exp}(\mathcal{F}) :=\textnormal{The group generated by } \{\; \mathrm{exp}(X)\in\mathrm{Diff}(M)\;|\; X\in \mathcal{F} \;\}.
		\end{equation*}
	\end{defi}
	
	\begin{rem} By definition, for any $\phi\in \mathrm{Diff}(M)$ and $X\in \CX(M)$, the pushforward of $X$ under $\phi$ is given by:
		\[(\phi_* X)_x=(d\phi) X_{\phi^{-1}x}=\left((d\phi\circ (\phi^{-1})^*) X\right)_x=\left(((\phi^{-1})^*\circ d\phi) X\right)_x.\]
	\end{rem}
	
	The following proposition was proven first in \cite{AndrSk} using infinite dimensional techniques. Ori Yudilevich and I provided a ``finite dimensional'' proof in \cite{AutOri} which we will show here: 
	
	\begin{prop}\label{prop:exp.in.aut}
		Let $\mathcal{F}$ be a singular foliation on $M$. For any $X\in \CX_c(M)$ such that $[X,\CF]\subset \CF$ we have that
		$$(\varphi_X^1)^{-1}(\mathcal{F})=\mathcal{F}.$$
		
		This also implies that
		\begin{equation*}
			\mathrm{exp}(\mathcal{F})\subset \mathrm{Aut}(\mathcal{F}).
		\end{equation*}
	\end{prop}
	
	\begin{proof}
		We need to show that $(\varphi_X^1)_*(Y)\in \mathcal{F}$ for all $Y\in\mathcal{F}$. Indeed, this implies that $(\varphi_X^1)^{-1}(\mathcal{F})\subset\mathcal{F}$, and since $X\in\mathcal{F} \Rightarrow -X\in\mathcal{F}$ and $(\varphi_X^1)_*((\varphi_{-X}^1)_*(Y))=Y$, also that $(\varphi_X^1)^{-1}(\mathcal{F})=\mathcal{F}$. 
		
		Let then $X,Y\in\mathcal{F}$ and let us shorten the notation for the flow of $X$ to $\varphi^t = \varphi_X^t$. Since $\mathrm{supp}(X)\subset M$ is compact, there exists a precompact open neighborhood $U$ of $\mathrm{supp}(X)$ in $M$. Let $\{\rho_U, \rho_V\}$ be a partition of unity subordinate to the open cover $\{U,V:=M\backslash \mathrm{supp}(X)\}$ of $M$. Since $U$ is precompact, $\rho_U$ has compact support, and hence $\rho_U\mathcal{F}$ is finitely generated by part (c) of Lemma \ref{lemma:unionoffinite}. Fixing generators $Y^1,...,Y^N\in\mathfrak{X}(U)$ of $\rho_U\mathcal{F}$, we may write
		\begin{equation*}
			Y = \rho_U Y + \rho_V Y =  \sum_{i=1}^N f_i Y^i + \rho_V Y,
		\end{equation*}
		for some $f_i\in C^\infty_c(U)$. Now, since $\varphi^t\big|_{M\backslash\mathrm{supp}(X)} = \mathrm{Id}$ for all $t$ and hence $(\varphi^1)_*(\rho_VY)=\rho_VY$, we see that the problem is reduced to showing that $(\varphi^1)_*(Y^i) = \sum_j f^i_j Y^j$, for some functions $f^i_j\in C^\infty(U)$. To this end, we compute the following:
		\begin{equation*}
			\begin{split}
				\frac{d}{dt}((\varphi^t)_* Y^i)_x &= \frac{d}{dt}(d\varphi^t) Y^i_{\varphi^{-t}(x)} \\ &= - \frac{d}{ds}\Big|_{s=0}(d\varphi^{t-s}) Y^i_{\varphi^{-t+s}(x)} \\ &= - (d\varphi^t) \frac{d}{ds}\Big|_{s=0} (d\varphi^{-s}) Y^i_{\varphi^s(\varphi^{-t}(x))} \\
				&= (d\varphi^t)[Y^i,X]_{\varphi^{-t}(x)}.
			\end{split}
		\end{equation*}
		We claim that $[Y^i,X] \in \rho_U\mathcal{F}$. Indeed, choosing $\lambda\in C^\infty_c(U)$ that is 1 on $\mathrm{supp}(X)$, we have $[Y^i,X]= [Y^i,\lambda X] = Y^i(\lambda) X + \lambda[Y^i,X] = Y^i(\lambda) X + [\lambda Y^i,X]\in \rho_U\mathcal{F}$, where we used that $X(\lambda)=0$, $X\in \rho_U\mathcal{F}\subset \CF$ and $\mathcal{F}$ is involutive. Hence, we can write $[Y^i,X] = \sum_j \gamma^i_j Y^j$ for some functions $\gamma^i_j\in C^\infty_c(M)$. We thus have:
		\begin{equation*}
			\begin{split}
				\frac{d}{dt}(d\varphi^t) Y^i_{\varphi^{-t}(x)} =  \sum_{j=1}^N \gamma^i_j(\varphi^{-t}(x)) \left((d\varphi^t)Y^j_{\varphi^{-t}(x)}\right).
			\end{split}
		\end{equation*}
		Now, for a fixed $x\in M$, this is a system of $N$ equations indexed by $i$, each an equality  of curves in $T_xM$. Fixing a basis of $T_xM$, every component is a linear first ordinary partial differential equation of the type $\dot{v}(t)=A(t)v(t)$, with $v:I\to \mathbb{R}^N$ and $A:I\to \mathrm{End}(\mathbb{R}^N)$, and its solution is given by $v(t)= e^{\int_0^t A(\epsilon)d\epsilon}v(0)$. Thus, writing $\gamma$ for the $N\times N$ matrix whose entries are $\gamma^i_j$, we have at $t=1$ that:
		\begin{equation*}
			((\varphi^1)_*Y^i)_x =  (d\varphi^1)Y^i_{\varphi^{-1}(x)} = \sum_{j=1}^N(e^{\int_0^1\gamma(\varphi^{-\epsilon}(x))d\epsilon})^i_j Y^j_x,
		\end{equation*}
		where the exponential is the exponential of $N\times N$ matrices. Clearly, the coefficients of $Y^j_x$ in the final expression are smooth functions of $x$, and hence we are done. As a bonus, we have also obtained an explicit formula for $(\varphi^1)_*Y^i$ in terms of the bracket of the $Y^i$'s with $X$ (which is encoded in the $\gamma$ matrices). 
	\end{proof}
	
	\begin{rem} Proposition \ref{prop:exp.in.aut} can be improved by saying that $\mathrm{exp}(\CF)\subset \mathrm{Aut}(\CF)$ is a normal subgroup. This is a consequence of the fact that for any $\phi\in \mathrm{Aut}(\CF)$ and $X\in \CF$ we have $\phi\circ \mathrm{exp}(X)\circ \phi^{-1}=\mathrm{exp}(\phi_* X)$ with $\phi_* X\in \CF$.
	\end{rem}
	
	\begin{lem}\label{lem:leaves} The leaves of $\CF$ are the orbits of the group $\mathrm{exp}(\CF)$.
	\end{lem}
	\begin{proof} It is a direct consequence of the proof of the Stefan-Sussmann theorem \cite[\S 2.3]{GrudMSch}.
	\end{proof}
	
	\section{Bisubmersions}\label{sec:bisub}
	
	Bisubmersions were first introduced in \cite{AndrSk}. In this section bisubmersions will be seen as a groupoid-like structure for the group of automorphisms. In lemma \ref{lem:bisect.diff} and in corollary \ref{cor:diff.bisect} we explain more explicitly what this means.
	
	Most of the results in this section are taken from \cite{AndrSk} with a variation in the order and interpretation. We will make notice of the original results.
	
	\begin{defi}\label{def:bisub} Let $(M,\CF_M)$ and $(N,\CF_N)$ foliated manifolds.
		\begin{enumerate}
			\item A bisubmersion between $(M,\CF_M)$ and $(N,\CF_N)$ is a smooth manifold $U$ with two (not necesarily surjective) submersions $\bs\colon U\fto M$ and $\bt\colon U\fto N$ such that $\bs^{-1}\CF_M=\bt^{-1}\CF_N=\ker_c(d\bs)+\ker_c(d\bt)$.
			\item A bisubmersion of $(M,\CF_M)$ is a bisubmersion between $(M,\CF_M)$ and itself.
		\end{enumerate}
	\end{defi}
	
	Being a bisubmersion is a local statement as the following lemma shows:
	
	\begin{lemma}\label{lem:bisub.loc} Let $(M,\CF_M)$ be a foliated manifold $U$ be a manifold, $\bt,\bs\colon U\fto M$ submersions and $\{U_i\}_{i\in I}$ a cover of $U$. The triple $(U,\bt,\bs)$ is a bisubmersion if and only if $(U_i,\bt|_{U_i},\bs|_{U_i})$ are bisubmersions for all $i\in I$.	
	\end{lemma}
	\begin{proof}This is a direct consequence of proposition \ref{prop:fol.loc.pro}.
	\end{proof}
	
	\begin{defi} Consider a foliated manifold $(M,\CF)$, a bisubmersion  $(V,\bt,\bs)$  and $x\in \bs(V)$. 
		\begin{enumerate}
			\item[(i)] A {\bf bisection} at $x$ consists of a local $\bs$-section $\sigma\colon M'\fto V$, where $M'$ is a neighbourhood of $x$ in $\bs(V)$, such that the image of $\sigma$ is transverse to the fibres of $\bt$.
			\item[(ii)] {Given a  diffeomorphism $\phi$ between open subsets of $M$, a bisubmersion $(V,\bt,\bs)$ is said to {\bf carry} $\phi$ at $v\in V$ if  there exists a bisection $\sigma$ through $v$ such that $\phi=\bt\circ\sigma$.}
		\end{enumerate}
	\end{defi}
	
	\begin{lem}\label{lem:e.bisect}
		Let $(V,\bt,\bs)$ be a bisubmersion of a foliated manifold $(M,\CF)$ and $v\in V$. There exists a bisection $\sigma$ near $\bs(v)$.
	\end{lem}
	
	A relation between bisections and the automorphism group is shown in the following lemma:
	
	\begin{lem}\label{lem:bisect.diff} Let $\phi$ be a diffeomorphism of $M$ carried by a bisubmersion $(V,\bt,\bs)$ at $v\in V$. Then $\phi$ preserves the foliation near $\bt(v)$, i.e. in a neighborhood of $\bt(v)$ we get $\phi^{-1}\CF=\CF$.	
	\end{lem}
	\begin{proof}
		Denote $\sigma:U_1\fto V$ the bisection and $U_2:=\bt\circ\sigma(U_1)$. Note that $\bs\circ \sigma:U_1\fto U_1$ is equal to the identity, then $\sigma^{-1}(\bs^{-1}\CF)=\CF$. Finally, using the property $\bs^{-1}\CF=\bt^{-1}\CF$ of bisubmersions, we get $\sigma^{-1}(\bt^{-1}\CF)=\CF$, which is $\phi^{-1}\CF=\CF$ in a neighborhood of $\bt(v)$ (namely $U_2$).
	\end{proof}
	
	In corollary \ref{cor:diff.bisect}, we will prove that the converse is also true, i.e. any element $\phi\in \mathrm{Aut}(\CF)$ can be carried locally by a bisubmersion. Before that, we will present certain constructions of bisubmersions which endow the collection of bisubmersions with a group like structure:

	\begin{defi}\label{def:inv.comp.bisub}
		Let $(U,\bt_U,\bs_U)$ and $(V,\bt_V,\bs_V)$ be bisubmersions.
		\begin{enumerate}
			\item[(i)] The {\bf inverse} bisubmersion of $U$ is $U^{-1}:=(U,\bs_U,\bt_U)$, the bisubmersion obtained by interchanging source and target.
			\item[(ii)] Consider the following diagram:
			\[\begin{tikzcd}
				& & U{}_{\bs_U} \!\times_{\bt_V} V \arrow[rd,"p_V"] \arrow[ld,swap,"p_U"] &  & \\
				& U \arrow[dr,"\bs_U"] \arrow[dl,swap, "\bt_U"] & & V \arrow[dl,swap,"\bt_V"] \arrow[dr,"\bs_V"] & \\
				M& 	& M & & M \\
			\end{tikzcd}\]
			where $p_U,p_V$ are the projections onto $U$ and $V$. Denote $W:= U{}_{\bs_U} \!\times_{\bt_V} V$, then the \textbf{composition} of $U$ with $V$ is defined as: $$U\circ V:=(W,\bt_W:=\bt_U\circ p_U,\bs_W:=\bs_V\circ p_V).$$
		\end{enumerate}
	\end{defi} 
	
	\begin{lem}\label{lem:sub.bisub} Let $(U,\bt,\bs)$ be a bisubmersion of $(M,\CF)$ and $\pi:W\fto U$ be a submersion. Then $(W,\bt\circ\pi,\bs\circ\pi)$ is a bisubmersion.
	\end{lem}
	\begin{proof} By lemma \ref{lem:bisub.loc}, this is a local statement. Using the constant rank theorem, it is enough to prove it for $W=U\times V$ and $\pi$ the projection onto $U$, in which case the Lemma is clear.
	\end{proof}
	
	\begin{prop} The inverse and compositions of bisubmersions are again bisubmersions.
	\end{prop}
	\begin{proof} The inverse of a bisubmersion is clearly a bisubmersion. Therefore we will only prove that the composition of two bisubmersions is again a bisubmersion.
		
		Use $(W,\bt_W,\bs_W)$, $p_U,p_V$ as in definition \ref{def:inv.comp.bisub} and denote $\mu:=\bs_U\circ p_U=\bt_V\circ p_V$. Note that, $\ker(d\mu)=\ker(dp_U)\oplus \ker(dp_V)$. From Lemma \ref{lem:sub.bisub}, it follows that $(W,\bt_W,\mu)$ and $(W,\mu,\bs_W)$ are bisubmersions, and also that $(\bt_W)^{-1}\CF=(\bs_W)^{-1}\CF$. We only need to prove that $(\bt_W)^{-1}\CF=\ker(d\bt_W)+\ker{(d\bs_W)}$. We will do it by double inclusion. $\supseteq$ is clear. For the other inclusion we have the following:
		\begin{align*}
			(\bt_W)^{-1}\CF &= \ker{(d\bt_W)}+\ker{(d\mu)}\\
			&= \ker{(d\bt_W)}+\ker{(dp_U)}+\ker(dp_V)\\
			&= \ker{(d\bt_W)}+\ker(dp_V)\\
			&\subseteq \ker{(d\bt_W)}+\ker(d\bs_W).
		\end{align*}
		since $\ker{(dp_U)} \subseteq \ker{(d\bt_W)}$ and $\ker(dp_V)\subset \ker(d\bs_W)$.
	\end{proof}
	
	\begin{defi}
		Given a foliated manifold $(M,\CF)$ and two bisubmersions {$U$ and $V$,} a smooth map $f\colon U\fto V$ is called a {\bf morphism} of bisubmersions if it commutes with the source and the target maps of $U$ and $V$.
	\end{defi}
	
	For any two bisubmersions $U$ and $V$, If $\sigma$ is a bisection of $U$ and $f\colon U\fto V$ a morphism, then $f\circ \sigma$ is a bisection of $V$ carrying the same diffeomorphism as $\sigma$.
	
	The following proposition can be found in \cite[\S 2.3]{AndrSk}.
	
	\begin{prop}\label{path} Given $x_0\in M$, let $X_1,\dots ,X_n\in \CF$ be vector fields whose classes in the fibre $\CF_{x_0}$ form a basis. For $v=(v_1,\dots, v_n)\in \RR^n$, put $\varphi_v= exp(\Sigma_i v_i X_i)$, {where $exp$ denotes the time one flow.} 
		
		Put $W=\RR^n\times M$, $\bs(v,x)=x$ and $\bt(v,x)=\varphi_v(x)$. There is a neighbourhood $U\subset W$ of ${(0,x_0)}$ such that $(U,\bt,\bs)$ is a bisubmersion.	
	\end{prop}
	
	\begin{defi}
		A bisubmersion as in proposition \ref{path}, {when it has $\bs$-connected fibres,} is called {\bf path holonomy bisubmersion}.
	\end{defi}

	\begin{rem}
		Neighbourhoods of points of the form $(0,x_0)$ in  path holonomy bisubmersions carry the identity.
	\end{rem}
	
	\begin{lem}\label{lem:sub.to.U} Let $(V,\bt_V,\bs_V)$ a bisubmersion and $v_0\in V$. Assume that $s(v_0)=x_0$ and that $V$ carries the identity diffeomorphism at $v_0$. Then for every path holonomy bisubmersion $U$ containing $(0,x_0)$, there exists an open neighbourhood $V'\subset V$ of $v_0$, and a submersion $f:V'\fto U$ which is a morphism of bisubmersions and $f(v_0)=(0,x_0)$.
	\end{lem}
	\begin{proof} The proof of this lemma can be found in \cite[\S 2.3]{AndrSk}
	\end{proof}

	\begin{cor}\label{carry}
		Let $(V_1,\bt_1,\bs_1)$ and $(V_2,\bt_2,\bs_2)$ be bisubmersions and $v_1\in V_1$, $v_2\in V_2$ be such that $\bs_1(v_1)=\bs_2(v_2)=:x_0$. Then:
		\begin{enumerate}
			\item[(i)] If there is a local diffeomorphism carried both by $V_1$ at $v_1$ and by $V_2$ at $v_2$, there exists an open neighbourhood $V'_1$ of $v_1$ in $V_1$ and a morphism $f:V_1'\fto V_2$ such that $f(v_1)=v_2$.
			\item[(ii)] If there is a morphism $g:V_1\fto V_2$ such that $g(v_1)=v_2$ then there exists an open neighbourhood $V'_2$ of $v_2$ in $V_2$ and a morphism $f:V_2'\fto V_1$ such that $f(v_2)=v_1$.
		\end{enumerate}
	\end{cor}
	\begin{proof}  \textbf{(i):} Take $(V_1,\bt_1,\bs_1)$ and $(V_2,\bt_2,\bs_2)$ two bisubmersions carrying a diffeomorphism $\phi$ at $v_1\in V_1$ and at $v_2\in V_2$ respectively. Note that $(V_1,\phi^{-1}\bt_1,\bs_1)$ and $(V_2,\phi^{-1}\bt_2,\bs_2)$ are bisubmersions that carry the identity at $v_1,v_2$ respectively.
		
		By proposition \ref{lem:sub.to.U}, for any path holonomy bisubmersion $U$ containing $(0,x_0)$, there exist submersion maps of bisubmersions $f_1\colon V_1'\fto U$ and $f_2\colon V_2'\fto U$ such that $f_1(v_1)=(0,x)=f_2(v_2)$.
		
		Because $f_2$ is a submersion, it is possible to find a local section $f_2^{-1}$ nearby $(0,x_0)$ such that $f_2^{-1}(0,x_0)=v_2$. Note that $f_2^{-1}$ is also a map of bisubmersions. Finally the map $f:=f_2^{-1}\circ f_1$ is a morphism of bisubmersion defined locally around $v_1$, going to $V_2$ and such that $f(v_1)=v_2$.
		
		\textbf{(ii):} This is a consequence of part (i).
	\end{proof}

	\begin{rem}\label{rem:invertmorph} Note that, if $V$ is a bisubmersion that carries the identity diffeomorphism around $x_0\in M$, and $U$ is a path holonomy bisubmersion containing $(0,x_0)$, then there is an open neighborhood $U'\subset U$ of $(0,x_0)$ which can be embedded into $V$. This follows from  lemma \ref{lem:sub.to.U} and corollary \ref{carry} (ii).
	\end{rem}

	Note that, if $\phi$ is carried by a bisubmersion $U$ and $\sigma$ by a bisubmersion $V$ then $\phi\circ\sigma$ is carried by $U\circ V$. Also, the inverse $\phi^{-1}$ is carried by $U^{-1}$.
	
	\begin{cor}\label{cor:diff.bisect} If $\phi\in \mathrm{Aut}(\CF)$ then for each $x\in M$ there exists a bisubmersion $V_x$ carrying $\varphi$ at an element $v_x\in V_x$.
	\end{cor}
	\begin{proof} Given $x\in M$ take $(V_x,\bt,\bs)$ a path holonomy bisubmersion for $x$. $V_x$ carries the identity at $v_x=(0,x)$. Note that using that $\varphi^{-1}\CF=\CF$ we get easily that $(V_x,\varphi\circ\bt,\bs)$ is a bisubmersion that carries $\varphi$ at $v_x$.
	\end{proof}
	
	\begin{rem} Unless there exists a global bisubmersion, not every element in $\mathrm{Aut}(\CF)$ can be seen as a bisection. The statement of corollary \ref{cor:diff.bisect} is local.
	\end{rem}
	
	\section{Definition of {Hausdorff} Morita equivalence}\label{sec:ME1}
	
	In this chapter we will introduce the main definition of article \cite{ME2018} by Marco Zambon and myself. Before we start, we state a lemma that will be used repeatedly, and which is an immediate consequence of \cite[proposition 7.1]{GinzburgGrot}:
	
	\begin{lem}\label{lem:subcon}
		Let $A$ and $B$ be manifolds, { $k\ge 0$}, and $f\colon A\fto B$ a surjective submersion with $k$-connected fibers. If $B$ is $k$-connected then $A$ is $k$-connected.
	\end{lem}	
	
	Now we state an original definition first introduced in \cite{ME2018} which is inspired by: Morita equivalence for Lie algebroids in \cite[\S 6.2]{GinzburgGrot}, for Poisson manifolds as in \cite[\S 9.2]{CFPois} and for regular foliations as in \cite[\S 2.7]{MolME}.
	
	\begin{defi}\label{def:defMEfol}
		Two singular foliations  $(M,\cF_M)$ and $(N,\cF_N)$
		are {\bf Hausdorff
			Morita equivalent} if there exists a manifold $P$ and two \emph{surjective submersions with connected fibers} $\pi_M:P\fto M$ and $\pi_N:P\fto N$  such that $\pi_M^{-1}\CF_M=\pi_N^{-1}\CF_N$. In this case we write $(M,\cF_M)\simeq_{ME} (N,\cF_N)$.
		\[\begin{tikzcd}
			& P \arrow[dl,swap, "\pi_M"] \arrow[dr, "\pi_N"]   &\\
			(M,\CF_M)& &(N,\CF_N)   
		\end{tikzcd}\]
	\end{defi}
	
	Definition \ref{def:defMEfol} states when two foliated manifolds have the same ``transverse geometry''. Following the proof of Lemma \ref{lem:bisect.diff} one can also motivate definition \ref{def:defMEfol} by means of the automorphism group, using its relation with the local automorphisms $\phi$ from $M$ to $N$ such that $\phi^{-1}\CF_N=\CF_M$.
	
	\begin{lem}\label{lem:HausMEfolER}
		Hausdorff Morita equivalence is an equivalence relation on foliated manifolds.
	\end{lem}
	\begin{proof}
		A foliated manifold $(M,\cF_M)$ is equivalent to itself, by means of $(M, Id_M, Id_M)$, therefore this relation is reflexive. It is  clearly symmetric. We now prove that it is transitive.
		
		For the transitivity, let $(M,\CF_M)$, $(N,\CF_N)$ and $(S,\CF_S)$ be foliated manifold such that $M\simeq_{ME} N$ and $N\simeq_{ME} S$. There exists manifolds $P_1,P_2$  and surjective submersions with connected fibers $\pi_M, \pi_N^1, \pi_N^2, \pi_S$ as in this diagram, inducing the above {Hausdorff } Morita equivalences.
		
		\[\begin{tikzcd}
			& & P_1\times_N P_2 \arrow[dl,swap, "\text{Pr}_1"] \arrow[dr, "\text{Pr}_2"] & &\\
			& P_1 \arrow[dl, swap, "\pi_M"] \arrow[dr,"\pi_N^1"]& & P_2 \arrow[dl, swap, "\pi_N^2"] \arrow[dr,"\pi_S"] &\\
			(M,\CF_M)& &(N,\CF_N) & & (S,\CF_S)
		\end{tikzcd}\]
		Take $P:= P_1\times_N P_2$, and denote the projections onto the factors by $\text{Pr}_1$  and $\text{Pr}_2$. 
		{The commutativity of the diagram implies that} 
		\[(\pi_M\circ \text{Pr}_1)^{-1}(\CF_M)=\text{Pr}_1^{-1}((\pi_N^1)^{-1}(\CF_N))={\text{Pr}_2^{-1}((\pi_N^2)^{-1}(\CF_N))}=(\pi_S\circ \text{Pr}_2)^{-1}(\CF_S).\]
		{The maps $\pi_M\circ \text{Pr}_1:P\fto M$ and $\pi_S\circ \text{Pr}_2:P\fto S$ are clearly surjective submersions. We    prove below that they have connected fibres, allowing to conclude that  $M\simeq_{ME} S$ via $P=P_1\times_N P_2$ and thus finishing the proof.} 
		
		{For any $m\in M$ we now show that $(\pi_M\circ \text{Pr}_1)^{-1}(m)$ is connected. Notice that the map}
		
		\[\text{Pr}_1:(\pi_M\circ \text{Pr}_1)^{-1}(m)\fto \pi_M^{-1}(m),\]
		is a surjective submersion with connected fibres, because  {the fibre over} $p_1\in \pi_M^{-1}(m)$ is \[\text{Pr}_1^{-1}(p_1)=\{(p_1,p_2)\st p_2\in (\pi_N^2)^{-1}(\pi_N^1(p_1))\}\cong (\pi_N^2)^{-1}(\pi_N^1(p_1)),\]
		which is connected. Then using the connectedness of $\pi_M^{-1}(m)$ {and lemma \ref{lem:subcon}} we get that $(\pi_M\circ {Pr_1})^{-1}(m)$ is connected. The same argument shows that $\pi_S\circ \text{Pr}_2:P\fto S$ also  has connected fibres.
	\end{proof}
	
	Note that Definition \ref{def:defMEfol} does not require the property of bisubmersions that $\pi_M^{-1}\CF_M=\pi_N^{-1}\CF_N$ equals $\ker_c(d\pi_M)+\ker_c(d\pi_N)$.
	The main reason for not including this property is that it is not needed to prove any of the features that we want Hausdorff Morita equivalence to have. 
	
	Another reason is that, given a  singular foliation $(M,\cF)$, there may not exist any global bisubmersion between $(M,\cF)$ and itself. Indeed, assume such a global bisubmersion $(U,\bt,\bs)$ exists.
	For every $p\in U$ we have $\cF/I_{\bs(p)}\cF\cong \bs^{-1}\cF/I_p(\bs^{-1}\cF)$.
	The dimension  of the latter is $\le 2 (dim(U)-dim(M))$, since the map:
	\begin{align*}
		\left(\ker_c(d\bs)/I_p \ker_c(d\bs)\right)\oplus \left(\ker_c(d\bt)/I_p \ker_c(d\bt)\right) &\to  \bs^{-1}\cF/I_p(\bs^{-1}\cF),\\
		[X]+[Y] & \mapsto  [X+Y],
	\end{align*} 
	is surjective. Combining these two facts we see that, at every $x\in M$, the dimension of $\cF/{I_x}\cF$ is bounded above by $2 (dim(U)-dim(M))$. However there exist singular foliations (on non-compact manifolds) for which this dimension is unbounded. A  concrete example is displayed in \cite[lemma 1.3]{AZ1}.
	
	\section{First invariants}\label{subsec:firstinv}
	
	The rough idea of definition \ref{def:defMEfol} is that two singular foliations are Hausdorff Morita equivalent if they have the same leaf space and if the ``transverse geometry'' at corresponding leaves is the same. In this chapter we will show what this means locally, but there is also a global and more abstract implication that we will show in theorem \ref{thm:MEfolgroids}.
	
	\begin{prop}\label{prop:corrleaves}
		Let  $(M,\cF_M)$ and $(N,\cF_N)$ be  singular foliations which are Hausdorff Morita equivalent, by means of $(U,\pi_M,\pi_N)$.   Then:
		\begin{enumerate}
			\item[(i)] There is a homeomorphism between the {leaf space}
			of $(M,\cF_M)$ and the {leaf space} of $(N,\cF_N)$: it maps the leaf through $x\in M$   to {the leaf of $\cF_N$ containing}    $ {\pi_N({\pi_M}^{-1}(x))}$. 
			{It preserves the codimension of leaves and the property of being an embedded leaf.}
			
			\item[(ii)] 
			{Let $x\in M$ and $y\in N$ be a points lying in corresponding leaves.}
			Choose slices $S_x$ at $x$ and $S_y$ at $y$. Then the foliated manifolds $(S_x,{\iota_{S_x}^{-1}\cF_M})$ and $(S_y,{\iota_{S_y}^{-1}\cF_N})$ are diffeomorphic.
			\item[(iii)] Let $x\in M$ and $y\in N$ be points lying in corresponding leaves. Then the isotropy Lie algebras $\g^{\cF_M}_x$ and $\g^{\cF_N}_y$ are isomorphic.
		\end{enumerate}
	\end{prop}
	\begin{proof}
		(i) {For every leaf $L_M$ on $M$, the preimage $\pi_M^{-1}(L_M)$ is a leaf of 
			$\pi_M^{-1}\CF_M=\pi_N^{-1}\CF_N$. Hence it equals $\pi_N^{-1}(L_N)$ for a unique leaf $L_N$ on $N$, which has the same codimension as $L_M$.} {Since $\pi_M$ and $\pi_N$ are continuous open maps, this assignment is a homeomorphism}. {If $L_M$ is an embedded leaf, then a chart on $M$ adapted to $L_M$ induces a  chart on $U$ adapted to $\pi_M^{-1}(L_M)$, and vice versa.} 
		
		(ii) By Definition \ref{def:defMEfol}, it suffices to work with the submersion $\pi_M\colon U\to M$. Take $u\in U$ and let $S_u$ be a transversal for $\pi_M^{-1}\CF_M$ at $u$. Then $S_x:=\pi_M(S_u)$ is a transversal for $\CF_M$ at $x:=\pi_M(u)$. Counting dimensions, {and shrinking $S_u$ if necessary,} we see that ${\pi_M|_{S_u}}\colon S_u\fto S_x$ is a diffeomorphism. The commutativity of the diagram  
		
		\[\begin{tikzcd}
			S_u \arrow[r,"\iota_{S_u}"] \arrow[d,swap, "\pi_M|_{S_u}", "\wr"']& U \arrow[d, "\pi_M"]  \\
			S_x  \arrow[r, "\iota_{S_x}"] & M\\
		\end{tikzcd}\]
		{implies that the singular foliations ${\iota_{S_x}^{-1}\cF_M}$ and $ {\iota_{S_u}^{-1}\cF_U}$ correspond under the above diffeomorphism, where $\cF_U:=\pi_M^{-1}\CF_M$.}
		
		(iii) follows from (ii), since the isotropy Lie algebra at a point coincides with the isotropy Lie algebra of the transverse foliation at that point, see \cite[Rem. 2.6]{AZ1}.
	\end{proof}
	
	\begin{ex}
		a) Given distinct  integers $k,l>0$, the singular foliation on the real line generated by the vector field $x^k \frac{\partial}{\partial x}$ and the one generated by $x^l \frac{\partial}{\partial x}$ lie in different Morita equivalence classes. This can be seen noticing that there is no diffeomorphism between neighbourhoods $S^k$ and $S^l$ of the origin
		that maps $x^k \frac{\partial}{\partial x}|_{S^k}$ to the product of  $x^l \frac{\partial}{\partial x}|_{S^l}$ with a no-where vanishing function, and then applying proposition \ref{prop:corrleaves} ii).
		Notice however that the underlying partitions into leaves and the isotropy {Lie algebras} are the same.
		
		b) Consider the singular foliations on $\RR^2$ given by the linear actions of $GL(2,\RR)$
		and  $SL(2,\RR)$. They have the same leaves, namely the origin and its complement. The isotropy Lie algebras at the origin are the Lie algebras of $GL(2,\RR)$
		and  $SL(2,\RR)$ respectively, hence by proposition \ref{prop:corrleaves} iii) these two singular foliations are not Hausdorff Morita equivalent.
	\end{ex}
	
	Recall from lemma \ref{lem:fiber.proy} that the projective foliations are foliations with fibers of constant dimension.
	
	\begin{prop}\label{prop:MEregproj}
		Hausdorff Morita equivalence of singular foliations {preserves the following families of singular foliations:}
		\begin{enumerate}
			\item [(i)] regular foliations
			\item [(ii)] projective foliations
		\end{enumerate}
	\end{prop}
	\begin{proof} {It suffices to show that, given a surjective submersion  
			$\Psi\colon M\to N$   and   a singular foliation $\CF$ on $N$, the pullback foliation $\Psi^{-1}\cF$ is regular (resp. projective) whenever $\CF$ is. For the regular case this is clear, implying (i). For the projective case, by example \ref{ex:L.alg.F} there is} 
		an almost injective Lie algebroid  $A$ associated to $\cF$. The pullback Lie algebroid $\Psi^{-1}A$ is also almost injective, as one checks using Def.  \ref{def:pullbackLA}(definition of the pullback for a Lie Algebroid), and its underlying foliation is $\Psi^{-1}\cF$ by lemma \ref{lem:pullbackalgfol} (which express how is the foliation of the pullback algebroid).\end{proof}
	
	\section{Examples}\label{sec:exam}
	
	{The next three subsections are dedicated to  examples 
		of  {Hausdorff} Morita equivalent singular foliations, starting with the elementary ones.}
	
	\subsubsection{{Elementary examples}}\label{subsec:elemex}

	\begin{ex}[Isomorphic foliations]
		Two foliated manifolds $(M,\CF_M)$ and $(N,\CF_N)$ are said to be isomorphic if there exists a diffeomorphism $\phi:M\fto N$ such that $\phi^{-1}\CF_N=\CF_M$. Two isomorphic foliated manifolds are {Hausdorff} Morita equivalent.
	\end{ex}
	
	\begin{ex}[Full foliations]
		Any two connected manifolds $M$ and $N$ with the full foliations {$\vX_c(M)$ and $\vX_c(N)$} are {Hausdorff} Morita equivalent, using $P=M\times N$ and its projection maps. 
	\end{ex}
	
	\begin{ex}[Zero foliations]
		Two manifolds $M$ and $N$ with the zero foliations are {Hausdorff} Morita equivalent if and only if they are diffeomorphic.
	\end{ex}
	
	\begin{ex}[Simple foliations]
		{A regular foliation $F$ on a manifold $M$ is called simple if the leaf space $M/F$ is a smooth manifold such that the projection map  is a submersion. The foliation $F$ on $M$ and the zero foliation on $M/F$ are Hausdorff Morita equivalent.}
	\end{ex}
	
	The following corollary also counts as an example:
	
	\begin{cor}\label{cor:productfol}
		If $(M_1,\CF_M^1)$ is Morita equivalent to $(M_2,\CF^2_M)$ and $(N_1, \CF_N^1)$ to $(N_2,\CF_N^2)$ then $(M_1\times N_1, \CF_M^1\times \CF_N^1)$ is Morita equivalent to $(M_2\times N_2,\CF_M^2\times \CF_N^2)$.
	\end{cor}
	\begin{proof}
		There are $P$ and $S$ with surjective submersions with connected fibers making the Morita equivalences:
		\[\begin{tikzcd}
			& \arrow[dl,swap,"\pi_1"]P \arrow[dr,"\pi_2"] & & & \arrow[dl,swap,"\sigma_1"]S \arrow[dr,"\sigma_2"] & \\
			M_1& & M_2 & N_1 & & N_2
		\end{tikzcd}\]
		Then using the proposition \ref{prop:product} we get that:
		\[\begin{tikzcd}
			& \arrow[dl,swap,"\pi_1\times \sigma_1"] P\times S \arrow[dr,"\pi_2\times \sigma_2"] & \\
			M_1\times N_1 & & M_2\times N_2 
		\end{tikzcd}\]
		is a Morita equivalence between $(M_1\times N_1, \CF_M^1\times \CF_N^1)$ and $(M_2\times N_2,\CF_M^2\times \CF_N^2)$.
	\end{proof}
	
	\subsubsection{{Examples obtained by pushing forward foliations}} 
	\label{subsection:expush}

	{Note that any bisubmersion $(V,\bt,\bs)$ 
		between foliated manifolds $(M,\CF_M)$ and $(N,\CF_N)$ (see Def. \ref{def:bisub}), when $\bs$ and $\bt$ have connected fibres, is a Morita equivalence 
		between $\cF_M|_{\bs(V)}$ and $\cF_N|_{\bt(V)}$. Here we construct Morita equivalences of this kind, starting from simple data for which concrete examples can be found quite easily.}
	
	\begin{rem}\label{rem:easyME}
		Given two {surjective} submersions $\bs \colon U \to M$ and $\bt \colon U \to N$  with connected fibres, let $\cF_M$ be a singular foliation on $M$ such that $\bs^{-1}(\cF_M)\supset  \Gamma_c(\ker d\bt)$. Then,  by lemma \ref{lem:submfol}, there is a unique singular foliation $\cF_N$ on $N$  such that  $\bs^{-1}(\cF_M)={\bt}^{-1}(\cF_N)$. {In particular,} $(M,\CF_M)\simeq_{ME} (N,\CF_N)$. {In other words,} we can ``transport'' the foliation ${\cF_M}$ on $M$ to a Hausdorff Morita equivalent foliation on $N$.
	\end{rem}
	
	\begin{cor}\label{cor:pushdowntwo}
		Given two submersions $\bs \colon U \to M$ and $\bt \colon U \to N$  with connected fibres, assume that 
		\begin{equation*}
			[\Gamma_c(\ker d\bs), \Gamma_c(\ker d\bt)]\subset \Gamma_c(\ker d\bs) + \Gamma_c(\ker d\bt)=:\cF_U.
		\end{equation*}
		Then there are  unique foliations $\cF_M$ and $\cF_N$ on $M$ and $N$ respectively such that  $\bs^{-1}(\cF_M)={\bt}^{-1}(\cF_N) = \cF_U$. In particular,  $(M,\CF_M)\simeq_{ME} (N,\CF_N)$.
	\end{cor}
	\begin{proof}
		{Apply lemma \ref{lem:submfol} for the foliation $\cF_U$ twice: to the map $\bs$ and to the map $\bt$.}
	\end{proof}
	
	An interesting special case of Cor. \ref{cor:pushdowntwo} is when the submersions arise from Lie group actions.
	
	\begin{cor}\label{cor:action}
		Consider two \emph{connected} Lie groups $G_1,G_2$ acting\footnote{The actions can be both right actions, both left actions, or one right and one left action.}
		freely and properly on a manifold $P$ with commuting actions. Then the following singular foliations are {Hausdorff} Morita equivalent:
		\begin{enumerate}
			\item the singular foliation on $P/G_1$  given by the induced $G_2$ action, 
			\item the singular foliation on $P/G_2$   given by the induced  $G_1$ action.\end{enumerate}
	\end{cor}
	\begin{proof} Since the {infinitesimal} generators of the $G_1$-action commute with those of the $G_2$-action, the hypotheses of Cor. \ref{cor:pushdowntwo} are satisfied. It is straightforward that the singular foliation on $P/G_1$   induced by the $G_2$ action pulls back to the singular foliation on $P$   induced by the $G_1\times G_2$ action, {and similarly for the singular foliation on $P/G_2$.}
	\end{proof}
	Notice that the {Hausdorff} Morita equivalence is realised by $P$ with the projection map.
	When the singular foliation on $P$ given by the $G_1\times G_2$ action is  regular, the induced foliations on $P/G_1$ and $P/G_2$ are also regular.
	
	Now we specialize {Cor. \ref{cor:action}} even further, taking $P$ to be a Lie group   and $G_1$,$G_2$ to be two connected closed subgroups acting  respectively by left and right multiplication. We  present  {two} examples.
	\begin{ex}
		Let ${P}=U(2)$, $G_1=SU(2)$, and let $G_2$ consist of the diagonal matrices in $SU(2)$
		(hence $G_2\cong U(1)$). 
		{The quotient of the left action of $SU(2)$ on $U(2)$ is $SU(2)\backslash U(2)\cong S^1$, since the homomorphism $det\colon U(2)\to S^1$ has kernel $SU(2)$.}  The   action of $G_2$ on $U(2)$ by right multiplication descends to the trivial action on $S^1$.
		Hence on $S^1$ we obtain the (regular) foliation by points.
		By Cor. \ref{cor:action}, it is { Hausdorff} Morita equivalent to the (regular) foliation on $U(2)/G_2$ by orbits of the left $SU(2)$-action.
	\end{ex}
	
	\begin{ex}
		{We apply Cor. \ref{cor:action} to actions of the Lie groups $SO(2n)$ and $U(n)$ on   $P=SO(2n+1)$.}
		We can include $SO(2n)$ in $SO(2n+1)$ {as matrices with  $1$ in the bottom right corner}. {Left multiplication induces a left action of $SO(2n)$ on $SO(2n+1)$ with quotient manifold $SO(2n)\backslash SO(2n+1) \cong S^{2n}$.}
		
		On the other hand, we can include $U(n)$ in $SO(2n+1)$ {as the unitary matrices with  $1$ in the bottom right corner.}
		{Right multiplication induces a right action of $U(n)$ on $SO(2n+1)$. The quotient manifold is $$SO(2n+1)/U(n) \cong J(2n+2),$$ where $J(2n+2)$ denotes the set of complex structures in $\RR^{2n+2}$ preserving the canonical inner product and orientation.} {In fact,   there is a diffeomorphism $SO(2n+1)/U(n) \cong SO(2n+2)/U(n+1)$, 
			induced by the {transitive} action of  $SO(2n+1)$ on $SO(2n+2)/U(n+1)$ inherited from the  left multiplication, which has isotropy group $U(n)$. In turn, $SO(2n+2)/U(n+1)\cong  J(2n+2)$ by considering the 
			action of $SO(2n+2)$ on $J(2n+2)$ by pullbacks, which has isotropy group $U(n+1)$.}
		
		Hence by Corollary. \ref{cor:action} the following singular foliations are { Hausdorff} Morita equivalent: 
		\begin{itemize}
			\item the singular foliation on $J(2n+2)$  induced by the action of \\ ${SO(2n)\subset SO(2n+2)}$  via pullbacks, 
			\item the singular foliation on $S^{2n}$  induced by the  action by {right   matrix multiplication of $U(n)\subset SO(2n+1)$.}
		\end{itemize}
		{Note that the South pole and North pole of $S^{2n}$ are the only fixed points of the action of $U(n)$, therefore on $S^{2n}$ we have a genuinely singular foliation. 
			As a consequence of Morita equivalence, the foliation on $J(2n+2)$ is also non-regular.} 
	\end{ex}

	\subsubsection{{Examples from Morita equivalent Poisson manifolds}}
	
	A Poisson manifold $(M,\Pi_M)$ gives rise to a singular foliation $\cF_{\Pi_M}$ that consists of $C^{\infty}_c(M)$-linear combinations of Hamiltonian vector fields on $M$. In particular, the leaves of $\cF_{\Pi_M}$ are exactly the symplectic leaves of the Poisson structure.

	\begin{ex}\label{ex:fulldualpair}
		Let $(M,\Pi_M)$ and $(N,\Pi_N)$  be Poisson manifolds. A \emph{full dual pair} \cite[\S 8]{We} consists of a symplectic manifold $(U,\omega)$ with surjective submersions $\bs\colon U\to M$ and $\bt\colon U\to N$ which are Poisson and anti-Poisson maps respectively, and such that $ker(d_u\bs)$ and $ker(d_u\bt)$ are symplectic orthogonal subspaces of $T_uU$ for all $u\in U$. 
		Notice that $\Gamma(\ker d\bs)$ is generated  by $\{X_{\bt^*g}:g\in C^{\infty}(N)\}$ as a $C^{\infty}(U)$-module, while $\Gamma(\ker d\bt)$ is generated by $\{X_{\bs^*g}:g\in C^{\infty}(M)\}$. Here we denote by $X_F$ the Hamiltonian vector field of the function $F$.
		
		A full dual pair with connected fibres is a global bisubmersion with connected fibres  for the foliations  $\cF_{\Pi_M}$ and  $\cF_{\Pi_N}$ (see Def. \ref{def:bisub}).
		Indeed, since $\bs$ is a Poisson map, for any Hamiltonian vector field $X_{g}$, an $\bs$-lift is given by $X_{\bs^*g}$, hence 
		$$\bs^{-1}(\cF_{\Pi_M})=\left<\{X_{\bs^*g}:g\in C^{\infty}(M)\}+\Gamma(\ker d\bs)\right>_
		{C^{\infty}_c(U)}=\Gamma_c(\ker d\bt)+\Gamma_c(\ker d\bs),$$
		and the analogous equation holds for $\bt$. As a consequence, $(M,\cF_{\Pi_M})\simeq_{ME} (N,\cF_{\Pi_N})$.
	\end{ex}

	\begin{cor}\label{cor:PoisME}
		If two Poisson manifolds are Morita equivalent
		\cite{xuME} then their singular foliations   are Hausdorff Morita equivalent.
	\end{cor}
	\begin{proof}
		Two Poisson manifolds are Morita equivalent if they are related by a complete full dual pair with simply connected fibers. Hence the statement follows from Ex. 
		\ref{ex:fulldualpair}.
	\end{proof}

	\cleardoublepage
	

	\chapter{Lie Groupoids and Lie Algebroids}\label{ch:3}
	
	In this chapter we introduce the notions of Lie groupoids, Lie algebroids and their relation with singular foliations. Most of the results can be found in \cite{CrFeLie}, \cite{MK2}  and \cite{MRIntr}, however we have some original results.
	
	\section{Lie Groupoids}\label{sec:Lie.Grpd}
	
	This section is an introduction to Lie groupoids and can be easily skipped by those familiar with the subject. The material presented here can be also found in \cite{CrFeLie}, \cite{MK2}, \cite{MRIntr} and \cite{Kwang}.
	
	As a motivation, we come back to section \ref{sec:bisub} where we introduced a finite dimensional object, namely the bisubmersions, to study an infinite dimensional group, the group of automorphisms preserving a foliation. The same motivation can be used to introduce Lie groupoids. Indeed, the bisections of a Lie groupoid $\CG\soutar M$ give rise to a subgroup of $\mathrm{Diff}(M)$. Depending on $\CG$, this subgroup preserves certain structures on $M$. We will give more details on this idea after definition \ref{defi:bisect}.
	
	\begin{definition} A {\bf topological groupoid} consists of:
		\begin{itemize}
			\item  Two topological spaces $\CG$ and $M$. Here $\CG$ is called the set of arrows and $M$ the set of objects.
			\item  Two surjective continuous maps: $\bs\colon \CG\fto M$ called the \textbf{source} map and $\bt\colon \CG\fto M$ called the \textbf{target} map.
			\item A  continuous map $\circ \colon \CG {}_\bs \!\times_\bt \CG\fto \CG$, called \textbf{composition}.
			\item A continuous map $e\colon M\fto \CG$, called \textbf{identity}.
			\item A continuous map $(-)^{-1}\colon \CG\fto \CG$, called \textbf{inverse}.
		\end{itemize}
		such that:
		\begin{itemize}
			\item \textbf{$\circ$ is associative:} for all $(g,h,s)\in \CG {}_\bs \!\times_\bt \CG {}_\bs \!\times_\bt \CG$ we have $(g\circ h)\circ s=g\circ (h\circ s)$.
			\item \textbf{$e$ is indeed an identity:} for all $x\in M$ we get $\bt(e_x)=\bs(e_x)=x$ and for all $g\in \CG$ we have $e_{\bt(g)}\circ g= g\circ e_{\bs(g)}=g$.
			\item \textbf{$(-)^{-1}$ is indeed an inverse:} for all $g\in \CG$ we have ($g^{-1})^{-1}=g$, $\bt(g)=\bs(g^{-1})$, $\bs(g)=\bt(g^{-1})$ with $g^{-1}\circ g=e_{\bs(g)}$ and $g\circ g^{-1}=e_{\bt(g)}$.
		\end{itemize}
		A topological groupoid is denoted as $\CG\soutar M$. It is {\bf source connected} if the source fibers are path connected (this also implies that the target fibers are path connected). It is {\bf open} if the source and target maps are open maps.
	\end{definition}

	\begin{defi} A {\bf Lie groupoid} is a topological groupoid $\CG\soutar M$ such that $\CG$ and $M$ are smooth manifolds, the source and target maps are surjective submersions and the composition, identity and inverse maps are smooth.	
	\end{defi}
	
	In particular, any Lie groupoid is an open topological groupoid, and any Lie groupoid with connected source fibers is a source connected groupoid (for a smooth manifold being connected and path connected are equivalent).
	
	\begin{rem} One of the reasons to ask $\bs$ to be a submersion is that $\CG {}_\bs \!\times_\bt \CG$ must have a canonical structure of smooth manifold to make sense of the smoothness of the composition map. Also note that the map $\bs$ is a submersion if and only if $\bt=\bs\circ(-)^{-1}$ is a submersion. 
	\end{rem}
	
	\begin{con} It is customary to allow the space of arrows $\CG$ of a Lie groupoid to be a non-Hausdorff manifold. Nevertheless it is required that each source fiber (and therefore each target fiber) must be Hausdorff.
	\end{con}
	
	\begin{rem}\label{rem:grpds} A groupoid can be seen as a small category, where every arrow is invertible.
	\end{rem}
	
	\subsection*{Examples of Lie groupoids}
	
	We start with three easy examples of Lie groupoids:
	
	\begin{ex}\label{ex:pair.gpd}({\bf Pair groupoid}) Take any manifold $M$ as the set of objects and $\CG:=M\times M$ as the space of arrows. Then the following data defines a Lie groupoid, called the pair groupoid:
		\begin{itemize}
			\item $\bs\colon \CG\fto M; (x,y)\mapsto y$ and $\bt\colon \CG\fto M;(x,y)\mapsto x$.
			\item $\circ \colon \CG {}_\bs \!\times_\bt \CG\fto \CG; (x,y)\circ(y,z)\mapsto (x,z)$.
			\item $e\colon M\fto \CG; x\mapsto (x,x)$.
			\item $(-)^{-1}\colon \CG\fto \CG; (x,y)\mapsto (y,x)$.
		\end{itemize}	
	\end{ex}
	
	\begin{ex}\label{ex:loc.hol.grpd} Let $M,S$ be manifolds (in particular take $M:=\RR^k$ and $S:=\RR^{n-k}$). Take $P=M\times S$  as the set of points and $\CG:=M\times M\times S$ the space of arrows. Then the following data defines a Lie groupoid:
		\begin{itemize}
			\item $\bs\colon \CG\fto P; (x,y,c)\mapsto (y,c)$ and $\bt\colon \CG\fto P;(x,y,c)\mapsto (x,c)$.
			\item $\circ \colon \CG {}_\bs \!\times_\bt \CG\fto \CG; (x,y,c)\circ(y,z,c)\mapsto (x,z,c)$.
			\item $e\colon P\fto \CG; (x,c)\mapsto (x,x,c)$.
			\item $(-)^{-1}\colon \CG\fto \CG; (x,y,c)\mapsto (y,x,c)$.
		\end{itemize}	
	\end{ex}
	
	Example \ref{ex:loc.hol.grpd} is the local picture of the holonomy groupoid of a regular foliation which will be given in section \ref{sec:hol.grpd}.
	
	A more general example is the following:
	\begin{ex}\label{ex:fib.grpd} Let $\pi\colon P\fto S$ be a surjective submersion. Take $P$ as the set of points and the fiber product $\CG:=P\times_S P$ as the space of arrows. Then the following data defines a Lie groupoid:
		\begin{itemize}
			\item $\bs\colon \CG\fto P; (x,y)\mapsto y$ and $\bt\colon \CG\fto P;(x,y)\mapsto x$.
			\item $\circ \colon \CG {}_\bs \!\times_\bt \CG\fto \CG; (x,y)\circ(y,z)\mapsto (x,z)$.
			\item $e\colon P\fto \CG; x\mapsto (x,x)$.
			\item $(-)^{-1}\colon \CG\fto \CG; (x,y)\mapsto (y,x)$.
		\end{itemize}
	\end{ex}
	Note that, taking $P=M\times S$ and $\pi\colon P\fto S$ the natural  projection, example \ref{ex:fib.grpd} reduces to example \ref{ex:loc.hol.grpd}.
	
	Now we will show an example that illustrates how Lie groupoids are a generalization of Lie groups. In fact, Lie groupoids can be seen as Lie groups where the multiplication is defined only for composable elements. 
	
	\begin{ex}\label{ex:G.group} A Lie group $G$ is a Lie groupoid over a point $G\soutar \{*\}$. Here source and target maps are canonical. The composition, identity and inverse maps correspond to the usual multiplication, identity and inverse maps on $G$.
	\end{ex}
	
	Now we continue with a different characterization of Lie groupoids. The following two examples illustrate Lie groupoids as generalized manifolds:
	
	\begin{ex}\label{ex:Gid}({\bf Identity groupoid}) A manifold $M$ can be seen as a Lie groupoid $M\soutar M$ with source and target given by the identity map. The multiplication, identity and inverse maps are given in a canonical way.
	\end{ex}
	
	Recall that manifolds (of dimension $n\in \NN$) can be defined as a collection of open balls on $\RR^n$ that are ``glued'' together by a family of maps satisfying a cocycle condition. With this idea in mind, a manifold $M$ can also be seen as the following groupoid:
	
	\begin{ex}\label{ex:cech.grpd}({\bf Čech Groupoid}) Let $M$ be a manifold and $\{U_i\}_{i\in I}$ a family of charts covering $M$ (each $U_i$ can be seen as an open ball in $\RR^n$). Take the disjoint union $\widehat{M}:=\sqcup_{i\in I} U_i$ as the base of the groupoid. There is a canonical map $\pi\colon \widehat{M}\fto M$.
		
		The Čech Groupoid is the groupoid of example \ref{ex:fib.grpd} for the map $\pi\colon \widehat{M}\fto M$. Note that the arrows are given by:
		$$\CG:=\widehat{M}\times_\pi \widehat{M}=\sqcup_{i,j\in I} U_i\times_M U_j.$$
		
		On $\widehat{M}$, one can define an equivalence relation: $x\sim y$ if and only if there exists $g\in\CG$ such that $\bt(g)=x$ and $\bs(g)=y$. Then it is easy to see that $\widehat{M}/\CG:=(\widehat{M}/\sim) \simeq M$.
		
	\end{ex}

	The Čech Groupoid and the identity groupoid of $M$ present some similarities. Indeed, they are Morita equivalent as we will see in section \ref{sec:ME.gpd}.
	
	Now we will combine examples \ref{ex:G.group} (as a Lie group) and \ref{ex:Gid} (as a manifold) into a more general view of what a Lie groupoid is about:
	
	\begin{ex}\label{ex:act.gpd}({\bf Action groupoid}) Let $G$ be a Lie group acting on a manifold $P$. Take as set of arrows $\CG:=G\times P$ and the set of points $P$. The following data defines a Lie groupoid called the {\bf action groupoid}:
		\begin{itemize}
			\item $\bs\colon \CG\fto P; (g,x)\mapsto x$ and $\bt\colon \CG\fto P;(g,x)\mapsto gx$.
			\item $\circ \colon \CG {}_\bs \!\times_\bt \CG\fto \CG; (h,gx)\circ(g,x)\mapsto (hg,x)$.
			\item $e\colon P\fto \CG; x\mapsto (Id_G,x)$.
			\item $(-)^{-1}\colon \CG\fto \CG; (g,x)\mapsto (g^{-1},gx)$.
		\end{itemize}	
	\end{ex}
	
	If $G$ acts freely and properly on $P$, then the quotient space $M:= P/G$ is a manifold. Note that in this case the action groupoid $G\times P$ is almost the same as the groupoid in example \ref{ex:fib.grpd} for the quotient map $\pi\colon P\fto M$. Indeed they are isomorphic, as we will show after definition \ref{def:grpd.mor}.
	
	\subsection*{Bisections and group of diffeomorphisms}
	
	As is the case for bisubmersions, see section \ref{sec:bisub}, groupoids are sometimes used to study subgroups of $\mathrm{Diff}(M)$. This characterization uses the notion of bisections.
	
	\begin{defi}\label{defi:bisect} Given a groupoid $\CG\soutar M$, a \textbf{bisection} consists of an $\bs$ section $\sigma:M\fto \CG$ that is transverse to the fibers of $\bt$. 
	\end{defi}
	
	Any bisection $\sigma$ defines a diffeomorphism on $M$ given by $\widehat{\sigma}:=\bt\circ\sigma\colon M\fto M$. 
	
	Moreover, the multiplication in $\CG$ defines a multiplication on bisections given by the formula  $$(\sigma\circ \phi) (x)= \sigma(\widehat{\phi}(x))\circ \phi(x).$$
	
	Note that $\widehat{\sigma\circ \phi}=\widehat{\sigma}\circ\widehat{\phi}$.
	
	On the identity groupoid $M\soutar M$, there is only one bisection, the identity map, this bisection carries the identity diffeomorphism. In general, for any Lie groupoid $\CG\soutar M$, the identity section $e:M\fto \CG$ is a bisection and carries the identity diffeomorphism. 
	
	The bisections of a Lie groupoid $\CG\soutar M$ have a group structure with multiplication given by the composition $\circ$ and the identity given by the identity section $e:M\fto \CG$. This group can be seen as a subgroup of $\mathrm{Diff}(M)$ under the map $\sigma\mapsto \widehat{\sigma}$.

	\begin{ex} In the sense of definition \ref{defi:bisect}, the group $\mathrm{Diff}(M)$ is given by the bisections of the pair groupoid $M\times M$. 
	\end{ex}
	
	\begin{ex}Note that bisections of the groupoid in example \ref{ex:loc.hol.grpd} are constant on $S$. This means that for all $c\in S$ they send the manifold $M\times\{c\}$ to itself.
		
		In example \ref{ex:fib.grpd}, the group of bisections preserves strongly\footnote{The diffeomorphism sends each fiber to itself} the fibers of the map $\pi$.
	\end{ex}
	
	\begin{ex} In the example \ref{ex:cech.grpd} (the Čech Groupoid), the local bisections give a family of maps satisfying the cocycle condition and they glue together the components of $\widehat{M}$, giving the manifold $M$ as result.
	\end{ex}

	\begin{ex} In example \ref{ex:act.gpd} about the action groupoid, for each $g\in G$ the map $\sigma_{g}\colon M\fto G\times M;p\mapsto (g,p)$ is a bisection that carries the diffeomorphism of multiplying by $g$ on $M$.
	\end{ex}
	
	\subsection*{Intrinsic structures of Lie groupoids:}
	
	To continue our illustration of Lie groupoids we will point out important features on their structure.
	
	For any Lie groupoid $\CG\soutar M$, the inverse map is a diffeomorphism on $\CG$. Therefore, for all $x\in M$, we get the following diffeomorphism between the fibers: 
	$$(-)^{-1}|_{\bs^{-1}(x)}\colon \bs^{-1}(x)\fto \bt^{-1}(x).$$
	
	Moreover, for every $g\in \CG$ there are two diffeomorphisms:
	\begin{itemize}
		\item The {\bf left multiplication}: 
		\begin{equation}\label{eq:left.mul}
			L_g\colon \bt^{-1}(\bs(g))\fto \bt^{-1}(\bt(g));h\mapsto g\circ h.
		\end{equation}
		\item The {\bf right multiplication}:
		\begin{equation}\label{eq:right.mul}
			R_g\colon \bs^{-1}(\bt(g))\fto \bs^{-1}(\bs(g));h\mapsto h\circ g.
		\end{equation}
	\end{itemize}	
	
	For every point $x\in M$ the Lie groupoid structure gives us:
	\begin{itemize}
		\item A Lie group, called the {\bf isotropy Lie group}, given by the set $$\CG_x:=\bs^{-1}(x)\cap \bt^{-1}(x).$$ 
		\item An immersed submanifold of $M$, called the {\bf orbit through} $x$, given by $O_x:=\bt(\bs^{-1}(x))$ (an equivalent description is $\bs(\bt^{-1}(x))$).
	\end{itemize}
	
	A Lie groupoid also has an associated topological space called the {\bf orbit space}, which is given by the set $M/\CG :=\{O_x \st x\in M\}$ equipped with the quotient topology via the map $Q\colon M\fto M/\CG;x\mapsto O_x$.
	
	\begin{ex} Let $G$ be a Lie group acting on a manifold $M$. For the action groupoid $G\times M$, the notions of left/right multiplication given in equations \ref{eq:left.mul} and \ref{eq:right.mul} coincide with the respective notions on $G$. Moreover, the isotropy groups, the orbits and the orbit space coincide with the respective notions for the $G$-action on $M$.
	\end{ex}
	
	\begin{defi} Let $\CG\soutar M$ be a Lie groupoid. A submanifold $S\subset M$ is {\bf saturated} if $\bs^{-1}(S)=\bt^{-1}(S)$.
	\end{defi}
	
	\begin{lem} Let $\CG\soutar M$ be a Lie groupoid, and $x\in M$. Then $O_x$ is a saturated submanifold of $M$. Moreover $S$ is saturated if and only if for all $s\in S$ we have $O_s\subset S$.	
	\end{lem}
	
	Let $\CG\soutar M$ be a Lie groupoid, and $S$ a { saturated} submanifold or an { open} subset of $M$.  Then $\CG_S:=\bs^{-1}(S)\cap \bt^{-1}(S)$ is a submanifold of $\CG$ with a canonical Lie groupoid structure with base $S$.
	
	\begin{defi} The Lie groupoid $\CG_S\soutar S$ is called the restriction groupoid of $\CG$ to $S$.
	\end{defi} 
	
	The following well known theorem helps us understand the structure of a Lie groupoid, its orbits and isotropy Lie groups.
	
	\begin{prop} Let $\CG\soutar M$ be a Lie groupoid, and $x\in M$. Then $\CG_x$ is a Lie group and a submanifold of $\CG$.
		
		The Lie group $\CG_x$ acts canonically on $\bs^{-1}(x)$ by right multiplication and $\bt\colon \bs^{-1}(x)\fto O_x$ is a right-principal $\CG_x$-bundle.
	\end{prop}
	
	For all $y\in O_x$, there exists $g\in \CG$ such that $\bs(g)=x$ and $\bt(g)=y$. Then, using the right multiplication by $g$ in $\CG$, we have that $\bs^{-1}(y)$ is diffeomorphic to $\bs^{-1}(x)$. By same reasoning, the Lie group $\CG_y$ is isomorphic to $\CG_x$. 
	
	Moreover, if there exists a global section $\sigma\colon O_x\fto \bs^{-1}(x)$ then there is a canonical $\CG_x$-action on $\CG_{O_x}$ and $(\bt,\bs)\colon \CG_{O_x} \fto O_x\times O_x$ is a principal $\CG_{x}$-bundle.
	
	\subsection*{Morphisms of Groupoids}
	
	As mentioned in remark \ref{rem:grpds}, any Lie groupoid can be thought of as a smooth category where every arrow is invertible, the space of arrows and points are smooth manifolds, and the composition, identity and inverse maps are smooth maps. As such, the definition of morphism for Lie groupoids is a functor:
	
	\begin{defi}\label{def:grpd.mor} Given two Lie groupoids $\CG\soutar M$ and $\CH\soutar N$, a \textbf{morphism of Lie groupoids} is given by two smooth maps $\hat{\pi}\Fr\CG\fto \CH$ and $\pi\colon M\fto N$ such that:
		\begin{itemize}
			\item $\hat{\pi}$ and $\pi$ commute with source and target of $\CG$ and $\CH$.
			\item for all $(g,h)\in \CG {}_{\bs} \!\times_{\bt} \CG$ we have $\hat{\pi}(g\circ h)=\hat{\pi}(g)\circ \hat{\pi}(h)$.
			\item for all $x\in M$ we have $\hat{\pi}(e_x)=e_{\pi(x)}$.
		\end{itemize}
	\end{defi}
	
	\begin{rem}A consequence of definition \ref{def:grpd.mor} is that for all $g\in \CG$, $\hat{\pi}(g^{-1})=(\hat{\pi}(g))^{-1}$.
	\end{rem} 
	
	\begin{ex} Let $\CG\soutar M$ be a Lie groupoid and $S$ a saturated or open submanifold of $M$. Then the inclusion map given by $\iota\colon \CG_S\fto \CG$ and $\iota\colon S\fto M$ is a Lie groupoid morphism.
	\end{ex}
	
	\begin{ex} Let $M$ be a manifold, $\CCI_M$ the identity groupoid $M\soutar M$, $\CG\soutar M$ any Lie groupoid and $\CP_M$ the pair groupoid $M\times M\soutar M$. There exist canonical morphisms of Lie groupoids $\CCI_M\fto \CG\fto \CP_M$ given by:
		\begin{itemize}
			\item $\CCI_M\fto \CG;x\fto e_x$
			\item $\CG\fto \CP_M; g\mapsto (\bt(g),\bs(g))$
		\end{itemize}
	\end{ex}
	
	For any Lie groupoid $\CG\soutar M$, the image of the canonical map $(\bt,\bs)\colon \CG\fto M\times M$ is an equivalence relation on $M$. The quotient of $M$ by this equivalence relation is exactly the orbit space of $\CG$ i.e. $M/\CG$.
	
	\begin{defi}\label{def:grpd.iso} Two groupoids $\CG\soutar M$ and $\CH\soutar N$ are isomorphic if there exists a groupoid morphism $\widehat{\pi}\colon \CG\fto \CH$, $\pi\colon M\fto N$ such that $\widehat{\pi}$ and $\pi$ are diffeomorphisms of manifolds.
	\end{defi}
	
	\begin{ex} $M$ is diffeomorphic to $N$ if and only if their identity groupoids $\CCI_M$ and $\CCI_N$ are isomorphic (equivalently if their pair groupoids $\CP_M$ and $\CP_N$ are isomorphic).
	\end{ex}
	
	\begin{ex} Let $G$ be a Lie group acting freely and properly on $P$, and denote $M:=P/G$. Let $P\times_M P$ be the groupoid given in example \ref{ex:fib.grpd} and $G\times P$ the action groupoid of example \ref{ex:act.gpd}. The map $\widehat{\pi}\colon G\times P\fto P\times_M P; (g,p)\mapsto (gp,p)$ defines an isomorphism of groupoids covering the identity (the base map $\pi$ is the identity on $P$).	
	\end{ex}
	
	\subsection*{Groupoid actions}
	
	Let $\CG\soutar M$ be a Lie groupoid (respectively, a topological open groupoid), let $P$ be a \emph{not necessarily Hausdorff manifold} and $\pi\colon P\fto M$ a surjective submersion (resp. a surjective continuous and open map).
	
	\begin{defi}\label{def:grpd.lact} A {\bf left $\CG$-action} on $P$ is a smooth (resp. continuous) map $\star \colon \CG {}_\bs\!\times_\pi P \fto P$ such that for all $(g,h)\in \CG {}_\bs\!\times_\bt \CG$ and $p\in P$:
		$$ g\star (h \star p)=(g\circ h)\star p,\;\;\;\;\;\;\;\;
		e_{\pi(p)}\star p=p, \;\;\;\;\;\;\;\; \pi(g\star p)=\bt(g).$$
		Such a manifold $P$ with a left $\CG$-action is called a {\bf left $\CG$-module} and $\pi$ is called its {moment map}.
		
		A {\bf right $\CG$-action} over $P$ is given by a smooth (resp. continuous) map $\star \colon P {}_{\pi}\!\times_{\bt} \CG  \fto P$ and satisfying:
		$$ (p\star g) \star h=p \star (g\circ h),\;\;\;\;\;\;\;\;
		p\star e_{\pi(p)}=p, \;\;\;\;\;\;\;\; \pi(p\star g)=\bs(g).$$
	\end{defi}
	
	To illustrate this definition, we will focus on left actions, in which case we have the following diagram:
	\[\begin{tikzcd}
		\CG \arrow[dr,shift left=.2em,"\bs"] \arrow[dr,shift right=.2em,swap, "\bt"]& \curvearrowright & P \arrow[dl, "\pi"]  \\
		& M & \\
	\end{tikzcd}\]
	The first and second equations of definition \ref{def:grpd.lact} say that the action is associative and unital. The third equation says that every $g$ with source $x$ and target $y$ defines a smooth map $g\star(-)\colon \pi^{-1}(x)\fto \pi^{-1}(y)$.
	
	Similarly for right actions, every $g\in \CG$ with source $x$ and target $y$ defines a map $(-)\star g\colon \pi^{-1}(y)\fto \pi^{-1}(x)$.
	
	The following example illustrates the reason why we allow $P$ to be a non-Hausdorff manifold. 
	
	\begin{ex} Let $\CG\soutar M$ be a Lie groupoid. Then $\bt:\CG\fto M$ is a left $\CG$-module by left multiplication on $\CG$.\\
		Moreover $\bs:\CG\fto M$ is a right module and a $(\CG,\CG)$-bimodule \ref{def:bimod.grpd}.
	\end{ex}

	The following lemma illustrates the relation between Lie group actions and Lie groupoid actions:
	\begin{lem} Let $G$ be a Lie group acting on a manifold $M$. Let $P$ be a manifold and $\pi\colon P\fto M$ a submersion. The action groupoid $G\times M$ acts on $\pi$ (from the left) if and only if the group $G$ acts on $P$ (from the left) and the map $\pi$ is equivariant (i.e. $\pi(gp)=g\pi(p)$ for all $g\in G$ and $p\in P$). 
		
		Similarly for right actions.
	\end{lem}
	\begin{proof} $(\Leftarrow)$ The map $\star\colon (G\times M)\times_M P\fto P; (g,x,p)\fto gp$ defines a left Lie groupoid action.
		
		$(\Rightarrow)$ The map $\cdot\colon G\times P\fto P;(g,p)\mapsto(g,\pi(p))\star p$ defines a left Lie group action commuting with the map $\pi$.
	\end{proof}
	
	\begin{defi}\label{def:bimod.grpd} A {\bf $(\CG,\CH)$-bimodule} for Lie groupoids $\CG\soutar M$ and $\CH\soutar N$ is a (not necessarily Hausdorff) manifold $P$ with two commuting actions.
	\end{defi}

	\section{The holonomy groupoid of a regular foliation}\label{sec:hol.grpd}
	
	An important example of a groupoid for this thesis is the holonomy groupoid of a singular foliation, which in the regular case has a nice construction and a canonical smooth structure. Indeed, it is a Lie groupoid.
	
	In this section we introduce such a groupoid for the regular case and describe its structure. Some references about this are \cite{MRIntr} and \cite{Kwang}. Here we present the holonomy groupoid taking into account the definition of singular foliations and bisubmersions.
	
	Something that we will use to define the holonomy groupoid is the local picture of a regular foliation. 
	
	\begin{lem} Let $\CF$ be a regular foliation on $M$ of dimension $k$. Then for all $x\in M$ there exists a chart $(U\subset \RR^n, \phi)$ such that $\phi^{-1}\CF$ is equal to the $\CI_c(U)$ span of $\{\de_{x_1},\cdots ,\de_{x_k}\}$. 
	\end{lem}
	
	Define the foliation $\CF_{loc}$ as the $\CI_c(\RR^n)$-span of $\de_{x_1},\cdots, \de_{x_k}$ in $\RR^n$. The foliated manifold $(\RR^n, \CF_{loc})$ has some special properties. One of them is that its leaves are given by $\RR^k\times\{c\}$ for each $c\in \RR^{n-k}$. Also, for each $x,y\in \RR^n$ lying in the same leaf (i.e. $x=(x_0,c)$ and $y=(y_0,c)$), and $\Sigma_x$, $\Sigma_y$ submanifolds of $\RR^n$ transverse to $\CF_{loc}$ at $x$ and $y$ respectively, there is a unique map from $\Sigma_x$ to $\Sigma_y$ preserving the leaves. Indeed, each leaf crosses $\Sigma_{x}$ once, and similarly for $\Sigma_y$, as shown in figure \ref{fig:lochol},
	
	\begin{figure}[h]
		\centering
		\scalebox{.4}{\includegraphics
			{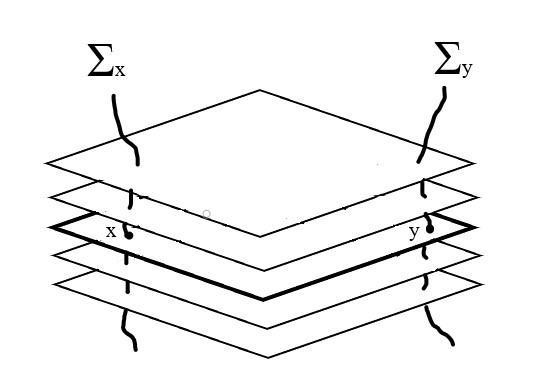}}
		\caption{$\CF_{loc}$ and transversals}
		\label{fig:lochol}
	\end{figure}
	
	In general, for any regular foliatiated manifold $(M,\CF)$ and two points $x,y\in M$ in the same leaf, there is no canonical map between transversals preserving the leaf. However, choosing a path $\gamma\colon [0,1]\fto M$ from $x$ to $y$, for each point $\gamma(t)$ there will be a chart $U_{\gamma(t)}$ where $\CF$ is isomorphic to $\CF_{loc}$. Then using compactness of $[0,1]$ there must exist finitely many $t_k\in[0,1]$ such that the family of charts $\{U_{\gamma(t_k)}\}_k$ cover $\gamma$. Now choose a transversal $\Sigma_{t_k}$ to $\CF$ at each $\gamma(t_k)$. Because of the local picture, there will be canonical maps from $\Sigma_x\fto \Sigma_{t_1}\fto \cdots \fto \Sigma_y$, as shown in figure \ref{fig:fulhol}.
	
	\begin{figure}[h]
		\centering
		\scalebox{.3}{\includegraphics{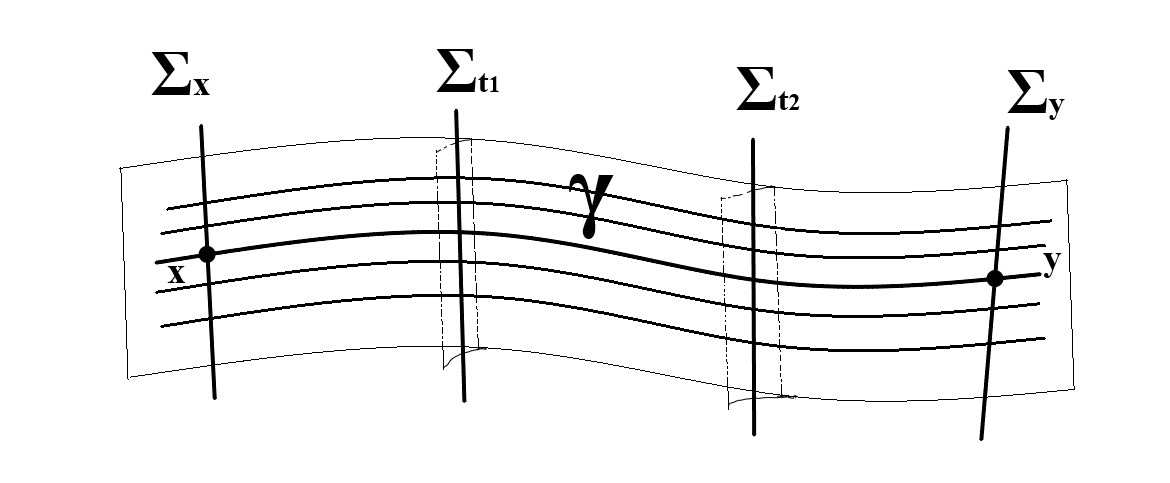}}
		\caption{The map between transversals using a path $\gamma$}
		\label{fig:fulhol}
	\end{figure}
	
	\begin{lem}\label{lem:Holo} For any regular foliatiated manifold $(M,\CF)$, $x,y\in M$ in the same leaf and a path $\gamma$ inside the leaf from $x$ to $y$, the composition map $\Sigma_x\fto \Sigma_{t_1}\fto \cdots \fto \Sigma_y$ doesn't depend on the choice of $t_k$ nor of the transversals $\Sigma_{t_k}$. It only depends on the homotopy class of $\gamma$ inside the leaf. Moreover, if two curves give the same map between $\Sigma_x$ and $\Sigma_y$, then they will give the same map between any other pair of transversals $\Sigma_x'$ and $\Sigma_y'$.
	\end{lem}
	
	\begin{proof}(Sketch) That the composition map $\Sigma_x\fto \Sigma_{t_1}\fto \cdots \fto \Sigma_y$ only depends on the homotopy class of $\gamma$ is a direct consequence of the fact that locally there is no choice for the maps.
		
		Take two homotopic paths $\gamma_1$ and $\gamma_2$ and a homotopy $\gamma\colon[0,1]^2\fto M$. For each $(t,s)\in[0,1]^2$ there is a neighborhood $U_{(t,s)}$ of $\gamma(t,s)$ where $\CF$ is isomorphic to $\CF_{loc}$. Using the compactness of $[0,1]^2$ there will be finitely many neighborhoods $U_{(t_k,s_j)}$ covering the homotopy, denote $U_{k,j}:=U_{(t_k,s_j)}$.
		
		One can choose a finite family of curves $\gamma_{j}$ and a finite cover $U_{k,j}$ such that, $\gamma_{\mathrm{min}(j)}=\gamma_0$, $\gamma_{\mathrm{max}(j)}=\gamma_1$ and such that for each $j$, the family $\{U_{(k,j)}\}_k$ covers $\gamma_{j}$ and $\gamma_{j+1}$. By this cover $\gamma_{j}$ and $\gamma_{j+1}$ give the same map from $\Sigma_x$ to $\Sigma_y$. Comparing 2 by 2 we get that $\gamma_0$ and $\gamma_1$ give the same map from $\Sigma_x$ to $\Sigma_y$.
		
		For the second statement, note that $U_{(0,0)}$ contains $\Sigma_x'$ and $\Sigma_x$. The statement is a direct consequence of the fact that inside $U_{(0,0)}$ there are no choices for the map from $\Sigma_x'$ to $\Sigma_x$. Using $U_{(1,1)}$ we get a similar result for $\Sigma_y$ and $\Sigma_y'$. 
	\end{proof}

	\begin{definition} We define an spacial groupoid. Let the space of arrows be the following set: \[\Pi(\CF):=\{\gamma\colon [0,1] \fto M \st \gamma \textnormal{ is in a single leaf}\}/\{\textnormal{homotopy in the leaf}\}.\]
		This space, together with the product by concatenation, the canonical source, target, identity and inverse map, is a set theoretic groupoid over $M$. Call $\Pi(\CF)$ the \textbf{monodromy groupoid} of $\CF$.
	\end{definition}
	
	We will prove that the monodromy groupoid is indeed a Lie groupoid.
	
	\begin{ex}(\textbf{Fundamental groupoid of $M$}) When $\CF=\CX_c(M)$ is the full foliation then $\Pi(\CF)$ is the collection of curves on $M$ up to homotopy. This set is also called the fundamental groupoid of $M$ and it is denoted by $\Pi(M)$. Note that the isotropy groups of $\Pi(M)$ are the fundamental groups of $M$. We will show later that $\Pi(M)$ is a Lie groupoid using the more general construction for $\Pi(\CF)$.
	\end{ex} 
	
	\begin{definition}  Let $[\gamma]\in \Pi(\CF)$ and $\Sigma_0,\Sigma_1$ be two transversals at $\gamma(0)$ and $\gamma(1)$ respectively. Denote $\widehat{\gamma}\colon \Sigma_0\fto \Sigma_1$ the map given in lemma \ref{lem:Holo}. The map $\widehat{\gamma}$ is called the {\bf holonomy transformation} of $[\gamma]$ between $\Sigma_0$ and $\Sigma_1$. 
	\end{definition}
	
	\subsection*{Smooth structure on $\Pi(\CF)$}
	
	For the regular foliated manifold $(\RR^n,\CF_{loc})$ the monodromy groupoid $\Pi(\CF_{loc})$ has a global smooth chart given by $\phi\colon \Pi(\CF_{loc})\fto \RR^k\times\RR^k\times \RR^{n-k};[\gamma_c] \mapsto(\gamma_c(1),\gamma_c(0),c)$ where $\gamma_c$ is a curve in $\RR^k\times\{c\}$ for some $c\in \RR^{n-k}$. Indeed $\Pi(\CF_{loc})$ is isomorphic to the Lie groupoid given in example \ref{ex:loc.hol.grpd}.
	
	We will use the local picture of any regular foliated manifold $(M,\CF)$ to give a chart for any $[\gamma]\in \Pi(\CF)$. To do so, take into account the following facts:
	\begin{itemize}
		\item Choose a chart centered at $\gamma(0)$; $\phi_0:U_0'\fto \RR^k\times \RR^{n-k}$ such that $\phi_0^{-1}\CF=\CF_{loc}$. This chart gives immediately a  transversal $\Sigma_0:=\phi^{-1}(\{(0,y)\st y\in \RR^{n-k}\})$ at $\gamma(0)$.
		
		Similarly for $\gamma(1)$, choose the chart $\phi_1:U_0'\fto \RR^k\times \RR^{n-k}$. Which gives the transversal $\Sigma_1:=\phi^{-1}(\{(0,y)\st y\in \RR^{n-k}\})$.
		
		\item The holonomy transformation $\widehat{\gamma}$ defines a local map $\widehat{\gamma}\colon\Sigma_0\fto \Sigma_1$. We will abuse notation and use $\Sigma_0=\phi_0(\Sigma_0)\subset \{0\}\times \RR^{n-k}\simeq \RR^{n-k}$ and $\Sigma_1=\phi_1(\Sigma_1)\subset \{0\}\times \RR^{n-k}\simeq \RR^{n-k}$. Therefore we will also call $\widehat{\gamma}$ the local diffeomorphism near $0\in \RR^{n-k}$ given by $\widehat{\gamma}$ and the compositions with the charts $\phi_0$ and $\phi_1$.
		
		\item There exists a unique (up to homotopy) smooth map $$\Sigma_0\times [0,1]\fto M;(y,t)\mapsto\gamma_y(t),$$ such that for all $y\in \Sigma_0$ the curves $\gamma_y$ sit inside a leaf of $\CF$, $\gamma_0=\gamma$, $\gamma_y(0)=y\in \Sigma_0$ and $\gamma_y(1)=\widehat{\gamma}(y)\in \Sigma_1$.
		
		\item There exists $U_0\subset U_0'$ and $U_1\subset U_1'$ such that, $\phi_0(U_0)$ and $\phi_1(U_1)$ are convex sets of $\RR^k\times \RR^{n-k}$, we abuse notation and call $U_0=\phi_0(U_0)$ and $U_1=\phi_1(U_1)$.
		
		Because $U_0\subset \RR^k\times \RR^{n-k}$ is a convex set, the leaves of $\iota^{-1}_{U_0}(\CF_{loc})$ are simply connected. Then, there is a unique (up to homotopy) curve inside the leaves connecting $(x_0,y)$ with $(0,y)$, denote this curve by $\sigma_{(x_0,y)}^0$. More precisely $$\sigma_{(x_0,y)}^0 (t):=\phi^{-1}_0((t x_0,y)).$$
		
		Similarly, on $U_1\subset \RR^k\times \RR^{n-k}$ there is a unique up to homotopy curve $\sigma_{(x_1,y)}^1$ that connects $(x_1,\widehat{\gamma}(y))$ with $(0,\widehat{\gamma}(y))$. More precisely $$\sigma_{(x_1,y)}^1 (t):=\phi^{-1}_1((t x_1,\widehat{\gamma}(y))).$$
	\end{itemize}
	
	There is an open set $W\subset \RR^k\times \RR^k\times \RR^{n-k}$ such that the following map is injective:
	\begin{equation}\label{eq:can.chrt}
		\phi:W\fto \Pi(\CF); (x_1,x_0,y)\fto [(\sigma_{(x_1,y)}^1)\circ\gamma_y\circ (\sigma_{(x_0,y)}^0)^{-1}]
	\end{equation}
	where $\circ$ denotes the concatenation of paths and $[-]$ the homotopy class inside the leaves. Note that $\phi(0,0,0)=[\gamma]$, and the curve $(\sigma_{(x_1,y)}^1)\circ\gamma_y\circ (\sigma_{(x_0,y)}^0)^{-1}$ starts in $\phi_0(x_0,y)$ and ends in $\phi_1(x_0,\widehat{\gamma}(y))$. See figure \ref{fig:charthol}. 
	
	\begin{figure}[h]
		\centering
		\scalebox{.3}{\includegraphics{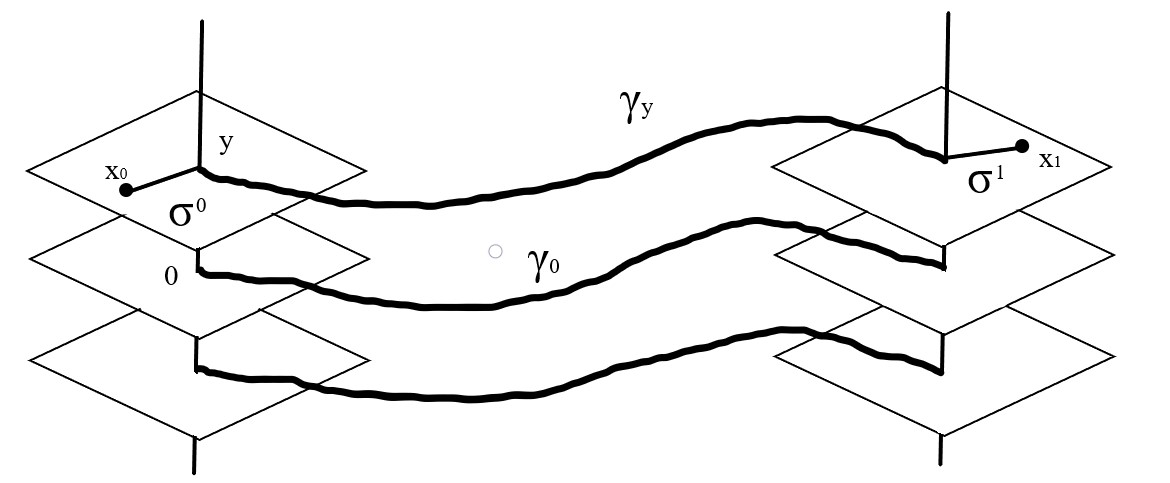}}
		\caption{The image of $(x_1,x_0,y)$ under the chart $\phi$.}
		\label{fig:charthol}
	\end{figure}
	
	We call the pair $(W,\phi)$ of equation (\ref{eq:can.chrt}) a \textbf{canonical chart} for $\Pi(\CF)$, since only depends on $[\gamma]$ and the charts $(U_0,\phi_0)$ and $(U_1,\phi_1)$. These injective charts give a locally euclidean topology on $\Pi(\CF)$. Also, they are smoothly compatible, giving a smooth structure on $\Pi(\CF)$.
	
	Moreover, $\Pi(\CF)$ is second countable because for any point $(p,q)\in M\times_{M/\CF} M$ there are countably many $[\gamma]$'s going from $p$ to $q$ and $M\times_{M/\CF} M$ is second countable. Therefore $\Pi(\CF)$ is indeed a (not necessarily Hausdorff) manifold.
	
	\begin{prop} $\Pi(\CF)$ is indeed a Lie groupoid.
	\end{prop}
	\begin{proof} The maps $\bt,\bs\colon \Pi(\CF)\fto M$ defined by the equations $\bt([\gamma])=\gamma(1)$ and $\bs([\gamma])=\gamma(0)$, can be expressed locally using the canonical charts of equation (\ref{eq:can.chrt}).
		
		In this local picture we get $\bs(x_1,x_0,y)=(x_0,y)$ and $\bt(x_1,x_0,y)=(x_1,\widehat{\gamma}(y))$. Therefore $\bt,\bs$ are submersions onto $M$. By analogous deduction, one can show that the product given by concatenation, the obvious identity and inverse maps are all smooth. As a result, $\Pi(\CF)$ together with these maps forms a Lie groupoid.
	\end{proof}

	The monodromy groupoid also has a relation with bisubmersions in section \ref{sec:bisub} definition \ref{def:bisub}, as the following proposition states.
	
	\begin{prop} Any Hausdorff open set $U\subset\Pi(\CF)$, together with the maps $\bt|_U$ and $\bs|_U$, forms a bisubmersion of $\CF$.
	\end{prop}
	\begin{proof} 
		In the local case, for $(\RR^n,\CF_{loc})$, the monodromy groupoid is basically $\Pi(\CF_{loc})=\RR^k\times \RR^k\times \RR^{n-k}$ with source and target defined as $\bt(x_1,x_0,y)=(x_1,y)$ and $\bs(x_1,x_0,y)=(x_0,y)$. The manifold $\Pi(\CF_{loc})$ is clearly a bisubmersion for $\CF_{loc}$ and so is any open subset.
		
		For a more general regular foliated manifold $(M,\CF)$ take $[\gamma]\in \Pi(\CF)$, $\phi_0$ and $\phi_1$ charts compatible with $\CF$ at $\gamma(0)$ and $\gamma(1)$. The canonical chart $\phi: W\subset \RR^k\times \RR^k\times \RR^{n-k}\fto \Pi(\CF)$ with the maps $\bs(\phi(x_1,x_0,y))=\phi_0(x_0,y)$ and $\bt(\phi(x_1,x_0,y))=\phi_1(x_1, \widehat{\gamma}(y))$ is a bisubmersion for $\CF$. Being a bisubmersion is a local property, therefore for any Hausdorff open set $U\in \Pi(\CF)$ the triple $(U,\bt,\bs)$ is a bisubmersion for $\CF$.	
	\end{proof}
	
	In fact in proposition \ref{prop:grpd.bisub} we will show that any Hausdorff open set $U$ of a Lie groupoid $\CG\soutar M$ is a bisubmersion for a foliation on $M$ that depends on $\CG$.
	
	\subsection*{The holonomy groupoid}
	
	On $\Pi(\CF)$ one can define an equivalence relation compatible with the Lie groupoid structure. We say that two homotopy classes $[\gamma]$ and $[\sigma]$ are holonomy equivalent $[\gamma]\sim_\mathrm{hol} [\sigma]$ if and only if, after choosing transversals, their holonomy transformations are equal $\widehat{\gamma}=\widehat{\sigma}$.
	
	\begin{defi} The \textbf{holonomy groupoid} of a regular foliated manifold $(M,\CF)$ is the Lie groupoid given by $\CH(\CF):=(\Pi(\CF)/\sim_\mathrm{hol})$.
	\end{defi}
	
	We endow $\CH(\CF)$ with the quotient topology, which makes the quotient map $Q\colon \Pi(\CF)\fto \CH(\CF)$ a continuous map.  Note that, for any canonical chart $\phi$ of $\Pi(\CF)$ we get that the map $Q\circ \phi\colon W\fto \Pi(\CF)\fto \CH(\CF)$ is injective. Therefore the canonical charts of $\Pi(\CF)$ can be used as charts for $\CH(\CF)$ and endow it with a smooth structure that makes the quotient map $\Pi(\CF)\fto \CH(\CF)$ into an étale map.
	
	\begin{ex} Let $M$ be a connected manifold, $\widehat{M}$ its universal covering space, $x_0\in M$ and $\CF:=\CX_c(M)$. Then $\Pi(\CF)$ is isomorphic to the fundamental groupoid of $M$ which is given by  $(\widehat{M}\times \widehat{M})/(\pi_1(x_0))$. Moreover the holonomy groupoid is $\CH(M)\cong M\times M$, the pair groupoid.
	\end{ex}
	
	\begin{ex}
		For the cilinder $M=S^1\times \RR$ take $\CF$ with leaves $S^1\times\{y\}$ for $y\in \RR$. Then:
		\begin{itemize}
			\item $\Pi(\CF)=\RR\times M$ with $\bs(t,(\theta,y)):=(\theta,y)$ and $\bt(t,(\theta,y)):=(e^{it}\theta,y)$.
			\item $\CH(\CF)=S^1\times  M$ with $\bs(\sigma,(\theta,y)):=(\theta,y)$ and $\bt(\sigma,(\theta,y)):=(\sigma\theta,y)$.
		\end{itemize}
	\end{ex}
	
	\begin{ex}\label{ex:moebius}
		Take $M$ to be the Moebius band $M:= \RR\times (-1,1)/\sim$ where $(x,y)\sim (x+k,(-1)^k y)$ for $k\in \ZZ$. Let $\CF$ be the foliation with leaves $[\RR\times \{y\}]$ for fixed $y\in(-1,1)$ (see figure \ref{fig:moebius}). Then:
		\begin{itemize}
			\item $\Pi(\CF)=\RR\times M$ with $\bs(t,[(x,y)]):=[(x,y)]$ and $\bt(t,[(x,y)]):=[(t+x,y)]$.
			\item $\CH(\CF)= S^1\times M$ with $\bs(\sigma,[(x,y)]):=[(x,y)]$ and $\bt(\sigma,[(x,y)]):=[((t_\sigma/\pi)+x,y)]$ where $t_\sigma\in \RR$ is any real such that $e^{it_\sigma}=\sigma$.
		\end{itemize}	
	\end{ex}
	
	\begin{figure}[h]
		\centering
		\scalebox{.4}{\includegraphics[ height=7cm, width=12cm]{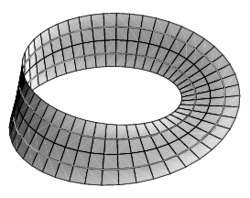}}
		\caption{The foliated manifold in example \ref{ex:moebius}}
		\label{fig:moebius}
	\end{figure}
	
	We describe in the following proposition a different characterization of the holonomy equivalence which will be used to generalize this notion to the non-regular case.
	
	\begin{prop}\label{prop:eq.reg.hol} $[\gamma]\sim_\mathrm{hol} [\sigma]$ if and only if there exists open neighborhoods $U,V\in \Pi(\CF)$ and morphisms of bisubmersions $\varphi\colon U\fto V$ such that $\varphi([\gamma])=[\sigma]$.
	\end{prop}
	\begin{proof} 
		Take $[\gamma_1],[\gamma_2]\in \Pi(\CF)$, such that $p:=\gamma_1(0)=\gamma_2(0)$ and $q:=\gamma_1(1)=\gamma_2(1)$. Take charts for $p,q$ adapted to $\CF$ and construct the canonical charts as given in equation (\ref{eq:can.chrt}), say $\phi_1\colon W_1\fto \Pi(\CF)$ for $[\gamma_1]$ and $\phi_2\colon W_2\fto \Pi(\CF)$ for $[\gamma_2]$. Note that $W_1,W_2\subset \RR^k\times \RR^k\times \RR^{n-k}$ are open neighborhoods of $(0,0,0)$; $\phi_1(0,0,0)=[\gamma_1]$ and $\phi_2(0,0,0)=[\gamma_2]$. We will prove that $\widehat{\gamma_1}=\widehat{\gamma_2}$ if and only if there is a local bisubmersion map $W_1\fto W_2$ sending $(0,0,0)$ to $(0,0,0)$:
		
		($\Rightarrow $) If $\widehat{\gamma_1}=\widehat{\gamma_2}$, the identity map, defined locally in $(0,0,0)$ on the neighborhood $W_1\cap W_2\subset \RR^k\times \RR^k\times \RR^{n-k}$, is clearly a bisubmersion map.
		
		($\Leftarrow$) Suppose that there is a local bisubmersion map $\varphi\colon W_1\fto W_2$ sending $(0,0,0)$ to $(0,0,0)$. Write $\varphi=(\varphi_1,\varphi_2,\varphi_3)$ and recall that $\bs_1(x_1,x_0,y)=\bs_2(x_1,x_0,y)=(x_0,y)$, $\bt_1(x_1,x_0,y)=(x_1,\widehat{\gamma}_1(y))$ and $\bt_2(x_1,x_0,y)=(x_1,\widehat{\gamma}_2(y))$. 
		By easy computations, using that $\varphi$ must commute with the sources and targets one shows that $\varphi=Id$ and $\widehat{\gamma_2}(y)=\widehat{\gamma_1}(y)$.
	\end{proof}

	\section{Lie algebroids}\label{sec:Lie.Algd}
	
	In this section we introduce the notion of Lie algebroids and its relation to singular foliations and Lie groupoids. Most of the material presented here is not original and can be found also in \cite{CrFeLie}, \cite{MK2} and \cite{Kwang}.
	
	Lie algebroids can be thought of as a generalization of the tangent bundle. We will motivate this idea in the examples \ref{ex:L.alg.TM} and \ref{ex:L.alg.Con}. 
	
	\begin{defi}\label{def:Lie.algd} A {\bf Lie algebroid } on a manifold $M$ is a triple $(A,\rho,[-,-])$, with $A\fto M$ a vector bundle, $\rho:A\fto TM$ a vector bundle morphism called anchor and $[-,-]$ a Lie bracket on the sections $\Gamma(A)$, such that:
		\[[X,fY]=f[X,Y]+(\rho(X)f) Y, \hspace{.3in} \forall \hspace{.1in} X,Y\in \Gamma(A) \y f\in \CI(M).\]
	\end{defi}
	
	\begin{rem} The map $\rho:\Gamma(A)\fto \CX(M)$ is a Lie algebra map and that $\rho$ is completely characterized by the Lie bracket $[-,-]$.
	\end{rem}
	
	\begin{ex}\label{ex:L.alg.TM} $TM$ with the identity map and the Lie bracket on vector fields is a Lie algebroid over $M$. 
	\end{ex}
	
	Note that the set of sections of $TM$ (the vector fields on $M$) corresponds to the set of derivations of the algebra $\CI(M)$, which can be seen as sections of a line bundle. Following this idea we give the next example:
	
	\begin{ex}\label{ex:L.alg.Con} Let $E\fto M$ be a vector bundle and $T\colon \Gamma(E)\fto \Gamma(E)$ be an $\RR$-linear transformation. $T$ is a differential operator of order $1$ iff for any function $f\in \CI(M)$ the commutator $[T,f](\sigma):= T(f\sigma)-fT(\sigma)$ is a $\CI(M)$-linear map, i.e. a vector bundle map. $T$ is a covariant differential operator iff for all $f\in \CI(M)$ the commutator $[T,f]$ is the multiplication by a function.
		
		Denote $\mathrm{CDO}(E)$ the set of covariant differential operators on the sections of $E$. For any $T\in \mathrm{CDO}(E)$, the map $[T,-]\colon \CI(M)\fto \CI(M)$ is always a derivation. This derivation is called the symbol of $T$.
		
		It is known that there exists a vector bundle $D\fto M$ such that $\Gamma(D)\simeq \mathrm{CDO}(E)$. On $\Gamma(D)$  there is a canonical Lie bracket $[-,-]$ given by the commutator in $\mathrm{CDO}(E)$.
		
		The equation $\rho(T)(f)=[T,f]$ defines an anchor $\rho:D\fto TM$. In conclusion $(D,\rho,[-,-])$ is a Lie algebroid. 	
	\end{ex}
	
	One of the motivations for this work is the singular foliation associated to a Poisson manifold, therefore an important example is the following: 
	
	\begin{ex}\label{ex:L.alg.Pois} Let $(M,\pi)$ be a Poisson manifold. This is a manifold $M$ together with a bi-vectorfield $\pi\in \Gamma(\wedge^2 TM)$, such that $\pi$ satisfies a Jacobi equation.
		
		There is a map $\pi^\sharp\colon T^*M\fto TM; \alpha\mapsto \pi(\alpha,-)$ and a Lie bracket on $\Omega^1 (M)$ given by the formula:
		\[[\alpha,\beta]=\CL_{\pi^\sharp\alpha} \beta- \CL_{\pi^\sharp\beta} \alpha-\pi(\alpha,\beta).\]
		$(T^*M,\pi^\sharp,[-,-])$ is a Lie algebroid and encodes the same information as the Poisson manifold.
	\end{ex}
	
	Before we continue let us give the notion of morphisms for Lie algebroids over the same manifold:
	
	\begin{defi} Let $(A,\rho,[-,-]_A)$, $(A',\rho',[-,-]_{A'})$ be Lie algebroids over $M$. A map of Lie algebroids from $A$ to $A'$ is a vector bundle map $\phi\colon A\fto A'$ covering $Id_M$, such that the induced map on sections is a Lie algebra map.
		
		Note that, the equality $\rho'\phi(x)=\rho(x)$ will be automatically satisfied.	
	\end{defi}
	
	\begin{ex} Given a Lie algebroid $(A,\rho,[-,-]_A)$, the anchor $\rho:A\fto TM$ is a Lie algebroid map to $(TM, Id_{TM},[-,-])$.
	\end{ex}
	
	\subsection*{The Lie algebroid of a Lie groupoid}
	
	A very well known result in Lie group theory is that for any Lie group $G$ there is an associated Lie algebra $\g$. In a direct analogy to this case, for any Lie groupoid $\CG$ there is an associated Lie algebroid. This subsection is dedicated to recall this construction. 
	
	Let $\CG\soutar M$ be a Lie groupoid.  To get a Lie algebroid from $\CG$, take into account section \ref{sec:Lie.Grpd} which gives us some intrinsic structures on $\CG$. We will construct a vector bundle $A$ on $M$, an anchor map $\rho$ and a Lie bracket on $\SEC(A)$.
	
	As the vector bundle we take $A:=\ker(d\bs)|_M$ (Identifying $M$ as the identity section on $\CG$). The vector bundle map $\rho:=d\bt\colon A\fto TM$ is the anchor map. For the Lie bracket one must notice the following:
	
	\begin{itemize}
		\item  Any $g\in \CG$ gives a smooth map $R_g\colon \bs^{-1}(\bt (g))\fto \bs^{-1}(\bs (g)); h\mapsto g\circ h$.
		\item The space $T_h(\bs^{-1}(x))= \ker(d_h\bs)$. Therefore $dR_g\colon \ker(d_h \bs)\fto \ker(d_{g\circ h}\bs)$.
	\end{itemize}
	
	This gives us a way to extend a section $X$ of $A$ to a vector field $\overrightarrow{X}$ on $\CG$ by $\overrightarrow{X}_g:=dR_g(X_{\bs(g)})$.
	
	Then the Lie bracket $[X,Y]$ for $X,Y\in \Gamma(A)$ is defined by the formula
	\[\overrightarrow{[X,Y]}=[\overrightarrow{X},\overrightarrow{Y}].\]
	
	\begin{defi}\label{ex:Algd,Gprd} Let $\CG\soutar M$ be a Lie groupoid. The Lie algebroid of $\CG$ is given by $(A,\rho,[-,-])$ constructed above.
	\end{defi}
	
	\begin{rem}
		It is possible to make the same construction using $A_L:=\ker(d\bt)|_M$ and the left invariant vectorfields $\overleftarrow{X}_g:= dL_g(X_{\bt(g)})$.
	\end{rem}
	
	\begin{ex} We saw in Example \ref{ex:G.group} that any Lie group is a Lie groupoid. Following the same construction we get the Lie algebroid $(\g,\rho,[-.-])$ where $\g\fto \{*\}$ is the Lie algebra associated to $G$, thought of as a vector bundle over a point; the anchor  $\rho$ is zero and the Lie bracket $[-,-]$ is the canonical one.
	\end{ex}
	
	\begin{ex} Let $\g$ be a Lie algebra. An infinitesimal action $\rho\colon \g\times M\fto TM$ is canonically a Lie algebroid. Where $A=\g\times M$ is a vector bundle $\rho$ is the anchor map and the Lie bracket on $\SEC(A)$ is given by the Lie bracket on $\g$. Moreover if this action can be integrated to a Lie group action $G\times M\fto M$, then $\g\times M\fto TM$ is the Lie algebroid of the action Lie groupoid $G\times M\soutar M$.
	\end{ex}
	
	\begin{ex} The Lie algebroid of the pair groupoid is $TM$ with the identity as anchor.	
	\end{ex}
	
	Although Example \ref{ex:L.alg.Con} is not needed in the rest of this thesis, it shows a quite interesting object, which is widely studied and has applications in many branches of mathematics. It is interesting to notice that there is a Lie groupoid which integrates this Lie algebroid of covariant differential operators. We describe its structure as follows:
	
	\begin{ex}({\bf Connection groupoid}) Let $M$ be a manifold and $\pi\colon E\fto M$ a vector bundle. Take:
		\[\CG:=\cup_{(y,x)\in M\times M} \{\textnormal{ Isomorphisms } E_x\fto E_y\}.\]
		
		The set $\CG$ is indeed a manifold, it can be seen as the vector bundle $\mathrm{Pr}_1^*E \otimes \mathrm{Pr}_2^*E^*$ over $M\times M$ where $\mathrm{Pr}_1,\mathrm{Pr}_2\colon M\times M\fto M$ are the first and second projections.
		
		We denote an element of $\CG$ as $\phi_{(y,x)}$ to indicate that it is a map from $E_x$ to $E_y$.
		\begin{itemize}
			\item The source is defined as $\bs\colon \CG\fto M; \phi_{(y,x)}\mapsto x$ and the target $\bt\colon \CG\fto M;\phi_{(y,x)}\mapsto y$.
			\item The composition is given by $\circ \colon \CG {}_{\bs} \!\times_{\bt} \CG\fto \CG; \phi_{(z,y)}\circ \sigma_{(y,x)}\mapsto (\phi\circ \sigma)_{(z,x)}$.
			\item The identity section is $e\colon M\fto \CG; x\mapsto Id_{E_x}$ where $Id_{E_x}\colon E_x\fto E_x$ is the identity map.
			\item And the inverse map is given by $(-)^{-1}\colon \CG\fto \CG; \phi_{(y,x)}\mapsto \phi^{-1}_{(x,y)}$.
		\end{itemize}
		One can see that the Lie algebroid of $\CG$ is the one of example \ref{ex:L.alg.Con} given by covariant differential operators.
	\end{ex}
	
	\begin{rem}\textbf{(Flat connections and integrability)} \label{rem:flat.conn} A flat connection $\nabla$ on a vector bundle $E\fto M$ is a Lie algebra map $\nabla\colon \CX(M)\fto \mathrm{CDO}(E);X\mapsto \nabla_X$ such that $\nabla_{fX}=f\nabla_X$ for all $f\in \CI(M)$ and such that the symbol of $\nabla_X$ is $X$. This condition implies that $\nabla$ can be seen as a Lie algebroid map $\nabla\colon TM\fto D$ where $D$ is the vectorbundle of covariant differential operators on $E$. If we suppose that $M$ is simply connected, it is possible to integrate the connection into a Lie groupoid map from the pair groupoid $M\times M$ to the connection groupoid $E\otimes E^*$.
		
		If $M$ is not simply connected, it is still possible to integrate the connection to a Lie groupoid map $\CG\fto E\otimes E^*$, where $\CG$ is the fundamental groupoid of $M$.
	\end{rem}
	
	\begin{rem}\textbf{(Flat connections and Holonomy)} From a different perspective, a connection $\nabla\colon \CX(M)\fto \mathrm{CDO}(E);X\mapsto \nabla_X$ gives us an Ehresmann connection $\rho_\nabla \colon \pi^*TM\fto TE$. If the connection is flat then the Ehresmann connection gives us an integrable distribution i.e. a Lie algebroid and a regular foliation $\CF_\nabla$ on $E$. Note that $\CH(\CF_\nabla)$, which is a groupoid over $E$, coincides with the general notion of holonomy for connections.
	\end{rem}
	
	\subsection*{Foliations, Lie algebroids and Lie groupoids}
	
	An important fact in this thesis is the close relation between Lie groupoids, Lie algebroids and singular foliations. We already saw in definition \ref{ex:Algd,Gprd} the relation between Lie groupoids and Lie algebroids and in section \S \ref{sec:hol.grpd} a groupoid associated to a regular foliation. Here we will relate Lie algebroids with foliations using the example \ref{ex:L.alg.F} and proposition \ref{prop:Folie.Algd}.
	
	\begin{ex}\label{ex:L.alg.F}({\bf Projective foliations as Lie algebroids}) Let $\CF$ be a projective foliation foliation on $M$ (In particular regular foliations are projective), then by Lemma \ref{lem:fiber.proy} there exists a vector bundle $A\fto M$ such that $\Gamma_c(A)\cong \CF$ as a $\CI(M)$-module. Then there exists a canonical almost injective Lie algebroid structure on $A$.
		
		In particular, when $\CF$ is regular, the fibers are given by $A_x:=F_x=\{X(x)\st X\in \CF\}$ for all $x\in M$, the anchor map $\rho:F\fto TM$ is the inclusion map; and the Lie bracket on $A$ is the usual Lie bracket on $M$.
	\end{ex}
	
	\begin{prop}\label{prop:Folie.Algd}({\bf Lie algebroids inducing singular foliations}) Let $(A,\rho,[-,-])$ be a Lie algebroid on $M$, then $\CF:=\rho(\Gamma_c(A))\subset \CX_c(M)$ is a singular foliation. 	
	\end{prop} 
	\begin{proof} The set $A$ being a vector bundle implies that $\CF$ is locally finitely generated and $\rho$ being a map of Lie algebras implies that $\CF$ is closed under the Lie bracket.
	\end{proof}
	
	The following proposition shows the relation between \textbf{Lie groupoids and bisubmersions}.
	
	\begin{proposition}\label{prop:grpd.bisub} Let $\CG\soutar M$ be a Lie groupoid, $A$ its Lie algebroid, and $\CF:=\rho(\SEC_c(A))$ its singular foliation, where $\rho$ is the anchor of $A$. Any Hausdorff open set $U\subset \CG$ together with $\bt|_U,\bs|_U$ is a bisubmersion for $\CF$. When $\CG$ is a Hausdorff Lie groupoid, then it is also a bisubmersion.
	\end{proposition}
	\begin{proof}(Sketch) Consider the set of right invariant vectorfields $\overrightarrow{A}:=\{\overrightarrow{X} \st X\in\SEC(A)\}$, and the set of left invariant vectorfields $\overleftarrow{A}:=\{\overleftarrow{X} \st X\in\SEC(A)\}$. Note that:
		\[\ker_c(d\bs|_U)= \left<\overrightarrow{A}|_U\right>_{\CI_c(U)} \y \ker_c(d\bt|_U)=\left< \overleftarrow{A}|_U\right>_{\CI_c(U)}.\]
		Using this one can show that:
		\[\bs^{-1}\CF=\bt^{-1}\CF=\left< (\overrightarrow{A}|_U+\overleftarrow{A}|_U)\right>_{\CI_c(U)}=\ker(d\bs|_U)+\ker(d\bt|_U).\]
	\end{proof}
	
	\begin{ex} Let $\CF$ be a regular foliation on a manifold $M$. Then the Lie algebroid of $\Pi(\CF)$ and $\CH(\CF)$ is $F$ as in example \ref{ex:L.alg.F}.
	\end{ex}

	\section{Pullbacks of Lie groupoids and Lie algebroids}\label{sec:pull.grpd.algd}
	
	We will start by following the book \cite{MK2} by Mackenzie to present the pullback for Lie algebroids and Lie groupoids. Later in this section, we state an original result from \cite{ME2018} by Marco Zambon and I, namely Lemma \ref{lem:pullbackalgfol}, which relates these two notions with the pullbacks of singular foliations.
	
	Before we start, we recall two lemmas that will be used repeatedly. The first one is Lemma \ref{lem:subcon}, which we already introduced in \S \ref{sec:ME1}:
	
	{\bf Lemma \ref{lem:subcon}.} {\it Let $A$ and $B$ be manifolds, { $k\ge 0$}, and $f\colon A\fto B$ a surjective submersion with $k$-connected fibers. If $B$ is $k$-connected then $A$ is $k$-connected.}

	The second one is a well known fact on topology and differential geometry:
	
	\begin{lem}\label{lem:openmapABC} Let $A,B,C$ be topological spaces (or manifolds) and $f\colon A\fto C$ {a continuous map} (or smooth map). If $g\colon B\fto C$ is a continuous and \emph{open} map (or a submersion), then {$\text{Pr}_1\colon A {}_f\!\times_g B\fto A$} is also a {continuous and \emph{open} map} (or a submersion), where the domain is endowed with the subspace topology.
		
		{Moreover, if {$g$} is surjective then {$\text{Pr}_1$} also is}.
		\[\begin{tikzcd}
			A {}_f\!\times_g B \arrow[r,"\Pr_2"] \arrow[d,swap, "\Pr_1"]& B \arrow[d, "g"]  \\
			A \arrow[r, "f"] & C\\
		\end{tikzcd}\]
		
	\end{lem}
	\begin{proof} We will only prove that the map $\text{Pr}_1$ is open. We leave as exercise the smooth and surjective cases.
		
		Take $U$ an open set in $A {}_f\!\times_g B$ and $a\in \text{Pr}_1(U)$, it suffices to show that $a$ is an interior point of $\text{Pr}_1(U)$.
		
		There exists $b\in B$ such that $(a,b)\in U$, moreover because $U$ is open there exists open sets $U_A\subset A$ and $U_B\subset B$ such that $(a,b)\in (U_A\times U_B)\cap (A {}_f\!\times_g B)\subset U$.
		
		Take $U_A':=U_A\cap f^{-1}(g(U_B))$ note that $U_A'$ is an open neighborhood of $a$ and moreover that $U_A'\subset \text{Pr}_1(U)$, which proves that $a$ is in the interior of $\text{Pr}_1(U)$.	
	\end{proof}
	Now we are ready to define the \textbf{pullback of a Lie groupoid}. Please note that this definition is analogous for \textbf{topological groupoids}.
	
	\begin{defi}\label{def:pullbackgroid}
		{Given a Lie groupoid $\CG\soutar M$ (or a topological groupoid) and a surjective submersion $\pi:P\fto M$ (or a continuous map), the manifold $$\pi^{-1} \CG:= P{}_\pi\!\times_{\bt} \CG {}_{\bs}\!\times_\pi P$$ is the space of arrows of a  Lie groupoid over $P$. (The   source and target maps are the first and third projections  {and the multiplication is induced by the multiplication on $\CG$}).} This Lie groupoid is called the {\bf pullback groupoid} of $\CG$ by $\pi$.
	\end{defi}
	
	\begin{prop}\label{prop:pullback.grpd}
		If $\CG\soutar M$ is a Lie groupoid (or an open topological groupoid) and $\pi\colon P\fto M$ is a surjective submersion (or a continuous open and surjective map), then $\pi^{-1} \CG$ is a Lie groupoid (or an open topological groupoid).
	\end{prop}
	\begin{proof}
		{We show that the target map of $\pi^{-1} \CG$ is a submersion (or open map). 
			The first projection $\text{Pr}_1$ of
			$\CG {}_\bs\!\times_\pi P$ is a submersion by lemma \ref{lem:openmapABC}, since $\pi\colon P \to M$ is a submersion. Hence the composition $\bt\circ\text{Pr}_1$ is a submersion. Again by lemma \ref{lem:openmapABC}, this implies that the first projection of 
			$P {}_\pi \!\times_{\bt\circ\text{Pr}_1}\left(\CG {}_\bs\!\times_\pi P\right)= \pi^{-1} G$ is a submersion, and this is precisely the target map of $\pi^{-1} \CG$.} For the source map, proceed similarly.
		For open topological groupoids proceed similarly.
	\end{proof}
	
	Moreover in the setting of Lemma \ref{GtoA}, if $\CG$ is source $k$-connected and $\pi$ has $k$-connected fibers, by lemma \ref{lem:subcon} we get that that $\pi^{-1}\CG$ is source $k$-connected.
	
	Now we introduce the \textbf{pullback of a Lie algebroid}:
	
	\begin{defi}\label{def:pullbackLA}
		Given a Lie algebroid $A$ over a manifold $M$ with anchor $\#:A\fto TM$, and a surjective submersion $\pi\colon P\fto M$, one checks that 
		$$\pi^{-1} A:= \pi^*(A) {}_\# \!\times _{d\pi} TP$$ is the total space of a vector bundle over $P$. It has a  natural Lie algebroid structure, with anchor $\widehat{\#}:=\text{pr}_2\colon \pi^{-1} A \fto TP$ being the second projection. {The Lie bracket is determined by its {restriction to} ``pullback sections'', which is given by the Lie brackets on $\CX(P)$ and $\SEC(A)$}.  We call this Lie algebroid the  pullback algebroid of $A$ over $\pi$. 
	\end{defi}
	
	These two definitions are nicely related by the following lemma:
	
	\begin{lem}\label{GtoA}   
		Consider a surjective submersion $\pi:P\fto M$.
		\begin{enumerate}
			\item[(i)] Let $\CG$ be a Lie groupoid over $M$, denote by $A$ its Lie algebroid. The Lie algebroid of the Lie groupoid $\pi^{-1} \CG$ is $\pi^{-1} A$.
			\item[(ii)] Let $A$ be an integrable Lie algebroid over $M$, denote by $\CG$ the source simply connected Lie groupoid integrating it. If the map $\pi$ has simply connected fibers, then the source simply connected Lie groupoid integrating $\pi^{-1} A$ is $\pi^{-1} \CG$.
		\end{enumerate}
	\end{lem}
	\begin{proof}
		The proof of part (i)   can be found in \cite[\S 4.3]{MK2}, so we address only the proof of part (ii). The Lie groupoid $\pi^{-1} \CG$ integrates $\pi^{-1} A$ by part (i). Therefore we only need to show that $\pi^{-1} \CG$ is source simply connected. Take $p\in P$. Its source fiber is
		\[\bs^{-1}(p)=\{(q,g,p) \st \pi(p)=\bs(g) \y \pi(q)=\bt (g)\}\simeq P {}_\pi \!\times_\bt \bs^{-1}(\pi(p)).\]
		The canonical submersion $\bs^{-1}(p)\fto \bs^{-1}(\pi(p))$ has simply connected fibers, {since the $\pi$-fibers are simply connected.} Using that $\bs^{-1}(\pi(p))$ is simply connected { and lemma \ref{lem:subcon}} we conclude that $\bs^{-1}(p)$ is simply connected.
	\end{proof}
	
	The following original result relates the two definitions above with the pullback of foliations:
	
	\begin{lem}  \label{lem:pullbackalgfol}
		Consider a surjective submersion $\pi:P\fto M$. Let $A$ be a Lie algebroid over $M$ with anchor $\#:A\fto TM$. Then the foliation of the pullback Lie algebroid $\CF_P:=\widehat{\#}(\SEC_c(\pi^{-1} A))$   equals $\pi^{-1}(\CF_M)$, where $\CF_M:=\#\SEC_c(A)$. 
	\end{lem}
	\begin{proof}
		For the inclusion ``$\supset$'' we argue as follows. For all $X \in\pi^{-1}(\cF_M)$ we have:
		\begin{equation} \label{MEGrFl}
			d\pi(X)=f_1 \pi^*(Y_1)+\dots f_n \pi^*(Y_n)
		\end{equation}
		for some $Y_1\dots Y_n\in \CF_M$ and $f_1,\dots,f_n\in \CI_c(P)$. There exists $\alpha_1,\dots,\alpha_n\in \Gamma_c(A)$ such that $\#(\alpha_i)=Y_i$. Denote $\widehat{\beta}:=f_1 \pi^* \alpha_1+\dots+f_n\pi^* \alpha_n$, a section in $\SEC\pi^*A$. Using eq.  \eqref{MEGrFl} we get that $(\widehat{\beta},X)\in \SEC_c(\pi^{-1} A)$ and moreover $\#(\widehat{\beta})=X$, so $X\in \#(\SEC_c(\pi^{-1} A))$.
		
		For the other inclusion take $(\widehat{\beta},X)\in \SEC_c(\pi^{-1} A)$. {The module of sections of $A$ is finitely generated due to} the Serre-Swan theorem, hence $\widehat{\beta}\in \SEC_c(\pi^{*}(A))$ can be written as $\widehat{\beta}= f_1 \pi^* \alpha_1+\dots+f_n\pi^* \alpha_n$ for some $f_1,\dots,f_n\in \CI_c(P)$ and $\alpha_1,\dots,\alpha_n\in \Gamma_c(A)$.
		{Since $\pi$ is a submersion, each $\# \alpha_i\in \CF_M$ can be lifted via $\pi$ to a vector field $X_i$ on $P$. By construction  $\sum_i f_iX_i$ lies  in $\pi^{-1}(\CF_M)$. The conclusion follows since
			the difference $X- \sum_i f_iX_i$ lies in $\Gamma_c(\ker \pi_*)$, and therefore in $\pi^{-1}(\CF_M)$. }
	\end{proof}
	
	\section{Morita Equivalence for Lie groupoids and Lie algebroids} \label{sec:ME.gpd}

	In examples \ref{ex:Gid} (the identity groupoid) and \ref{ex:cech.grpd} (the Čech groupoid), we saw two Lie groupoids that are similar, in the sense that both represent the manifold $M$ in two different ways. By definition the Čech groupoid $\CG\soutar \widehat{M}$ is the pullback of the identity groupoid $M\soutar M$ under the canonical surjective submersion $\widehat{M}\fto M$, where $\widehat{M}$ is the disjoint union of an open cover of $M$. Having this in mind helps to illustrate the following definition by Mackenzie in \cite{MK2}:
	
	\begin{defi}\label{def:MEgroid}
		Two Lie groupoids $\CG\rightrightarrows M$ and $\CH\rightrightarrows N$ are \textbf{Morita equivalent} if  there exists a (Hausdorff) manifold $P$, and two surjective submersions $\pi_M:P\fto M$ and $\pi_N:P\fto N$ such that $\pi^{-1}_M \CG\cong \pi_N^{-1}\CH$ as Lie groupoids.
		
		In this case, we call $(P,\pi_M,\pi_N)$ a Morita equivalence between $\CG$ and $\CH$.
		
		If $\CG$ and $\CH$ are open topological groupoids, it is asked for $P$ to be a topological space, $\pi_M$, $\pi_N$ to be open surjective maps and $\pi^{-1}_M \CG\cong \pi_N^{-1}\CH$ as topological groupoids.
	\end{defi}	
	
	\begin{rem} Recall that provided an open topological groupoid $\CG\soutar M$, and $\pi\colon P\fto M$ a surjective open map then $\pi^{-1}\CG$ is an open topological groupoid over $P$.
	\end{rem}
	
	One of the most important implications of this section is proposition \ref{prop: implications}. There we will show a relation between Morita equivalence of Lie groupoids, Lie algebroids and Hausdorff Morita equivalence of singular foliations.
	
	\subsection*{Equivalent characterizations of Morita equivalence for Lie groupoids}
	
	{Definition \ref{def:MEgroid}  is equivalent to several other characterizations of Morita equivalence for Lie groupoids, as   was proved in \cite{MdMkGrd},   {(see also \cite{LauMatXu}).} An analogous statement holds also for open topological groupoids, upon replacing   submersions by continuous open  maps. In this subsection we recall these results and prove implications that are used later, {the main one being corollary \ref{cor:SConnMEGrpd2}.}
		
		{We start recalling the notion of weak equivalence, {as given in \cite[\S 1.3]{DPronk}}, and of  bitorsor.}

		\begin{defi}\label{def:weakeq}
			Let $\CH\soutar M$ and $\CG\soutar P$ be two Lie groupoids (respectively, topological groupoids). A morphism $\widehat{\pi}\colon\CG\fto \CH$ is a {\bf weak equivalence} if:
			\begin{enumerate}
				\item[(i)] $\CG\fto \pi^{-1} \CH; \;\;\gamma\mapsto (\bt(\gamma),\widehat{\pi}(\gamma),\bs(\gamma))$ is an isomorphism,
				\item[(ii)] $\bt\circ\text{Pr}_1\colon \CH {}_{\bs}\!\times_{\pi} P\fto M$ is a surjective submersion (resp. a surjective continuous and open map).
			\end{enumerate}
			{Here $\pi\colon P\to M$ denotes the base map covered by $ \widehat{\pi}$.}
		\end{defi}
		
		\begin{rem}\label{rem:weakeq}
			\begin{itemize}
				\item[i)] Looking at a groupoid as a small category, a weak equivalence is the same thing as a {fully}  faithful and essentially surjective functor.
				\item[ii)] When a map $\pi\colon P\fto M$ is completely transverse (transverse to the orbits and meeting every orbit) to a Lie groupoid $\CH\soutar M$, then the projection $\pi^{-1} \CH \fto \CH$ is a weak equivalence.
				\item[iii)] {If $\CH$ is an open topological groupoid and ${\pi}$ is a continuous, open and surjective map, then condition (ii) in definition \ref{def:weakeq} is automatically satisfied, as can be showed using}
				lemma \ref{lem:openmapABC}.
			\end{itemize}
		\end{rem}
		Recall definition \ref{def:grpd.lact} (of Lie groupoid actions), in which a Lie groupoid $\CG\soutar M$ acts on a not necessarily Hausdorff manifold $P$ via a moment map $\pi\colon P\fto M$.  	
		\begin{defi} If the left $\CG$-action is free and proper (the map $\star \colon \CG {}_\bs\!\times_\pi P \fto P$ is proper) then $P/\CG$ is a manifold and we say that $P$ is a {\bf $\CG$-principal left bundle}. Similarly with a right $\CG$-action.
		\end{defi}	
		\begin{defi}A $(\CG,\CH)$-bimodule $P$ that is principal with respect to both actions and such that $P/ \CG\cong N$ and $\CH \backslash P\cong M$  is called a {\bf $(\CG,\CH)$-bitorsor}.
		\end{defi}
		
		\begin{ex} Any Lie groupoid $\CG\soutar M$ is a $(\CG,\CG)$-bitorsor, with the canonical actions given by left and right multiplication on $\CG$. 
		\end{ex}
		
		The following statement can be found in \cite[\S 2.5]{MdMkGrd}
		\begin{lem}\label{lem:bitorsor}
			Consider a Lie groupoid $\Gamma\soutar K$, a $\Gamma$-principal bundle $S$, a $\Gamma$-module $Q$ and a map $f\colon Q\fto S$ preserving the $\Gamma$ actions. Then $Q/\Gamma$ is a manifold.
		\end{lem}
		
		The following proposition gives a equivalent characterizations of ME for Lie groupoids.
		\begin{prop}\label{prop:AllME} Let $\CG\soutar M$ and $\CH\soutar N$ be Lie groupoids. The following statements are equivalent: 
			\begin{enumerate}
				\item[(i)] There exists a Lie groupoid $\Gamma$ and two weak equivalences $\Gamma\fto \CG$ and $\Gamma\fto \CH$.
				\item[(ii)] There exists a $(\CG,\CH)$-bitorsor $P$.
				\item[(iii)] $\CG$ and $\CH$ are Morita equivalent ({Def. \ref{def:MEgroid}}).
			\end{enumerate}
		\end{prop}
		
		The proof of this statement can be found in \cite{MdMkGrd} and \cite[prop. 2.4]{LauMatXu}. We review its proof here.  
		
		\begin{proof}
			$(i)\Rightarrow (ii)$: Consider a Lie groupoid $\Gamma\soutar K$ with weak equivalences $\widehat{\pi}_M\colon\Gamma\fto \CG$ and $\widehat{\pi}_N\colon\Gamma\fto \CH$. Therefore $\Gamma\cong\pi_M^{-1} \CG\cong \pi_N^{-1} \CH$. We get that $Q_\CG:=\CG {}_{\bs}\!\times_{\pi_M} K$ is a $(\CG,\Gamma)$-bitorsor. Using a similar argument we get that $Q_\CH$ is a $(\Gamma,\CH)$-bitorsor.   The (not necessarily Hausdorff) manifold $Q:=(Q_\CG\times_K Q_\CH)$ 
			{has a diagonal $\Gamma$-action with the canonical map to  $K$ {as moment map}. Applying Lemma \ref{lem:bitorsor} to the map $Q\fto Q_\CG$ we see that
				$$P:=(Q_\CG\times_K Q_\CH)/\Gamma$$ is a (not necessarily Hausdorff) manifold.} One can  check that it is a $(\CG,\CH)$-bitorsor. 
			
			$(ii)\Rightarrow (iii)$:  Consider a $(\CG,\CH)$-bitorsor $P$, with moment maps $\pi_M\colon P\fto M$ and $\pi_N\colon P\fto N$. 
			{Then $$\Gamma= \CG {}_\bs\!\times_{\pi_M} P {}_{\pi_N}\!\times_\bt \CH$$  has a} natural structure of Lie groupoid over $P$ with $\bt(g,p,h)=gph$, $\bs(g,p,h)=p$ and multiplication given canonically by $\CG$ and $\CH$. Then the maps $\Gamma\fto \pi_M^{-1}G; \;(g,p,h)\mapsto (p,g^{-1},gph)$ and $\Gamma\fto \pi_N^{-1} H;\;(g,p,h)\mapsto (p,h,gph)$ are isomorphisms of Lie groupoids.
			
			This shows that $\pi_M^{-1}\CG\cong \pi_N^{-1} \CH$ as Lie groupoids over {the} not necessarily Hausdorff manifold $P$. Now take  a Hausdorff cover $\{U_i\}_{i\in I}$ of $P$ and let $\tilde{P}:=\sqcup_i U_i$. There is a canonical submersion $\pi\colon \tilde{P}\fto P$. It is easy to see that $(\tilde{P},\pi_M\circ \pi, \pi_N\circ \pi)$ is a Morita equivalence. 
			
			$(iii)\Rightarrow (i)$: Given  a Morita equivalence $(P,\pi_M,\pi_N)$ between $\CG$ and $\CH$, call $\Gamma:=\pi_M^{-1}\CG\cong \pi_N^{-1}\CH$. The {natural projections $\Gamma\fto \CG$ and $\Gamma\fto \CH$ are weak equivalences}.
		\end{proof}
		
		\begin{rem}\label{rem:topopenME}
			{Proposition \ref{prop:AllME} also holds for open topological groupoids,
				as can be proven using Lemma \ref{lem:openmapABC}}. {For arbitrary topological groupoids, this is not the case.}
		\end{rem}
		
		\begin{cor}\label{cor:SConnMEGrpd2}
			{Let $k\ge0$.} If $\CG\soutar M$ and $\CH\soutar N$ are source $k$-connected Morita equivalent {Hausdorff} Lie groupoids, then {there exists a Hausdorff $(\CG,\CH)$-bitorsor $P$}.
			Moreover this bitorsor is a Morita equivalence with $k$-connected fibres ({in the sense of Def. \ref{def:MEgroid})}.
		\end{cor}
		
		\begin{proof}   Following the implications $(iii)\Rightarrow (i)\Rightarrow (ii)$ in the proof of proposition \ref{prop:AllME}, one sees that the bitorsor $P$ constructed there is Hausdorff. {Then use the implication $(ii)\Rightarrow(iii)$ to prove that $P$ is a Morita equivalence.}
			
			{Note that} $P$ being a bitorsor, the fibres of $\pi_M\colon P\fto P/\CH{\cong M}$ are {diffeomorphic} to the source fibers of $\CH$, {which are $k$-connected by assumption}. A similar argument holds for the fibres of $\pi_N\colon P\fto P/\CG\cong N$.
		\end{proof}

		\begin{rem}
			The Morita equivalence $P$   in corollary \ref{cor:SConnMEGrpd2} is a (global) bisubmersion for the underlying foliations, as we now show.
			Using the implication ``$(ii)\Rightarrow (iii)$'' in proposition \ref{prop:AllME} we get an isomorphism of {Lie} groupoids $$\pi_M^{-1}\CG\cong \pi_N^{-1} \CH\cong \CG\ltimes P\rtimes \CH.$$ 
			Denote by $\CF_M$ and $\CF_N$ the  foliations underlying $\CG$ and $\CH$. Using {lemma \ref{GtoA} i)} and lemma \ref{lem:pullbackalgfol} we get that the foliations  underlying $\pi_M^{-1}\CG$ and $\pi_N^{-1} \CH$ are $\pi_M^{-1}\CF_M$ and $\pi_N^{-1}\CF_N$ respectively.
			Since $P$ is a {$(\CG,\CH)$-bitorsor,} the foliation underlying  the Lie groupoid $\CG\ltimes P\rtimes \CH$ is $\Gamma_c(ker(d\pi_M))+\Gamma_c(ker(d\pi_N))$. Hence $$\pi_M^{-1}\CF_M=\pi_N^{-1}\CF_N=\Gamma_c(ker(d\pi_M))+\Gamma_c(ker(d\pi_N)).$$
		\end{rem}
		
		\subsection*{Morita equivalence for Lie algebroids}
		The definition of Morita equivalence was extended to Lie algebroid by Viktor Ginzburg  \cite{GinzburgGrot}:
		
		\begin{defi}\label{def:MEalgoid}
			Consider Lie algebroids $A_M$ and $A_N$ over manifolds $M$ and $N$ respectively. We say they are \textbf{Morita equivalent} if there exists a manifold $P$ and two surjective submersions $\pi_M\colon P\fto M$ and $\pi_N:P\fto N$ with  simply connected fibres such that $\pi_M^{-1}( A_M)\cong \pi_N^{-1}(A_N)$ as Lie algebroids over $P$.
		\end{defi}
		This definition can be motivated by the following proposition:
		
		\begin{prop}\label{prop:MEalggroids}
			(i) If $\CG_M$ and $\CG_N$ are Morita equivalent {Hausdorff} Lie groupoids with source simply connected fibres, then their Lie algebroids are Morita equivalent.
			
			(ii) If $A_M$ and $A_N$ are Morita equivalent integrable Lie algebroids, then the  source simply connected Lie groupoids integrating them are Morita equivalent.
		\end{prop}
		\begin{proof}
			(i): By corollary \ref{cor:SConnMEGrpd2} we {obtain the existence of} a Morita equivalence with source simply connected fibres, then using part (1) of lemma \ref{GtoA} we get the desired result.
			
			(ii): is clear by part (2) of lemma \ref{GtoA}.
		\end{proof}

		\begin{rem}
			The essential difference between  Morita equivalence for Lie algebroids and {Hausdorff Morita equivalence} for singular foliations (Def. \ref{def:MEalgoid} and Def. \ref{def:defMEfol}) is that the former requires \emph{simply} connected fibres whereas the latter only connected fibres. This difference is reflected at the groupoid level too, as we now explain.
			
			On the one hand, given two Morita equivalent integrable Lie algebroids, their source simply connected Lie groupoids are Morita equivalent, see proposition \ref{prop:MEalggroids}(ii). On the other hand, singular foliations  also have an associated groupoid, namely the holonomy groupoid defined by Androulidakis and Skandalis  \cite{AndrSk}. In theorem \ref{thm:MEfolgroids} we will show that if two singular foliations are {Hausdorff} Morita equivalent, then their holonomy groupoids are also Morita equivalent. But the holonomy groupoid of a foliation does not have  simply connected fibres in general: on the contrary, it is an adjoint groupoid.
		\end{rem}
		
		Recall that Lie groupoids give rise to Lie algebroids, which in turn give rise to singular foliations (see  example \ref{ex:Algd,Gprd} and proposition \ref{prop:Folie.Algd}). The following proposition gives a link between Morita equivalence in these three different settings:
		
		\begin{prop}\label{prop: implications}
			Let $\CG\soutar M$ and $\CH\soutar N$ $k$-connected {Hausdorff} Lie groupoids for $k\geq 1$, denote their Lie algebroids by $A_M$ and $A_N$, and denote by    $\CF_M=\#(\Gamma_c(A_M))$ and $\CF_N=\#(\Gamma_c(A_N))$ the corresponding singular foliations. Each of the following statements implies the following one (i.e.  $(i)\Rightarrow (ii)   \Rightarrow (iii)$):
			\begin{enumerate}
				\item[(i)]  $\CG$ and $\CH$ are Morita equivalent,
				\item[(ii)]  there exists a  manifold $P$ and surjective submersions with $k-connected$ fibres $\pi_M\colon P\fto M$ and $\pi_N:P\fto N$ satisfying $\pi_M^{-1}A_M\cong \pi_N^{-1} A_N$,
				\item[(iii)] the foliated manifolds $(M,\CF_M)$ and $(N,\CF_N)$ are Hausdorff Morita equivalent.
			\end{enumerate}    
		\end{prop}
		\begin{proof}
			For the first implication, use corollary \ref{cor:SConnMEGrpd2} to get $P$, then we use   twice   part (i) of lemma \ref{GtoA}.
			The second implication   follows using twice  lemma \ref{lem:pullbackalgfol}. 
		\end{proof}
		
		\begin{rem}
			If $A_M$ and $A_N$ are Morita equivalent Lie algebroids then their singular foliations are Morita equivalent. Indeed, the fibres of the maps appearing in Def. \ref {def:MEalgoid} are in particular connected, so the above proposition applies.\end{rem}
		
		\begin{ex} { 
				Consider two {connected} Lie groups $G_1,G_2$ acting freely and properly on a manifold $P$ with commuting actions. The transformation groupoids $G_2\ltimes (P/G_1)$ and $G_1\ltimes (P/G_2)$ are Morita equivalent, since under the quotient maps {$P\to P/G_i$} they pull back
				to the transformation groupoid  $(G_1\times G_2)\ltimes P$. Hence applying proposition \ref{prop: implications} we obtain an alternative proof of the statement of corollary \ref{cor:action}, i.e. the respective foliations are Hausdorff Morita equivalent.}
			
			Moreover, if $G_1,G_2$ are simply connected, then their corresponding Lie algebroids are Morita equivalent.
		\end{ex}

		\cleardoublepage

		\chapter{Holonomy groupoids and Morita equivalence for foliations}\label{ch:conclusion}
		
		In this chapter we summarize the results of article \cite{ME2018} by Marco Zambon and I. It is the most important chapter of this thesis.
		
		We start this chapter describing the notion of holonomy groupoid of a singular foliation given in \cite{AndrSk}. Later we discuss smoothness issues of this groupoid. In the third section we will prove that the pullback groupoid of the holonomy groupoid, under certain maps, is the holonomy groupoid of the pullback foliation. Then we will give the notion of holonomy transformation for elements in the holonomy groupoid.
		
		To finish this chapter we show how Hausdorff Morita equivalence of singular foliations is related to Morita equivalence of their associated holonomy groupoids.
		
		\section{The holonomy groupoid of a singular foliation}\label{sec:holconstr}
		
		We review the   construction of the holonomy groupoid of a foliated manifold, which was first introduced by Androulidakis and Skandalis in \cite{AndrSk}. For this section we follow \cite[\S 2, \S 3.1]{AndrSk}. We also present some original results on the underlying topology of this groupoid.
		
		In this whole section we fix a foliated manifold $(M,\CF)$. The "building blocks" for the holonomy groupoid of $\CF$ are the path-holonomy bisubmersions (definition \ref{def:bisub}).
		
		Let $(M,\CF)$ be a foliated manifold, $x_0\in M$ a point and $X_1,\dots, X_k\in \CF$ a collection of vector fields such that their modular classes under $I_{x_0} \CF$, which we denote by $[X_1] ,\cdots,[X_k]$, are a basis for $\CF_{x_0}:= \CF/I_{x_0} \CF$. Recall that the path holonomy bisubmersions are given by an open neighborhood $W\subset \RR^k\times M$ of $(0,x_0)$, with $\bs=\mathrm{Pr}_M$ and $\bt$ given the flows of $X_1,\cdots,X_k$. 
		
		Note that neighbourhoods of points of the form $(0,x)$ in a path holonomy bisubmersion carry the identity diffeomorphism on $M$ and can be embedded in any bisubmersion that also carries the identity diffeomorphism. This follows from  {proposition} \ref{path} (ii) and corollary \ref{carry} (ii).
		
		The holonomy groupoid will be given by a family of path holonomy bisubmersions, glued together by an equivalence relation, which we describe as follows:
		
		\begin{defi}\label{def:eq.rel}
			Let $(M,\CF)$ be a foliated manifold, $(U_1,\bt_1,\bs_1)$ and $(U_2,\bt_2,\bs_2)$ bisubmersions for $\CF$, $u_1\in U_1$ and $u_2\in U_2$. We say that $u_1$ is \textbf{equivalent} to $u_2$ if there is a local morphism of bisubmersions $\varphi\colon U_1\supset U'_1\fto U_2$ with $\varphi(u_1)=u_2$.
		\end{defi}
		
		An equivalent description of this equivalence is given by the following lemma:
		
		\begin{lem}\label{prop:equiv2} Let $(M,\CF)$ be a foliated manifold and $(U_1,\bt_1,\bs_1)$, $(U_2,\bt_2,\bs_2)$ bisubmersions for $\CF$. The elements $u_1\in U_1$ and $u_2\in U_2$ are equivalent if and only if $U_1$ and $U_2$ carry the same local difeomorphism at $u_1$ and $u_2$ respectively.
		\end{lem}
		\begin{proof}
			It is a direct consequence of {corollary \ref{carry} (i)}.
		\end{proof}
		
		It is easy to see that definition \ref{def:eq.rel} gives an equivalence relation on any family of bisubmersions.

		\begin{defi}\label{def:atlas.bi}\label{def:adapted}
			Let $\CU=\{U_i\}_{i\in I}$ a family of bisubmersions of $\CF$.
			\begin{itemize}
				\item A bisubmersion $U'$ is \textbf{adapted} to $\CU$ if for any $u'\in U'$ there is $u\in U\in \CU$ which is equivalent to $u'$. A family of bisubmersions $\CU'$ is adapted to $\CU$ if any element $U'\in \CU'$ is adapted to $\CU$.
				\item We say that $\CU$ is an atlas if:
				\begin{enumerate}
					\item For all $x\in M$ there exists $U\in \CU$ that carries  the identity diffeomorphism locally at $x$.
					\item The inverses and finite compositions of elements in $\CU$ are adapted to $\CU$.
				\end{enumerate}
				\item Two atlases are \textbf{equivalent} if they are adapted to each other.
			\end{itemize}
		\end{defi}
		
		\begin{ex}\label{ex:UdeG}
			Let $\CG\soutar M$ be a Lie groupoid and $\CF_\CG$ its associated foliation. If $\CG$ is Hausdorff then $(\CG,\bt,\bs)$ is a bisubmersion and an atlas for $\CF_\CG$. Indeed through the identity bisection $e\colon M\fto \CG$ we have that $\CG$ satisfies condition 1. of definition \ref{def:atlas.bi}; and thanks to the inverse and composition of $\CG$, it also satisfies condition 2.
			
			If $\CG$ is not Hausdorff, take $\CU_\CG:=\{U_i\}_{i\in I}$ a Hausdorff cover for $\CG$. Then $\CU_\CG$ is an atlas of bisubmersions for $\CF_\CG$. 
		\end{ex}
		
		\begin{prop}
			[{\bf Groupoid of an atlas}]
			\label{prop:groidatlas} Let $(M,\CF)$ be a foliated manifold and $\CU=\{(U_i,\bt_i,\bs_i)\}_{i\in I}$ an atlas of bisubmersions for $\CF$.
			\begin{enumerate}
				\item[(i)] Denote the set: 
				$$\CG(\CU):= \left\{\left(\sqcup_{i\in I} U_i\right) / \sim \right\},$$
				
				where $\sim$ is the equivalence relation given in definition \ref{def:eq.rel}.
				
				Let $Q=(q_i)_{i\in I}\colon \sqcup_i U_i \fto \CG(\CU)$ be the quotient map.
				\item[(ii)] There are maps $\bt,\bs\colon \CG(\CU)\fto M$ such that $\bs\circ q_i=\bs_i$ and $\bt\circ q_i=\bt_i$.
				\item[(iii)] There is a groupoid structure on $\CG(\CU)$ with set of objects $M$, source and target maps $\bs$ and $\bt$ defined above and such that $q_i(u)q_j(v)=q_{U_i\circ U_j}(u,v)$.
			\end{enumerate}
		\end{prop}
		
		Given a foliated manifold $(M,\CF)$ and an atlas of bisubmersions $\CU=\{U_i\}_{i\in I}$  for $\CF$, we endow $\CG(\CU)$ with the quotient topology, i.e.  the finest topology that makes the quotient map: $Q\colon\sqcup_i U_i \fto \CG(\CU)$ continuous.
		
		\begin{lem}\label{lem:openmap}
			Given a foliated manifold $(M,\CF)$ and  an atlas of bisubmersions $\CU=\{U_i\}_{i\in I}$ for $\CF$, the quotient map $Q\colon\sqcup_i U_i \fto \CG(\CU)$ is open.
		\end{lem}
		\begin{proof}
			We will prove that given an open subset $A$   of $\sqcup_i U_i$, the preimage $Q^{-1}(Q(A))$ is open.
			
			Take $x\in Q^{-1}(Q(A))$. {Denote by $U\in\CU$ the bisubmersion} such that $x\in U$. There exists $y\in A$ such that $Q(x)=Q(y)$. {Notice that $A$ itself is a bisubmersion. By the definition of the equivalence relation in proposition \ref{prop:groidatlas},} 
			there is a neighbourhood $U'\subset U$ of $x$ and a  morphism of bisubmersions $f\colon U'\fto {A}$ sending $x$ to $y$. {Hence $Q(U')\subset Q(A)$, or in other words }$U'\subset Q^{-1}(Q(A))$, therefore $x$ is an interior point of $Q^{-1}(Q(A))$.
		\end{proof}

		\begin{lem}\label{lem:adapt}
			If the atlas $\CU_1$ is adapted to $\CU_2$ then
			\begin{enumerate}
				\item[(i)] there is a {canonical injective open morphism of topological} groupoids $\varphi\colon \CG(\CU_1)\fto \CG(\CU_2)$,
				\item[(ii)] $\varphi$ is surjective if and only if $\CU_2$ equivalent to $\CU_1$. In that case $\varphi$ is an isomorphism of topological groupoids.
			\end{enumerate}
		\end{lem}
		\begin{proof}{The map $\varphi$ is induced by morphisms of bisubmersions.
				More precisely:} 
			there is a well defined map $\widehat{\varphi}\colon \sqcup_{U\in \CU_1}U \fto \CG(\CU_2)$, given by   {$u\mapsto [f(u)]$ where $f$ is any  morphism of bisubmersions from a neighbourhood of $u$ to a bisubmersion in $\CU_2$}. Following the same argument as in the proof of Lemma \ref{lem:openmap} we get that $\widehat{\varphi}$ is an open map. Also $\widehat{\varphi}$ factors through the quotient map $Q_1$, {yielding an open, continuous and injective map $\varphi$.}
			\begin{equation} \label{diag:cont1}
				\begin{tikzcd}\sqcup_{U\in \CU_1} U \arrow[d,"Q_1"] \arrow[dr,"\widehat{\varphi}"] & \\
					\CG(\CU_1) \arrow[r,"\varphi"]   & \CG(\CU_2).
				\end{tikzcd}
			\end{equation}
			This map is surjective if and only if for any $u_2\in U_2\in \CU_2$ there is an $u_1\in U_1\in \CU_1$ equivalent to it, i.e. if $\CU_2$ is adapted to $\CU_1$.
		\end{proof}

		\begin{defi}\label{def:pathholatlas}
			Let $(M,\CF)$ be a foliated manifold. {A {\bf path holonomy atlas} is an atlas generated by a family of path holonomy bisubmersions $\{(U_i,\bt_i,\bs_i)\}_{i\in I}$ such that $\cup\bs_i(U_i)=M$}. 
		\end{defi}
		
		{The following lemma implies easily corollary \ref{cor:phatlas} and corollary \ref{cor:adapted}, which together are the content of  \cite[Examples 3.4(3)]{AndrSk}.} 
		
		\begin{lem}\label{lem:smallelements} {Let $(M,\cF)$ be a foliated manifold.
				Let $U\subset \RR^n\times M$ be a path holonomy bisubmersion, and $\CV$ any atlas of bisubmersions for $\cF$. Then $U$ is adapted to $\CV$.}
		\end{lem}
		\begin{proof} We have to show that around any point of $U$ there is a locally defined morphism of bisubmersions to an element of $\CV$, see definition \ref{def:adapted}.
			
			Let $u=(v,x)\in U$. {If $v=0$ we can simply apply Remark \ref{rem:invertmorph}, so in the following we assume $v\neq 0$.}
			Denote by $X_1,\dots,X_n$ the vector fields in $\cF$ used to construct the path holonomy bisubmersion $U$.  Extend $v$ to a basis $v^1:=v, v^2,\dots,v^n$ of $\RR^n$, and consider the path holonomy bisubmersion $\tilde{U}$ given by the local generators $\sum_i v_i^1X_i,\dots, \sum_i v_i^nX_i$ of $\cF$. The points $u\in U$ and $((1,0,\dots,0),x)\in \tilde{U}$ are equivalent by {corollary \ref{carry}(i)}, since the constant bisections through them carry the same diffeomorphism. Hence in the rest of the proof we can assume that $v=(1,0,\dots,0)$.
			
			Since $X_1$ is compactly supported and hence complete, we can consider the path $\gamma\colon [0,1]\to M, \gamma(h)=\mathrm{exp}_x(hX_1)$, where $\mathrm{exp}$ denotes the time one flow. For every $h\in [0,1]$, apply the diffeomorphism $\exp(hX_1)$ to $X_1,\dots,X_n$. This yields    elements of $\cF$, the first one being $X_1$, which form a generating set for $\cF$ near $\gamma(h)$. Denote by $U_h$  the path holonomy bisubmersion  they give rise to. 
			
			By remark \ref{rem:invertmorph} there exists an open neighborhood $U_h'$ of $(0,\gamma(h))$ and a morphism of bisubmersions from $U_h'$ to a bisubmersion in $\CV$.
			Shrinking $U_h'$ if necessary, we can assume that it is of the form ${B_{{r_h}}}\times M'_h$ where  $B_{{r_h}}\subset \RR^n$ is the open ball with radius $r_h$ and $M'_h\subset M$.
			
			By the compactness of $[0,1]$, there are finitely many $h_1,\dots,h_k\in [0,1]$ such that $M'_{h_1},\dots,M'_{h_k}$ cover the image of $\gamma$. Hence
			there is a positive integer $N$ such that, for all $h\in [0,1]$, the point $(\frac{1}{N}v,\gamma(h))$ is contained in one of the $U_{h_i}'$.
			The composition 
			\begin{equation}\label{eq:comp}
				\left(\frac{v}{N}, \gamma\Big(\frac{N-1}{N}\Big)\right)\circ\dots\circ \left(\frac{v}{N}, \gamma\Big(\frac{1}{N}\Big)\right)\circ\left(\frac{v}{N},x\right)
			\end{equation} is well-defined\footnote{For instance, $\bt((\frac{v}{N},x))=
				exp_x(\frac{1}{N}X_1)=\gamma(\frac{1}{N})$.}.
			Further, it is equivalent to $u=(v,x)\in U$ since the constant bisections through {
				$u$ and through the composition \eqref{eq:comp}} 
			both carry the  diffeomorphism  $exp(X_1)$. Since each of the elements we are composing in \eqref{eq:comp} lies in the domain of a morphism of bisubmersions  to a bisubmersion in $\CV$, the composition also does.
		\end{proof}
		
		\begin{cor}\label{cor:phatlas}
			\begin{enumerate}
				\item [(i)] A {path holonomy  atlas} is adapted to any  atlas. 
				\item [(ii)] {Any two path holonomy atlases are equivalent. Therefore any of them defines the same (up to isomorphism) topological groupoid.} 
			\end{enumerate}
		\end{cor}
		\begin{proof}
			{(i) 
				A path holonomy atlas consists of  finite compositions of path holonomy bisubmersions. Hence the statement follows from lemma \ref{lem:smallelements}.}

			(ii)  {An immediate consequence of part (i) is that any pair of path holonomy atlases are adapted to each other, therefore they define the same  topological groupoid.}
		\end{proof}
		
		\begin{defi}\label{def:holgroidph}
			Let $(M,\CF)$ be a foliated manifold. The groupoid over $M$ associated to a path holonomy atlas, as in Proposition \ref{prop:groidatlas}, is called {\bf holonomy groupoid} and  is denoted $\CH(\CF)$.
		\end{defi}

		\begin{cor}\label{cor:adapted}
			There exists a {canonical} injective open {morphism of topological groupoids}
			\[\varphi\colon \CH(\CF)\fto \CG,\]
			where $\CG$ is any groupoid given by an atlas of bisubmersions for $\CF$ as in Proposition \ref{prop:groidatlas}.
		\end{cor}
		\begin{proof}
			This follows from   corollary \ref{cor:phatlas} (i)  {applied to the atlas defining $\CG$, and from   {lemma \ref{lem:adapt} (i)}.}
		\end{proof}
		
		Because of corollary \ref{cor:adapted}, we can see the holonomy groupoid as an open neighborhood of the identity for any groupoid $\CG$ given by an atlas of bisubmersions. We will later prove that it is the source connected component of $\CG$.
		
		\begin{defi}\label{def:scon.bisub} Let $\CS$ be a family of source connected bisubmersions such that $\cap_{U\in \CS} \bs(U)=M$ and for every $u\in U\in \CS$ there is an element $e_u\in U$ carrying the identity {diffeomorphism} nearby $\bs(u)$. The atlas $\CU$ generated by $\CS$ is called a \textbf{source connected atlas}.	
		\end{defi}
		
		In particular any path holonomy atlas is a source connected atlas.
		
		\begin{lem}\label{lem:sconnholgroupid}
			Take $\CU$ a source connected atlas, then the groupoid $\CG(\CU)$ is source connected.
		\end{lem}
		\begin{proof}
			Take $\CS$ the family of bisubmersion of definition \ref{def:scon.bisub} with atlas $\CU$. Denote by $Q\colon \sqcup_{U\in \CU} U\fto \CG(\CU)$ the {(surjective)} quotient open map.
			
			{ Note that, by hypothesis for all $U\in\CS$, the points of $Q(U)$ \emph{can be connected to the identity through {a continuous path in} an $\bs$-fiber of $\CG(\CU)$}. Now we  prove {the same statement for any point of} $Q(U_k\circ\dots\circ U_1)$, where {$k\ge 2$} and $U_k,\dots ,U_1\in \CU$.}
			
			By induction, suppose {that the statement holds for all points of $Q(U_{k-1}\circ\dots\circ U_1)$}.
			Take  ${u:=}u_k\times\dots \times u_1 \in U_k\circ \dots \circ U_{1}$, {and denote $p:=\bs(u_k)\in M$.} By hypothesis there exists a curve $\gamma(t)$ in {a source fibre of} $U_k$ {joining $u_k$ with $u_p\in U_k$ that carries the identity at $p$.} Then $Q(\gamma(t)\times\dots \times u_1)$
			{is a curve in an $\bs$-fibre of $\CH(\cF)$
				that connects $Q(u)$ with an element of $Q(U_{k-1}\circ\dots\circ U_1)$. Hence the statement holds for $Q(U_k\circ U_{k-1}\circ\dots\circ U_1)$.}
		\end{proof}
		
		\begin{cor}\label{cor:scon.op.hol}
			The holonomy groupoid of a foliated manifold $(M,\CF)$ is a source connected, open topological groupoid.
		\end{cor}
		\begin{proof}
			By   Lemma \ref{lem:sconnholgroupid} any path holonomy atlas $\CU$, which is a source connected atlas, gives rise to a source connected groupoid, therefore $\CH(\CF):=\CG(\CU)$ is source connected.
			
			To prove that $\CH(\CF)$ is an \textbf{open} groupoid, denote by $Q\colon \sqcup_{U\in \CU} U\fto \CH(\CF)$ the quotient map. 
			The following diagram commutes, {where we denote by $\bs_H$ the source map of the holonomy groupoid}:
			\[\begin{tikzcd}
				\sqcup_{U\in \CU} U\arrow[dr," {\bs}"] \arrow[d,swap,"Q"] & \\
				\CH(\CF) \arrow [r,swap,"\bs_H"] & M 
			\end{tikzcd}\]
			{Recall that $\CH(\CF)$ is endowed with the quotient topology}.
			Using that $Q$ is continuous {and surjective}, and that $ {\bs}$ is a submersion and therefore   an open map, it  follows that $\bs_H$ is open map. A similar argument can be used for $\bt_H$.
		\end{proof} 
		
		A new result is the following:
		\begin{prop}\label{prop:sconn.grpd} If $\CU$ is a source connected atlas then $\CG(\CU)\simeq \CH(\CF)$.
		\end{prop}
		\begin{proof} By Cor.  \ref{cor:adapted} 
			we   get the existence of a natural injective open morphism {of topological groupoids} {$\varphi\colon \CH(\CF)\fto \CG(\CU)$}. 
			
			It is sufficient to show that $\varphi$ is surjective. Note that by lemma \ref{lem:sconnholgroupid}, $\CH(\CF)$ and $\CG(\CU)$ are $\bs$-connected. {It is a general fact that} any $\bs$-connected topological groupoid is generated by any symmetric neighbourhood of the identities \cite{MK2}. Using that $\varphi$ is an open morphism covering the identity we get that $\CG(\CU)$ is generated by the image of $\varphi$, which implies that $\varphi$ is surjective.	
		\end{proof}
		
		Moreover, there is a minimality property for the holonomy groupoid, given by the following corollary:
		
		\begin{cor}\label{rem:hol.sconn} Let $\CG\soutar P$ be a source connected Lie groupoid and $\CF_{\CG}$ its associated foliation. Then  $\CH(\CF_\CG)\cong \CG/\sim$  where the latter equivalence relation identifies two points when they carry the same local diffeomorphism. This implies that there is a surjective morphism $\CG\fto \CH(\CF_\CG)$.
		\end{cor}
		\begin{proof} Let $\CS_{\CG}$ be a Hausdorff cover of a neighborhood of the identity bisection in $\CG$ such that $e(\bs(U))\subset U$ for all $U\in \CS_{\CG}$. Denote $\CU_{\CS}$ the atlas generated by $\CS_{\CG}$. {Then} {$\CU_{\CS}$ is a {source connected atlas}, so that $\CG(\CU_\CG) \cong \CH(\CF_\CG)$ by Prop. \ref{prop:sconn.grpd}}. One can show that $\CU_{\CS}$ is adapted to the atlas $\CU_{\CG}$ given in example \ref{ex:UdeG}, and therefore that $\CG(\CU_\CG)\cong \CG/\sim$. Consequently, $\CH(\CF_\CG)\cong \CG/\sim$.
		\end{proof}
		
		\begin{rem} A Lie groupoid is called effective if the only bisection carrying the identity is the identity bisection. Any Lie groupoid $\CG\soutar M$ can be turned into a effective open topological groupoid, by $\CG/\sim$ where the latter equivalence relation identifies two points when they carry the same local diffeomorphism. Then any effective source connected Lie groupoid is isomorphic to the holonomy groupoid of its associated singular foliation.	
		\end{rem}
		
		\begin{cor}\label{rem:hol.regfol} If $\CF$ is a regular foliation, then $\CH(\CF)$ coincides with the usual notion of holonomy groupoid given in Section \ref{sec:holconstr}.
		\end{cor}
		\begin{proof} Note that $\Pi(\CF)$ is a source connected atlas for $\CF$. By Cor.\ref{rem:hol.regfol}, it is sufficient to show that two homotopy classes of curves $[\gamma_1],[\gamma_2]\in \Pi(\CF)$ have the same holonomy transformation if and only if there exists a local map of bisubmersions sending $[\gamma_1]\mapsto[\gamma_2]$. This is a clear consequence of Proposition \ref{prop:eq.reg.hol}.
		\end{proof}
		
		\section{Smooth structure on the holonomy groupoid}\label{sec:smth.hol}
		
		The smooth structure on the holonomy groupoid, when it exists, is as follows:
		
		\begin{defi}\label{def:smoothgr}
			Given a foliated manifold $(M,\CF)$ and  an atlas of bisubmersions $\CU=\{U_i : i\in I\}$ for $\CF$, we say that $\CG(\CU)$ is {\bf smooth} if there exists a ({necessarily unique})  smooth structure on it that makes the quotient map $Q\colon \sqcup_i U_i \fto \CG(\CU)$ a submersion.
		\end{defi}
		
		It is easy to see that if $\CG(\CU)$ is smooth, then it is a Lie {groupoid}. One of the most important results of this section is theorem \ref{thm:smt.hol}, in which it is shown that the smooth structure of the holonomy groupoid does not depend on the chosen atlas. 
		
		With this aim, we check the compatibility of smooth structures on equivalent atlases, using the canonical maps of Lemma \ref{lem:adapt}.
		
		\begin{lem}\label{lem:smt.adapt}
			Consider two atlases of bisubmersions $\CU_1$ and $\CU_2$, with $\CU_1$ adapted to $\CU_2$:
			\begin{enumerate}
				\item[(i)] If there is a smooth structure on $\CG(\CU_1)$ and $\CG(\CU_2)$, then the {morphism of} groupoids $\varphi\colon \CG(\CU_1)\fto \CG(\CU_2)$ of Lemma \ref{lem:adapt} is an injective (smooth) submersion and a morphism of Lie groupoids.
				\item[(ii)] Assume that $\CG(\CU_1)$ is smooth and $\CU_2$ is adapted to $\CU_1$ (or equivalently that the map $\varphi\colon \CG(\CU_1)\fto \CG(\CU_2)$ is surjective). Then $\CG(\CU_2)$ is smooth and $\varphi$ is an isomorphism of Lie groupoids.
			\end{enumerate}
		\end{lem}
		\begin{proof}
			(i) Consider $x\in \CG(\CU_1)$ and $y\in \mathrm{img}(\varphi)\subset \CG(\CU_2)$. Because $\CU_1$ is adapted to $\CU_2$, there is a bisubmersion $U_x\in \CU_1$ and a smooth map $\varphi_x\colon U_x \fto \sqcup_{\CU_2} U$. Because $y\in \mathrm{img}(\varphi)$, there is a bisubmersion $U_y\in \CU_2$ and a smooth map $\phi_y\colon U_y \fto \sqcup_{\CU_1} U$. Because of the smooth structures on $\CG(\CU_1)$ and $\CG(\CU_2)$, the maps $Q_1:\sqcup_{\CU_1} U\fto \CG(\CU_1)$ and $Q_2:\sqcup_{\CU_2} U\fto \CG(\CU_2)$ are smooth submersions. Therefore we have the following commutative diagram:
			\[\begin{tikzcd}
				U_x \arrow[d,"i"] \arrow[drr,"\varphi_x", near start] & & U_y \arrow[d,"i"] \arrow[dll,crossing over,swap,"\phi_y", near start] \\
				\sqcup_{\CU_1} {U} \arrow[d,"{Q_1}"]   &   & \sqcup_{\CU_2} U  \arrow[d,"Q_2"] \\
				\CG(\CU_1) \arrow[rr,"\varphi"] &  &   \CG(\CU_2)  \\
			\end{tikzcd}\]
			Because of the commutativity of the diagram, $\varphi$ must be a smooth map and a submersion.
			
			(ii) Since $\varphi$ is a homeomorphism by lemma \ref{lem:adapt} (ii), we can use it to transport the smooth structure on $\CG(\CU_1)$ to $\CG(\CU_2)$. With this structure the map $\varphi$ is a submersion (it is a diffeomorphism).
			By the proof of statement (i), for all $y\in \sqcup_{\CU_2} U$, the following diagram commutes:
			\[\begin{tikzcd}
				& & U_y \arrow[d,"i"] \arrow[dll,crossing over,swap,"\phi_y", near start] \\
				\sqcup_{\CU_1} {U} \arrow[d,"{Q_1}"]   &   & \sqcup_{\CU_2} U  \arrow[d,"Q_2"] \\
				\CG(\CU_1) \arrow[rr,"\varphi"] &  &   \CG(\CU_2)  \\
			\end{tikzcd}\]
			Here $Q_1\circ \phi_y$ and $\varphi$ are submersions, so $Q_2$ is a submersion.
		\end{proof}
		
		\begin{theorem}\label{thm:smt.hol}
			If for a source connected atlas $\CU$ we have that $\CG(\CU)$ is smooth, then the groupoids of any two source connected atlases are smooth and diffeomorphic. In other words, the smooth structure on the holonomy groupoid is well defined.
		\end{theorem}
		\begin{proof}
			This is a direct consequence of Lemma \ref{lem:smt.adapt}.
		\end{proof}
		
		\begin{corollary} The smooth structure on the holonomy groupoid of a regular foliation is well defined and it is diffeomorphic to the classical notion given in Section \ref{sec:holconstr}.
		\end{corollary}
		\begin{proof} This is a direct consequence of Corollary \ref{rem:hol.regfol}, Theorem \ref{thm:smt.hol} and the fact that the quotient map $Q\colon \Pi(\CF)\fto \CH(\CF)$ is a smooth submersion (indeed, it is a local diffeomorphism).
		\end{proof}
		
		Moreover, we have the following proposition, which is a direct consequence of \cite[\S 1.2]{AZ5}. It states explicitly when there is a smooth structure on the holonomy groupoid.
		
		\begin{prop}\label{prop:proy.fol}
			Given a foliated manifold $(M,\CF)$, we have that $\CH(\CF)$ is smooth if and only if $\CF$ is projective.
		\end{prop}
		
		\subsection*{Longitudinal smooth structure on the Holonomy groupoid}
		
		The holonomy groupoid of a foliated manifold $(M,\cF)$ is not always smooth, but by results of Claire Debord   \cite{Debord2013}, for any point $x\in M$ there is a smooth structure on the restriction of the holonomy {groupoid} to the leaf $L$ through $x$, {making it a Lie groupoid} (and  consequently on the isotropy group at $x$, {making it a Lie group}). More precisely, following  {\cite[Def. 2.8]{AZ1}, there exists a smooth structure on $H(\cF)_L$ -- the restriction of the holonomy groupoid to the leaf--, such that for any  path holonomy atlas $\{U_i\}$ of $\CF$, the quotient map $Q_L\colon \sqcup_i (U_{i})_L \fto H(\cF)_L$ is a submersion.} 
		
		\section{Pullback holonomy groupoid}\label{sec:pullb.grpd}
		
		In this section we show that the holonomy groupoid of the pullback foliation is the pullback of the holonomy groupoid of the foliation. {This fact will be used later in \S \ref{sec:me.hol.grpd}, where we will derive relations between Hausdorff Morita equivalence of singular foliations and Morita equivalence of their respective holonomy groupoids.}
		
		\subsubsection{An isomorphism of {topological} groupoids}
		
		We prove the following isomorphism of topological groupoids, which is one of the main results of this thesis.
		
		\begin{thm}\label{thm:pullbackgroid}
			Let $(M,\CF)$ be a foliated manifold and $\pi\colon P\fto M$ a surjective submersion {with connected fibres}.
			Then there is a canonical isomorphism of topological groupoids $$\CH(\pi^{-1}(\CF))\cong \pi^{-1}(\CH(\CF)).$$
		\end{thm} 
		{The pullback of a topological groupoid  is defined as in Def. \ref{def:pullbackgroid}}. For the sake of exposition, we first sketch a proof of Theorem \ref{thm:pullbackgroid} in the case of regular foliations.
		
		\begin{proof}[Proof of theorem. \ref{thm:pullbackgroid} for regular foliations]
			Assume that $\CF$ is a regular foliation, then $\pi^{-1}\CF$ is also a regular foliation. In this case the holonomy groupoid is given by holonomy classes of paths, as in Section \ref{sec:holconstr}. Define the {groupoid morphism}
			
			\begin{align*}
				\varphi\colon \CH(\pi^{-1}(\CF))&\fto \pi^{-1}(\CH(\CF))\\
				[\tau]\;\;&\mapsto(\tau(1),[\pi(\tau)],\tau(0)).
			\end{align*}
			
			We show that this map is  injective. Take two paths $\tau,\tilde{\tau}$ in leaves of $P$ with the same initial point $p_0$ and  final point $p_1$.
			For $i=0,1$, if $\Sigma_i$ is a transversal to $\pi^{-1}(\CF)$ at $p_i$ then $\pi(\Sigma_i)$ is a transversal of $\CF$ at $\pi(p_i)$. Let $\Phi, \tilde{\Phi}\colon\Sigma_0\fto\Sigma_1$ be the holonomy maps given by $\tau$ and $\tilde{\tau}$ respectively, and $\phi,\tilde{\phi}\colon \pi(\Sigma_0)\fto \pi(\Sigma_1)$ the holonomy maps given by $\pi(\tau)$ and $\pi(\tilde{\tau})$. The following diagram commutes:
			\[\begin{tikzcd}
				\Sigma_0 \arrow[d,"\pi"] \arrow[r,"\Phi"] & \Sigma_1 \arrow[d,"\pi"] \\
				\pi(\Sigma_0) \arrow[r,"\phi"] & \pi(\Sigma_1) \\ 
			\end{tikzcd},\]
			and the analog diagram for $\tilde{\Phi}, \tilde{\phi}$ too.
			The vertical maps $\pi\colon \Sigma_i \fto \pi(\Sigma_i)$ are diffeomorphisms (notice that the codimensions of $\CF$ and $\pi^{-1}\CF$ are equal).
			{Hence if $\pi(\tau)$ and $\pi(\tilde{\tau})$ have the same holonomy, i.e.  $\phi=\tilde{\phi}$, then $\Phi=\tilde{\Phi}$}.
			
			To prove the surjectivity of $\varphi$, take $(p,[\gamma],q)\in \pi^{-1}(\CH(\CF))$ where $p,q\in P$ and $\gamma$ is a curve   {in a leaf} of $\CF$ that connects $\pi(q)$ with $\pi(p)$. 
			The hypotheses on $\pi$ imply that $\text{Pr}_1\colon \gamma^{*}P:=[0,1] {}_\gamma \!\times_\pi P \fto [0,1]$ is a surjective submersion with connected fibres, hence  $\gamma^{*}P$ is a connected manifold and  therefore a path connected space. Take a curve  $\sigma\colon [0,1]\fto \gamma^{*}P$ that connects $(0,q)$ with $(1,p)$.
			We have the following commutative diagram:
			\[\begin{tikzcd}
				\left[0,1 \right] \arrow[r,"\sigma"] &\gamma^{*}P=[0,1] {}_\gamma \!\times_\pi P \arrow[d,"\text{Pr}_1"] \arrow[r,"\text{Pr}_2"] & P \arrow[d,"\pi"] \\
				&\left[0,1 \right] \arrow[r,"\gamma"] & M.
			\end{tikzcd}\]
			The curve
			$\widehat{\gamma}:= \text{Pr}_2\circ \sigma$ lies in a leaf of $\pi^{-1}(\cF)$ and joins $q$ with $p$. Since $\text{Pr}_1 \circ\sigma\colon [0,1]\fto[0,1]$ is a continuous and surjective function homotopic to the identity, from the commutativity of the  diagram it follows that $\pi\circ \widehat{\gamma}$ and $\gamma$ are homotopy equivalent and so holonomy equivalent. {Hence} $\varphi([\widehat{\gamma}])=(p,[\gamma],q)$,   proving that $\varphi$ is surjective and therefore bijective. 
		\end{proof}

		We now turn to the proof of Theorem \ref{thm:pullbackgroid}. {The first step is to state and prove proposition \ref{super},} which requires some preparation. We first focus on
		pullbacks of atlases {of bisubmersions}, which are relevant for
		the l.h.s. of the isomorphism claimed there.
		We state first   \cite[Lemma 2.3]{AndrSk}, which allows us to pull back bisubmersions.
		
		\begin{lem}\label{pullbi}
			Let $(M, \CF)$ be a foliated manifold, $(U, \bt, \bs)$ a bisubmersion for $\CF$ and $\pi\colon P\fto M$ a surjective submersion. Consider the preimages $P^\bs:=\pi^{-1}(\bs(U))$ and  $P^\bt:=\pi^{-1}(\bt(U))$. Define $$\pi^{-1}(U):=P^\bt {}_{\pi} \!\times_{\bt} U {}_\bs \!\times_\pi P^\bs.$$ Let $\tau,\sigma\colon \pi^{-1}(U)\fto P$ be the  projections {onto the first and third component. Then $(\pi^{-1}(U),\tau,\sigma)$} is bisubmersion for $\pi^{-1}(\CF)$.
		\end{lem}
		\begin{proof}
			The following 
			diagram commutes:
			\[
			\begin{tikzcd}
				\pi^{-1}(U) \arrow[d,"\text{Pr}_U"] \arrow[r,shift right=.3em,swap,"{\sigma}"] 
				\arrow[r,shift left=.3em,"\tau"] & P \arrow[d,"\pi"] \\
				U \arrow[r,shift right=.3em,swap,"\bs"] \arrow[r,shift left=.3em,"\bt"] & M
			\end{tikzcd}
			\]
			Moreover, since $\pi$ is a  submersion one can prove that  $\tau$ and $\sigma$ are submersions. {For the same reason $\text{Pr}_U$ is a submersion, and applying lemma \ref{lem:sub.bisub} {to it} we obtain that $(\pi^{-1}(U),\bt\circ \text{Pr}_U, \bs\circ \text{Pr}_U)$ is a bisubmersion for $\CF$. Using the commutativity of the diagram  we get that $(\pi^{-1}(U),\tau, \sigma)$ is a bisubmersion for $\pi^{-1}\CF$.}
		\end{proof}
		
		\begin{defi}
			We call the bisubmersion $\pi^{-1}(U)$ given in lemma \ref{pullbi} the \textbf{pullback bisubmersion} of $U$.
		\end{defi}
		\begin{lem}\label{lem:pullbackatlas}
			Let $\CU$ be an atlas of bisubmersions for $\CF$. Then $\pi^{-1} \CU:=\{\pi^{-1}(U) \st U\in\CU\}$ is an atlas of bisubmersion for $\pi^{-1}(\CF)$.
		\end{lem}
		\begin{proof}
			It is clear that the union of the elements of $\pi^{-1} \CU$ covers $P$. We now check that the compositions of elements in $\pi^{-1} \CU$ are adapted to $\pi^{-1} \CU$. 
			To do so, take $U_2,U_1\in \CU$. Note that we have a canonical morphism of {bisubmersions} 
			\begin{equation}\label{eq:morphcomp}
				\pi^{-1}(U_2)\circ \pi^{-1}(U_1)  \fto \pi^{-1}(U_2\circ U_1);\;\;(p,u_2,a)\times(a,u_1,q)\mapsto(p,u_2\times u_1,q).  
			\end{equation}
			Moreover, since $\CU$ is an atlas, at   each point $u_2\times u_1$ the bisubmersion $U_2\circ U_1$ is adapted to some $U_{u_1\times u_2}\in \CU$. This means that there is a small neighbourhood $V'\subset U_1\circ U_2$ containing $u_2\times u_1$, and a morphism of bisubmersions $f\colon V'\fto U_{u_1\times u_2}$. {Composing a suitable restriction of the morphism \eqref{eq:morphcomp} with the natural ``lift'' of $f$ we obtain
				a morphism of bisubmersions}  $$(p,u_2,a)\times(a,u_1,q)\mapsto(p,f(u_2\times u_1),q)\in \pi^{-1}(U_{u_1\times u_2})$$
			{into an element of $\pi^{-1}\CU$.}
			{For inverses of elements in $\pi^{-1} \CU$ one  proceeds  similarly}.
		\end{proof} 
		
		\begin{prop}\label{super}
			Let $\CU$ be an atlas of bisubmersions on a foliated manifold $(M,\CF)$, and denote by $\CG(\CU)$ the groupoid given by $\CU$. Let $\pi\colon P\fto M$ be a surjective submersion,  $\pi^{-1} \CU$ the pullback atlas, and denote by $\CG(\pi^{-1} \CU)$ the groupoid of this atlas. Then there is a canonical {isomorphism of topological groupoids} 
			$$\CG(\pi^{-1} \CU)\cong\pi^{-1}(\CG(\CU)).$$
		\end{prop}
		
		\begin{proof}
			{The quotient map $Q\colon\sqcup_{U\in\CU} U \fto \CG(\CU)$ lifts to
				a canonical map 
				\begin{equation}\label{map:pullbackatlas}
					Id\times Q\times Id\colon \left(\sqcup_{U\in\CU} \pi^{-1}U\right) \fto \pi^{-1} (\CG(\CU)),\hspace{0.2in} (p,u,q)\mapsto (p,[u],q),
				\end{equation}
				where we denote  $[u]:=Q(u)$.}
			We will show that this map factors through the projection 
			map associated to the atlas $\pi^{-1}\CU$, determining 
			a map  $\Phi\colon \CG(\pi^{-1}\CU)\fto \pi^{-1}(\CG(\CU))$,
			which moreover   is an isomorphism of topological groupoids. 
			\begin{equation}
				\begin{tikzcd}\label{diag:iQi}
					\left(\sqcup_{U\in\CU} \pi^{-1}U\right) \arrow[dr,"Id\times Q\times Id"] \arrow[d,"Q_\pi"]&  \\
					\CG(\pi^{-1}\CU) \arrow[r,dashed,"\Phi"] & \pi^{-1} (\CG(\CU))
				\end{tikzcd}  
			\end{equation}

			{The fact that $\Phi$ is well-defined and injective follows from the claim below (respectively, from the implications ``$\Rightarrow$'' and ``$\Leftarrow$'').} The surjectivity of $\Phi$ is clear because the map $Id\times Q\times Id$ given in \eqref{map:pullbackatlas} is surjective. The fact that $\Phi$ is a homeomorphism holds because both $Q_\pi$   and $Id\times Q\times Id$ are open maps, {by lemma \ref{lem:openmap}. The map $\Phi$ is a groupoid morphism as a consequence of proposition \ref{prop:groidatlas} {(iii)} and of the morphism of bisubmersions \eqref{eq:morphcomp}. Hence we are left with proving the following claim for all $ (p_0,u_0,q_0), (p_0,v_0,q_0)\in\sqcup_{U\in\CU} \pi^{-1}U$.}\\
			
			\noindent\underline{{ Claim:}}   {\it $Q_\pi(p_0,u_0,q_0)=Q_\pi(p_0,v_0,q_0)$ in $\CG(\pi^{-1}\CU)$ if and only if $[u_0]=[v_0]$ in $\CG(\CU)$}.
			
			``$\Rightarrow$'': 
			{By assumption} 
			there exist bisections $\sigma_u$ and $\sigma_v$ {through}  $(p_0,u_0,q_0)$ and $(p_0,v_0,q_0)$ respectively, carrying the same diffeomorphism {of $P$}. Since $\pi$ is a submersion, there exists a neighbourhood $W$ of $\pi(q_0)\in M$ and a $\pi$-section $q\colon W\fto P$, such that $q(\pi(q_0))=q_0$. Finally $\text{Pr}_2 \circ\sigma_u\circ q $ is a bisection {through} $u_0$ carrying the same diffeomorphism {of $M$} as the bisection $\text{Pr}_2\circ\sigma_v\circ q $ {through} $v_0$. Therefore $[u_0]=[v_0]$ in $\CG(\CU)$.
			
			``$\Leftarrow$'': Let {$ u_0\in U$ and $v_0\in V$ be equivalent points of} $\sqcup_{U\in \CU} U$, and let $p_0,q_0\in P$  lie in the fibre of $\bt(u_0)=\bt(v_0)$ and $\bs(u_0)=\bs(v_0)$ respectively. Then there exists a neighbourhood $U'$ of $u_0$ inside $U$ and a morphism {of bisubmersions} $f\colon U'\fto V$ {such that $f(u_0)=v_0$}.  
			{Lifting it}  we get a morphism of bisubmersions $$\widehat{f}\colon \pi^{-1}U'\fto \pi^{-1}V\hspace{.1in} ;\hspace{.1in} (p,u,q)\mapsto (p,f(u),q)$$ such that $\widehat{f}(p_0,u_0,q_0)=(p_0,v_0,q_0)$. This shows that $Q_\pi(p_0,u_0,q_0)=Q_\pi(p_0,v_0,q_0)$ in $\CG(\pi^{-1}\CU)$.
		\end{proof}
		
		{We now take the second step for the proof of Theorem \ref{thm:pullbackgroid}.}
		
		\begin{prop}\label{prop:injmap}
			{Let $\pi\colon P\fto M$ be a surjective submersion { with connected fibres} and $\cF$ a foliation on $M$.
				Let $\CU$ be a path holonomy atlas for $\CF$.}
			There is a canonical  {isomorphism of topological groupoids}
			\[\CH(\pi^{-1}(\CF))\cong \CG(\pi^{-1}(\CU))
			.\]
		\end{prop}
		\begin{proof}
			{Proposition  \ref{lem:pullbackatlas}} shows that $\pi^{-1}(\CU)$ is an atlas of bisubmersions for $\pi^{-1}\CF$. Because $\pi$ is source connected, we get that $\pi^{-1}(\CU)$ is a source connected atlas and by proposition \ref{prop:sconn.grpd} we get the desired result.
		\end{proof}

		\begin{proof}[Proof of theorem. \ref{thm:pullbackgroid}]
			{Let $\CU$ be a path holonomy atlas for $\CF$. 
				We have a composition of isomorphisms
				$$\CH(\pi^{-1}(\CF))\cong \CG(\pi^{-1}(\CU))\cong \pi^{-1}(\CH(\CF)),$$
				where the first isomorphism is the one obtained in proposition \ref{prop:injmap} and the second isomorphism is given by proposition \ref{super} using $\CH(\CF)=\CG(\CU)$.}
		\end{proof}
		
		\subsubsection{{Preservation of smoothness}}
		
		{The isomorphism of Theorem \ref{thm:pullbackgroid} preserves smooth structures, whenever they are present. We now elaborate on this. Recall that the smoothness was defined in Definition \ref{def:smoothgr}. That definition fits well with the isomorphism given in Proposition \ref{prop:injmap}, as follows:
			
			\begin{lem}\label{lem:smoothtwo}  {Let $\CU$ be an atlas on a foliated manifold $(M,\CF)$, and let $\pi\colon P\fto M$ be a surjective submersion. Assume that $\CG(\CU)$ is smooth. Then   $\CG(\pi^{-1} \CU)$ is also smooth,  and the map $\CG(\pi^{-1} \CU)\cong\pi^{-1}(\CG(\CU))$ in proposition \ref{super} is an isomorphism of Lie groupoids}.
			\end{lem}
			\begin{proof} Being the pullback of a Lie groupoid by a submersion,  $\pi^{-1}(\CG(\CU))$ is a Lie groupoid.
				Since the map in proposition \ref{super} is a homeomorphism, we can use it to transport the smooth structure on $\pi^{-1}(\CG(\CU))$ to $\CG(\pi^{-1}\CU)$.  
				Since the quotient map $Q$ onto  $\CG(\CU)$ is a submersion, it follows that the map $Id\times Q\times Id$ given in eq. \eqref{map:pullbackatlas} is a submersion too. 
				The commutativity of diagram \eqref{diag:iQi} implies that  $Q_{\pi}$ is a submersion  onto $\CG(\pi^{-1}\CU)$ endowed with the above smooth structure. The uniqueness in definition  \ref{def:smoothgr} finishes the argument.
			\end{proof}

			{The smooth version of Theorem  \ref{thm:pullbackgroid} is the following:}
			
			\begin{prop}\label{prop:pullLie}
				{Let $\pi\colon P\fto M$ be a surjective submersion with connected fibres and $\CF$ a foliation on $M$. If $\CH(\CF)$ is smooth (i.e. if $\CF$ is a projective foliation by proposition \ref{prop:proy.fol}) then the map $\varphi\colon \CH(\pi^{-1}(\CF))
					\cong\pi^{-1}(\CH(\CF))$ given in theorem. \ref{thm:pullbackgroid} {is an isomorphism of} \emph{Lie groupoids}.}
			\end{prop}
			\begin{proof}
				Let $\CU$ be a path holonomy atlas   of $\cF$. We check that the composition $$\CH(\pi^{-1}(\CF))\cong \CG(\pi^{-1}(\CU))\cong \pi^{-1}(\CH(\CF))$$ appearing in the proof of Theorem  \ref{thm:pullbackgroid} is a composition of Lie groupoid isomorphisms.
				The second map is a Lie groupoid isomorphism, by Lemma \ref{lem:smoothtwo}.
				
				The first map is a Lie groupoid isomorphism: apply Lemma \ref{lem:smt.adapt} (ii) to  $\CU_2:=\pi^{-1}\CU$ (which is an atlas of bisubmersion for $\pi^{-1}\CF$ by lemma  \ref{lem:pullbackatlas}), to a  path holonomy atlas  $\CU_1$  of $\pi^{-1}\cF$, and use that $\varphi$ is surjective (see  proposition \ref{prop:injmap}).
			\end{proof}
			
			Even when the holonomy groupoid is not smooth, the map $\varphi\colon \CH(\pi^{-1}(\CF))\cong\pi^{-1}(\CH(\CF))$ given in Theorem \ref{thm:pullbackgroid} preserves the longitudinal smooth structure, as the following theorem says:
			
			\begin{prop}\label{isotropull}
				Let $\pi\colon P\fto M$ be a surjective submersion with connected fibres, $\CF$ a foliation on $M$ and $\CU$ a path holonomy atlas for $\CF$.
				{The map $\varphi\colon \CH(\pi^{-1}(\CF)) \cong\pi^{-1}(\CH(\CF))$ {of theorem  \ref{thm:pullbackgroid}} restricts to the following isomorphisms of Lie groupoids:}
				\begin{enumerate}
					\item[(i)] $(\CH(\pi^{-1}(\CF)))_{\widehat{L}}\cong (\pi^{-1}(\CH({\CF})))_{\widehat{L}}$, {for the restrictions to any leaf $\widehat{L}\subset P$,}
					\item[(ii)] $(\CH(\pi^{-1}(\CF)))_p\cong (\pi^{-1}(\CH(\CF)))_p$, for the isotropy Lie groups at any $p\in P$.
				\end{enumerate} 
			\end{prop}
			
			\begin{rem}
				{There is canonical isomorphism of Lie groups} 
				$(\pi^{-1}(\CH(\CF)))_p\cong \CH(\CF)_{\pi(p)}$.
			\end{rem}
			\begin{proof}
				We prove only (i), since (ii) is a direct consequence. 
				Any leaf in $P$ is of the form $\widehat{L} =\pi^{-1}(L)$ for some leaf $L$ in $M$. We have 
				$$(\pi^{-1}(\CH(\cF)))_{\widehat{L}}= \widehat{L}{}_{\pi}\!\times_{\bt} (\CH(\cF)_{L}) {}_{\bs}\!\times_{\pi} \widehat{L}_{\pi}=
				\pi^{-1}(H(F)_L).$$
				Take a path holonomy atlas  $\CU=\{U_i\}_{i\in I}$ for $\CF$ and note that $(\pi^{-1} U_i)_{\widehat{L}}=\widehat{L}{}_{\pi}\!\times_{\bt} ((U_i)_L) {}_{\bs}\!\times_{\pi} \widehat{L}_{\pi}$. The map $Q_L\colon \sqcup_i (U_{i})_L \fto \CH(\cF)_L$ is a submersion, by the above definition of smooth structure on $ \CH(\cF)_L$, therefore the map $$Id\times Q_L\times Id\colon \sqcup_i (\pi^{-1}U_{i})_{\widehat{L}} \fto (\pi^{-1}\CH(\cF))_{\widehat{L}}$$ is a submersion.
				
				{This allows us to apply the arguments of the proof of lemma \ref{lem:smoothtwo}   to groupoids over $\widehat{L}$ (rather than over $P$). The proof of proposition \ref{prop:pullLie} delivers the desired conclusion.}
			\end{proof}
			
			\section{Holonomy transformations}\label{sec:hol.trans}
			
			{Given a regular foliation, a classical construction, recalled in Section \ref{sec:holconstr}, associates to every path in a leaf its holonomy (a germ of diffeomorphisms between slices transverse to the foliation). We review the extension of this construction to singular foliations
				\cite[\S 2]{AZ2} and show that it is invariant under   pullbacks.}

			\begin{defi}\label{holfrom}
				Let $(M,\cF)$ be a singular foliation, and $x,y\in M$ lying in the same leaf.  
				Fix a transversal $S_{x}$ at $x$, as well as a transversal $S_y$ at $y$. A \emph{holonomy transformation  from $x$ to $y$} is an element of 
				$$ \frac{GermAut_{\cF}(S_x, S_y)}{exp(I_x \cF)|_{S_x}}.$$ 
				Here $GermAut_{\cF}(S_x,S_y)$ is the space of germs at $x$ of locally defined diffeomorphisms preserving $\cF$   mapping $S_x$ to $S_y$, restricted to $S_x$.
				Further $exp(I_x \cF)|_{S_x}$
				is the space of germs at $x$ of time-one flows of time-dependent vector fields in $I_x \cF$ mapping $S_x$ to itself, restricted to $S_x$.
			\end{defi}
			
			Holonomy transformations are relevant because the holonomy groupoid maps canonically into them \cite[theorem. 2.7]{AZ2}.

			\begin{thm}\label{globalaction} Let $x, y \in (M,\cF)$ be points in the same leaf $L$, and fix  transversals $S_x$ at $x$ and $S_y$ at $y$. 
				Then there is a well defined map 
				\begin{align}\label{Phixy}
					\Phi^{\cF} \colon \CH(\cF)_x^y \rightarrow \frac{GermAut_{\cF}(S_x, S_y)}{exp(I_x \cF)|_{S_x}},\quad h \mapsto \langle \tau\rangle.   
				\end{align}
				
				Here $\tau$ is defined as follows, given $h \in \CH(\cF)_x^y{:=\bt^{-1}(y)\cap \bs^{-1}(x)}$:  
				\begin{itemize}
					\item take any bisubmersion $(U,\bt,\bs)$ in the path-holonomy atlas with a point $u\in U$ satisfying $[u]=h$,
					\item  take any section $\bar{b} \colon S_x \to U$  of $\bs$ through $u$ {transverse to the $\bt$-fibers} such that $(\bt\circ \bar{b})(S_x)\subset S_y$, 
				\end{itemize}
				and define $\tau:=\bt\circ \bar{b} \colon S_x \to S_y$. 
			\end{thm}
			For all $x,y$ the map  $\Phi^{\cF}$ is injective \cite[theorem. 2.20]{AZ2} {and assembles to a groupoid morphism \cite[theorem. 2.7]{AZ2}}.
			In the case of regular foliations, the map $\Phi^{\cF}$ describes the usual geometric notion of holonomy.
			
			\begin{rem}\label{rem:linholtrafo}
				{Linearizing any representative of $\Phi^{\cF}(h)$ one associates to $h$ a well-defined linear map $T_xS_x\to T_yS_y$. Notice that $T_xS_x$ can be identified with the normal space $N_xL$ to the leaf at $x$. Hence, when $x=y$, we obtain a representation of the isotropy Lie group $\CH(\cF_M)_x^x$ on $N_xL$ \cite[\S 3.1]{AZ2}.}
			\end{rem}

			Let $(M,\CF)$ be a foliated manifold and $\pi\colon P\fto M$ a surjective submersion with connected fibers. Recall that there is a canonical surjective morphism
			
			$$\widehat{\pi}\colon \CH(\pi^{-1}(\CF)) \cong\pi^{-1}(\CH(\CF)) \to \CH(\CF),$$ {where the isomorphism is given in Theorem \ref{thm:pullbackgroid}.}
			We now show that the holonomy transformations associated to a point in $\CH(\pi^{-1}(\CF)) $ and to its image coincide. 
			
			\begin{prop}\label{thm:holtrafo}

				For every  $h\in \CH(\pi^{-1}(\CF))$, the holonomy transformations associated to $h$ and    to $\widehat{\pi}(h)$ coincide, under the obvious identifications. More precisely: fix  slices  $S_{x}$ at $x:=\bs(h)\in P$ and $S_y$ at $y:=\bt(h)$, transverse to  $\pi^{-1}(\cF)$. Then
				$$\Phi^{\pi^{-1}(\cF)}(h) \in \frac{GermAut_{\pi^{-1}(\cF)}(S_x, S_y)}{exp(I_x \pi^{-1}(\cF))|_{S_x}}$$
				and
				$$\Phi^{\cF}(\widehat{\pi}(h)) \in \frac{GermAut_{\cF}(S_{\pi(x)}, S_{\pi(y)})}{exp(I_{\pi(x)} \cF)|_{S_{\pi(x)}}}$$
				coincide under the diffeomorphisms $S_x\cong S_{\pi(x)}:=\pi(S_x)$ and $S_y\cong S_{\pi(y)}:=\pi(S_y)$ obtained restricting $\pi$.
			\end{prop} 
			\begin{proof}
				Let  $h\in \CH(\pi^{-1}(\CF))$. By theorem. \ref{globalaction}, 
				${\Phi^{\pi^{-1}(\cF)}(h)}$
				is obtained using  a bisubmersion $V$ in the path-holonomy atlas of $(P,\pi^{-1}(\cF))$, a point $v\in V$ with $[v]=h$, and 
				a certain section through $v$. {By  proposition \ref{prop:injmap}}  
				the groupoid $ \CH(\pi^{-1}(\CF))$ {is isomorphic to  $\CG(\pi^{-1}(\CU))$, which} is constructed out of the atlas $\pi^{-1}(\CU)$ where $\CU$ is a path-holonomy atlas for $(M,\cF)$. 
				{This means that} there is a bisubmersion $U$ in $\CU$ and a morphism of bisubmersions $$\psi\colon V\to \pi^{-1}(U)$$ defined near $v$. We have $\psi(v)=(y,u,x)\in \pi^{-1}(U)$ for some $u\in U$. Further, applying $\psi$ to any bisection of $V$ we obtain a bisection of 
				$\pi^{-1}(U)$ carrying the same diffeomorphism. Hence we can work on the latter bisubmersion  instead of on $V$.
				
				Take any section $\bar{b} \colon S_x \to \pi^{-1}(U)$ of $\bs$ through $(y,u,x)$  {transverse to the $\bt$-fibres} such that $(\bt\circ \bar{b})(S_x)\subset S_y$.
				{Due to the diffeomorphism $S_x\cong S_{\pi(x)}$}, there is a unique section $b\colon S_{\pi(x)}\to U$ through $u$ such that $$\bar{b}(p)=(*,b(\pi(p)),p)$$ for any $p\in S_x$. (Here $*$ denotes the unique point of $S_y$ that corresponds to $(\bt\circ b)(\pi(p))$ under the identification $S_y \cong  S_{\pi(y)}$.)
				The diffeomorphisms $$\bt\circ \bar{b}\colon S_x \to S_y \;\;\text{      and      }\;\;
				\bt\circ  {b}\colon S_{\pi(x)} \to S_{\pi(y)}$$ coincide under the natural identification between slices. The former is a representative of $\Phi^{\pi^{-1}(\cF)}(h)$, while the latter is a representative  of $\Phi^{\cF}([u])$. {We conclude noticing that $[u]=\widehat{\pi}(h)$, as can be seen using the proof of} proposition \ref{super}.
			\end{proof}
			
			\section{Morita equivalent holonomy groupoids}\label{sec:me.hol.grpd}
			
			{The holonomy groupoid of a singular foliation (see definition \ref{def:holgroidph}) is not a Lie groupoid in general, but just an open topological groupoid. {Recall that we defined Morita equivalence for open topological groupoids in Section \ref{sec:ME.gpd}, more precisely in { Definition \ref{def:MEgroid}.}
					
					{We now can finally state the main results of this thesis:}
					
					\begin{thm}\label{thm:MEfolgroids}
						Given two foliated manifolds $(M,\CF_M)$ and $(N,\CF_N)$. If they are Hausdorff Morita equivalent then:
						
						\begin{itemize}
							\item Their holonomy groupoids $\CH(\CF_M)$ and $\CH(\CF_N)$ are Morita equivalent as open topological groupoids.
							\item For $L_M$ a leaf of $\CF_M$ and $L_N$ the corresponding leaf of $\CF_N$, we have that $\CH(\CF_M)|_{L_M}$ and $\CH(\CF_N)|_{L_N}$ are Morita equivalent as Lie groupoids. (Therefore the isotropy Lie groups are isomorphic).
							\item If $\CF_M$ is projective (for example a regular foliation), then $\CF_N$ is projective and the Lie groupoids $\CH(\CF_M)$ and $\CH(\CF_N)$ are Morita equivalent as Lie groupoids.
						\end{itemize} 
					\end{thm} 
					
					\begin{proof}
						{For the first part apply twice theorem \ref{thm:pullbackgroid}, noticing that submersions are open maps.}
						
						For the second part use twice Proposition \ref{isotropull}.
						
						For the projective case, Proposition \ref{prop:proy.fol} and Proposition \ref{prop:pullLie} give us the desired result.
					\end{proof}
					
					Furthermore, combining   with  Corollary \ref{cor:SConnMEGrpd2} we obtain:
					
					\begin{prop}\label{prop:equivproj}
						Provided  their holonomy groupoids are Hausdorff, two projective  foliations  are Hausdorff Morita equivalent  if{f} their holonomy groupoids  are Morita equivalent as  Lie groupoids. 
					\end{prop}
					
					\subsection*{Revisiting Hausdorff Morita equivalence of singular foliations}
					
					In the definition of {Hausdorff} Morita equivalence between two singular foliations  $(M,\cF_M)$ and $(N,\cF_N)$, Def. \ref{def:defMEfol}, it is required that the maps
					$\pi_M\colon P\fto M$ and $\pi_N\colon P\fto N$ be \emph{surjective submersions  with connected fibers}.
					It is tempting to think that {Hausdorff} Morita equivalence of  singular foliations
					can be phrased weakening these three conditions, i.e. that adopting weaker conditions one obtains the same equivalence classes of singular foliations. This is not the case:
					
					\begin{prop}\label{prop:3conditions}
						We do not obtain the same equivalence classes of singular foliations if we replace any of the three conditions in Def. \ref{def:defMEfol} as follows:
						\begin{itemize}
							\item ``Surjective'' by ``meets every leaf of the singular foliation'',
							\item ``Submersion'' by ``is transverse to the singular foliation''  {\cite[Def. 19]{AndrSk}},
							\item ``{With} connected fibres`` by ``such that the preimages of leaves are connected''.
						\end{itemize}
					\end{prop}
					
					\begin{rem}
						{The first two items above are motivated by what occurs for Lie groupoids.}
						The Morita equivalence of two Lie groupoids $\CG\rightrightarrows M$ and $\CH\rightrightarrows N$  can be equivalently phrased by replacing the condition 
						that the maps $\pi_M\colon P\fto M$ and $\pi_N\colon P\fto N$ in Def. \ref{def:MEgroid} are   surjective submersions with the following condition: these maps are transversal to the orbits of the Lie groupoids  ($G$ and  $H$ respectively) and
						meet every orbit. This fact can be found in \cite{MK2} and \cite{JoaoCrainic2017}, {and  follows also from Proposition \ref{prop:AllME}.}
						
					\end{rem}

					{To prove Proposition \ref{prop:3conditions}} it suffices to display examples of maps $\pi\colon P\fto (M,\CF)$ in which {each} of the conditions on the left hand side is weakened and so that the holonomy groupoid $\CH(\pi^{-1}\CF)$ is \emph{not} Morita equivalent to $\CH(\CF)$. Indeed, in this case, $(P,\pi^{-1}\CF)$ and  $(M,\CF)$ can not be {Hausdorff} Morita equivalent, due to  {Theorem \ref{thm:MEfolgroids}}.
					
					We now display the examples mentioned above, involving only regular foliations. 
					
					\begin{ex}\label{ex:surj}({\bf ``Surjective'' is needed}) Take $M$ to be the Moebius band  $$M:= \RR\times (-1,1)/\sim$$ where $(x,y)\sim (x+3k,(-1)^k y)$ for $k\in \ZZ$. Take $$P:= M\backslash \{\overline{(2,0)}\},$$ the Moebius band without a point in the ``middle circle" (the equivalence class of $(2,0)$).

						\noindent Let  $$\pi:P\hookrightarrow M$$ be the inclusion. On $M$ take the regular (rank one) foliation $\CF$ given by horizontal vector fields, then $\pi^{-1}(\CF)$ is also given by horizontal vector fields. Note that  $\pi$ is a submersion with connected fibres that meets every orbit, but it is not surjective.
						
						The isotropy group at $\overline{(0,0)}\in M$ of the holonomy groupoid $\CH(\CF)$ is  
						isomorphic to $\ZZ_2$. But the isotropy group at $\overline{(0,0)}\in P$ of the holonomy groupoid $\CH(\pi^{-1}\CF)$ is trivial (the leaf through that point is contractible). Therefore the two holonomy groupoids can not be Morita equivalent.
						
						\begin{figure}[h]
							\caption{The {manifold} $P$}
							\centering
							\scalebox{0.35}{\includegraphics{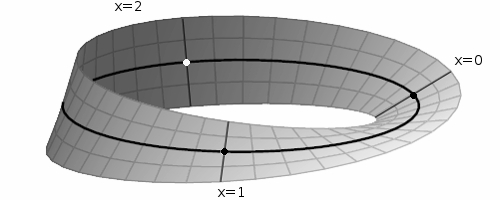}}
						\end{figure}
					\end{ex}

					\begin{ex}({\bf ``Submersion'' is needed}) Take  $P:=\RR \sqcup \left(\RR\backslash\{0\}\right)$, $M:=\RR$, and define $\pi\colon P\fto M$ {so that it sends the copy of $\RR$ to the point $0\in M$ and $\RR\backslash\{0\}$ to $M$ by the inclusion}. On $M$ take the full foliation. The map $\pi$ is surjective, has connected fibres and it is transverse to the foliation in $M$, but it is not a submersion. 
						
						The pullback foliation on $P$ is also the full foliation, but $P$ has three connected components. Hence the spaces of leaves are not homeomorphic and the holonomy groupoids  are not Morita equivalent. 
					\end{ex}
					
					\begin{ex}({\bf ``{With} connected fibres'' is needed}) 
						This example is a variation of Ex. \ref{ex:surj}.
						Take the Moebius band $M$  as in that example. {Let $$M':= \RR\times (-1,1)/\sim'$$ where $(x,y)\sim' (x+k,(-1)^k y)$ for $k\in \ZZ$. Notice that $M'$ is a smaller Moebius band, and since the equivalence classes of $\sim$ are contained in those of $\sim'$, there is a 
							natural quotient map  $q\colon M\to M'$ which is a $3$ to $1$ covering map.} 
						
						Let $P$ be $M$ with a point removed,  as in  Ex. \ref{ex:surj}. {Let} $$\pi '\colon P\fto M'$$ be  {the restriction of $q$ to $P$}. On $M'$ take the regular (rank one) foliation given by horizontal vector fields. Then $\pi '$ is a surjective submersion with connected preimages of leaves, but whose fibres are not connected (all fibres consist of three points, except for one that consists of two points). As in the first example,  the isotropy groups of the corresponding holonomy groupoids are $\ZZ_2$ at the point $\overline{(0,0)}\in M '$ and  the trivial group at $\overline{(0,0)}\in P$. Hence the  holonomy groupoids can not be Morita equivalent.
					\end{ex}
					
					\subsection*{{An extended   equivalence for singular foliations}}
					
					{Our notion of  Hausdorff 
						Morita equivalence (Def. \ref{def:defMEfol}) has certain drawbacks, which originate from the fact that {the}  space of arrows of a Lie groupoid is not necessarily Hausdorff:}
					
					\begin{itemize}
						\item If two non-Hausdorff source connected Lie groupoids are Morita equivalent, then their singular foliations might  {not} be Hausdorff Morita equivalent. (Compare with Proposition \ref
						{prop: implications}).
						\item As a consequence, we have to add a Hausdorffness assumption\footnote{Notice  that a  Hausdorffness assumption is needed also to match the Morita equivalence of Lie groupoids and Lie algebroids, see Proposition \ref{prop:MEalggroids}.} in Proposition \ref{prop:equivproj} on projective foliations.
					\end{itemize}  
					
					In an attempt to extend the notion of Hausdorff Morita equivalence so that the above drawbacks do not occur, we propose to allow the manifold $P$ in Def. \ref{def:defMEfol} to be \emph{non-Hausdorff}.
					
					A first issue to address is the notion of singular foliation on a  non-Hausdorff manifold. In section \S \ref{sec:fol.sheaf} we saw that on a (Hausdorff) manifold,
					Def. \ref{defi:sing.fol}  (in terms of compactly supported vector fields) is    equivalent to the characterisation given in that section (in terms of subsheaves). 
					On a non-Hausdorff manifold   $V$, this is no longer the case. Indeed the notion obtained extending trivially Def. \ref{defi:sing.fol} is quite restrictive, the main reason being that there might be points $p\in V$ where all compactly supported  vector fields vanish. However the sheaf of smooth vector fields on $V$ (a sheaf of $C^{\infty}$-modules) is well-behaved. Hence we propose to define a  singular foliation on a possibly non-Hausdorff manifold $V$ as \emph{an involutive, locally finitely generated subsheaf of the sheaf of smooth vector fields.} 
					
					A second issue to address is how to
					extend the notion of pullback foliation to a non-Hausdorff manifold. In section \S \ref{sec:fol.sheaf}, for a Hausdorff manifold the sheaf associated to a pullback foliation is given by 
					$\widehat{\iota_U^{-1}(\pi^{-1}\CF)}=\widehat{\pi|_U^{-1}\CF}$ for every open subset $U$.
					For a non-Hausdorff manifold $V$ and a submersion $\pi:V\fto M$ to a manifold, we define the pullback foliation as the following 
					{subsheaf} $\CS^{\pi^{-1}\CF}$ of the sheaf of vector fields: for any open (possibly non-Hausdorff) subset $U\subset V$,
					\[\CS^{\pi^{-1}\CF}(U):=\{X\in\CX(U) \st X|_H\in \widehat{\pi|_H^{-1} \CF} \text{ for all open Hausdorff subsets }H\subset U\}\]
					
					With the above ingredients at hand we can propose the following definition.
					
					\begin{defi}\label{def:NHMEfol}
						Two singular foliations  $(M,\cF_M)$ and $(N,\cF_N)$
						are {\bf Morita equivalent} if there exists a \emph{possibly non-Hausdorff} manifold $P$ and two \emph{surjective submersions with  connected Hausdorff fibres} $\pi_M\colon P\fto M$ and $\pi_N\colon P\fto N$  such that $\CS^{\pi_M^{-1}\CF_M}=\CS^{\pi_N^{-1}\CF_N}$ as subsheaves of $\CX_P$.
						\[\begin{tikzcd}
							& P \arrow[dl,swap, "\pi_M"] \arrow[dr, "\pi_N"]   &\\
							(M,\CF_M)& &(N,\CF_N)   
						\end{tikzcd}\]
					\end{defi}

					We then expect 
					\begin{itemize}
						\item  the following extension of Proposition \ref{prop: implications}: \emph{if two ({possibly} non-Hausdorff) Lie groupoids are Morita equivalent, then their singular foliations are Morita equivalent}.
						\item  to carry out the construction of the holonomy groupoid (Def. \ref{def:holgroidph}) starting from the sheaf-theoretic characterization of 
						singular foliation, even for a non-Hausdorff foliated manifold. Further we expect the following improvement of Theorem \ref{thm:MEfolgroids} to hold:  
						
						\emph{Morita equivalent singular foliations have holonomy groupoids which are Morita equivalent as open topological groupoids.} 
						\item  the following improvement of Proposition \ref{prop:equivproj}:
						\emph{
							Two projective singular foliations are Morita equivalent if{f} their holonomy groupoids  are Morita equivalent as  Lie groupoids.}
					\end{itemize} 
					
					\cleardoublepage
					
					
					\chapter{Quotients of foliated manifolds}\label{ch:5}
					
					In this section we assume the following setting: a foliated manifold $(P,\CF)$ with a surjective submersion $\pi:P\fto M$ with connected fibers. Here $M$ is seen as a quotient of $P$. We first show that, under some compatibility conditions between $\CF$ and $\pi$, there exists a singular foliation $\CF_M$ on $M$ given by a push-forward of $\CF$. Then we study the foliated manifold $(M,\CF_M)$ as a quotient of $(P,\CF)$.
					
					In the first section we will show that there exists a surjective morphism of topological groupoids $\Xi\colon \CH(\CF)\fto \CH(\CF_M)$, which implies that $\CH(\CF_M)$ is a quotient of $\CH(\CF)$. In the second section we give some preliminaries on quotients of Lie groupoids.
					
					In the third section we will recall Lie $2$-groups and their relation with quotients of groupoids. Finally in the fourth section we give a characterization of the $\Xi$-fibers, and describe a Lie $2$-group action on $\CH(\CF)$ with orbits lying inside these fibers.
					
					\section{{Quotients of holonomy groupoids}}\label{sec:Qfolman}
					
					In this section we show that given  a {surjective} submersion with connected fibers $\pi\colon P \to M$ and an ``invariant'' singular foliation $\cF$ on $P$ there is an induced singular foliation $\cF_M$ on $M$, which can be regarded as a quotient of the former.
					The main statement of this section is that the holonomy groupoid of $\cF_M$ is a quotient of the 
					holonomy groupoid of $\cF$, see Thm. \ref{thm:sur.hol}.
					We also give an explicit characterization of the quotient map when $\cF$ is a  pullback-foliation.}
				
				\subsection*{The main Theorem}\label{subsec:main}
				
				We recall \cite[Lemma 3.2]{AZ1}, about quotients of foliated manifolds. 
				
				\begin{prop}\label{prop:submfol} Let $\pi : P \to M$ be a {surjective} submersion with connected fibers. Let $\cF$ be a singular foliation on $P$, such that $\Gamma_c(\ker d\pi) \subset \cF$. Then there is a unique singular foliation  $\cF_M$ on $M$
					with $\pi^{-1}(\cF_M) = \cF$. 
				\end{prop} 
				
				In this section we consider the following set-up:
				\begin{center}
					\fbox{
						\parbox[c]{12.6cm}{\begin{center}
								A foliated manifold $(P,\cF)$ and a surjective submersion with connected fibers $\pi\colon P\fto M$, such that 
								\begin{equation}\label{eq:bracketpres}
									[\Gamma_c(\ker d\pi),\CF] \subset \Gamma_c(\ker d\pi)+\CF.
								\end{equation}
								Denote by $\CF_M$ the singular foliation on $M$ satisfying
								\begin{equation}\label{eq:pullpres}
									\CF^\text{big}=\pi^{-1}\CF_M.
								\end{equation}
								where $\CF^\text{big}:=\Gamma_c(\ker d\pi)+\CF$, as in proposition \ref{prop:submfol}
							\end{center}
					}}
				\end{center}

				The foliation $\cF_M$ is obtained from $\cF$ by a quotient procedure, {more precisely $\cF_M=\pi_*\cF$, see Lemma \ref{lem:projgen2}
					ii)} at the end of this section. By this fact it is natural to ask whether the same is true for the respective holonomy groupoids. The answer is true and we will discuss it in Theorem \ref{thm:sur.hol}. To continue we use the following lemma which is proven at the end of this section.
				
				\begin{lemma}\label{lem:all.sconbisub}	
					Any source connected atlas $\CU$ for $\CF$ satisfies that the family $$\pi\CU:=\{(U,\pi\circ \bt,\pi\circ\bs) \st U\in\CU\},$$
					generates a source connected atlas for $\CF_M$.
				\end{lemma} 
				
				Recall that a source connected atlas is an atlas generated by a family of bisubmersion which carry the identity diffeomorphism and are source connected, as definition \ref{def:scon.bisub}. The groupoid of a source connected atlas for $\CF$ is the holonomy groupoid of $\CF$ (See Prop. \ref{prop:sconn.grpd}).
				
				\begin{thm}\label{thm:sur.hol} Let $\CU$ be a source connected atlas for $\CF$, by Lemma \ref{lem:all.sconbisub}, the family $\pi\CU:=\{\pi U:=(U,\pi\circ \bt,\pi\circ\bs) \st U\in\CU\}$ generates a source connected atlas for $\CF_M$. The map $\Xi$ given as follows is a surjective open morphism of topological groupoids covering $\pi$:
					\begin{equation}\label{eq:Xidescr}
						\Xi \colon \CH(\CF)\fto \CH(\CF_M);	[u]\mapsto [u]_M,
					\end{equation}
					where $[u]_M$ is the class of $u\in \pi U\in \pi \CU$ as a bisubmersion for $\CF_M$. Note that the construction of $\Xi$ is canonical.
					
					Moreover, if $\CH(\CF)$ and $\CH(\CF_M)$ are smooth then $\Xi$ is a surjective submersion.
				\end{thm}
				
				\begin{proof}	
					The map $\Xi$ is well defined. If $u_1\in U_1\in \CU$ and $u_2\in U_2\in \CU$ are equivalent, then exists a morphism of bisubmersions sending $u_1$ to $u_2$, using this same morphism it is clear that $u_1\in \pi U_1\in \pi\CU$ is equivalent to $u_2\in \pi U_2\in \pi \CU$ as bisubmersions for $\CF_M$.
					
					The image of $\Xi$ lies indeed inside $\CH(\CF_M)$. By construction $\Xi(\CH(\CF))=Q_M(\sqcup_{U\in \CU} \pi U)$. Then by Lemma \ref{lem:all.sconbisub} the family $\pi \CU$ generates a source connected atlas $\CU_M$ for $\CF_M$, then $Q_M(\sqcup_{U\in \CU} \pi U)\subset Q_M(\sqcup_{U\in \CU_M} \pi U)= \CH(\CF_M)$ is an open subset. It is also clear that $\Xi$ covers $\pi$ and sends the identity bisection of $\CH(\CF)$ to the identity bisection of $\CH(\CF_M)$, then $\Xi(\CH(\CF))$ is an neighborhood of the identities of $\CH(\CF_M)$.
					
					To prove that it is a morphism of set theoretic groupoids we only need to prove that it preserves the composition. This is equivalent to the following statement: for any $U_1,U_2\in \CU$ we have that the map $\pi(U_1\circ U_2)\fto \pi U_1\circ \pi U_2 ; (u_1,u_2)\mapsto(u_1,u_2)$ is a morphism of bisubmersions for $\CF_M$, which is clearly true.
					
					To see that $\Xi$ is an open morphism of topological groupoids we only need to check that it is a continuous open map. This is clear because for any $U\in \CU$ the identity map $U\fto \pi U;u\mapsto u$ is a continuous open map. Recall that the topology on $\CH(\CF)$ and $\CH(\CF_M)$ is the quotient topology, which makes the maps $Q\colon U\fto \CH(\CF)$ and $Q_M\colon \pi U\fto \CH(\CF_M)$ continuous. One sees also that these maps are open under this topology (see Lemma \ref{lem:openmap}).
					
					Note that, if $\CH(\CF)$ and $\CF(\CF_M)$ are Lie groupoids, using the same argument as before, one check that $\Xi$ is a submersion (the maps $Q\colon U\fto \CH(\CF)$ and $Q_M\colon \pi U\fto \CH(\CF_M)$ are submersions).
					
					The map $\Xi$ is surjective. Recall that $\Xi(\CH(\CF))$ is a neighborhood of the identities of $\CH(\CF_M)$. Because $\Xi$ is a morphism of topological groupoids $\Xi(\CH(\CF))$ is a symmetric set closed under compositions. It is well known that any $\bs$-connected topological groupoid is generated by any symmetric neighborhood of the identities \cite{MK2}. Because $\CH(\CF_M)$ is $\bs$-connected then $\Xi(\CH(\CF))=\CH(\CF_M)$ and $\Xi$ is surjective.
				\end{proof}
				
				\begin{rem} Note that by the proof of Theorem \ref{thm:sur.hol}, the family $\pi\CU:=\{\pi U:=(U,\pi\circ \bt,\pi\circ\bs) \st U\in\CU\}$ is indeed a source connected atlas for $\CF_M$.
				\end{rem}
				
				{For regular foliations, the morphism $\Xi$ admits a familiar description.}
				
				{\begin{prop}\label{prop:regXi} When both $\cF$ and $\cF_M$ are regular foliations, the morphism $\Xi\colon \CH(\CF)\fto \CH(\CF_M)$ can be easily described by
						$$\Xi([\gamma]_{hol})=[\pi\circ \gamma]_{hol},$$
						for each curve $\gamma\colon[0,1]\fto P$ inside a leaf of $\CF$. {Here $[-]_{hol}$ denotes  holonomy classes.}
				\end{prop}}
				
				\begin{proof} {The {Lie} groupoids $\CH(\CF)$ and $\CH(\CF_M)$ are source connected atlases for $\CF$ and $\CF_M$ respectively. 
						The map $\widehat{\pi}\colon \CH(\CF)\fto \CH(\CF_M);[\gamma]_{hol}\mapsto[\pi\circ \gamma]_{hol}$ is a submersion, {since its a Lie groupoid morphism integrating the fiber-wise surjective Lie algebroid morphism $\pi_*$}. This implies that $(\CH(\CF),\bt_M \circ \widehat{\pi},\bs_M \circ \widehat{\pi})$  is a bisubmersion for $\CF_M$, by \cite[Lemma 2.3]{AndrSk}. Notice that the latter triple equals $(\Pi(\CF),\pi\circ \bt,\pi\circ \bs)$, which hence is a bisubmersion.}
					
					{ We can thus compute 
						$$\Xi([\gamma]_{hol})=[\pi\circ\gamma]_{hol},$$
						where the first equality holds by the above description of $\Xi$ and the fact that
						$\widehat{\pi}\colon (\CH(\CF),\pi\circ \bt,\pi\circ \bs)\fto (\CH(\CF_M),\bt_M,\bs_M)$ is a morphism of bisubmersions.}
				\end{proof}
				
				We present an example where $\cF$ is a regular foliation  and $\CF_M$ is a genuinely singular foliation. Notice that the holonomy groupoid of the former foliation has discrete isotropy groups, whereas for the latter the isotropy groups are not discrete in general.
				
				\begin{ex}\label{ex:non.sing.q}
					{
						Consider the cylinder $P:=S^1\times \RR$ with coordinates $(\theta,y)$ and
						the   regular foliation\footnote{The foliation $\cF$ is the quotient by the natural $\ZZ$-action of the foliation on the $x$-$y$-plane whose leaves are given by $graph(e^{x+c})$  on the open upper plane, $graph(-e^{x+c})$ on the open lower plane (with $c$ varying through all real numbers), and the line $\{y=0\}$.} $\cF$  given by the  integral curves of the nowhere vanishing vector field $X:=\de_{\theta}+y\de_{y}$.
						The circle $U(1)$ acts on the cylinder $P$ by rotations of the first factor, preserving the foliation $\cF$. The singular foliation $\CF^\text{big}$ on $P$ has three leaves (two open leaves, separated by the middle circle).
						The quotient map $$\pi\colon P=S^1 \times \RR \to  M:=P/U(1)\cong \RR$$ is the second projection.  On the quotient, the induced foliation is $\cF_M=\langle y\de_{y}\rangle$, a genuinely singular foliation.}
					
					{For the holonomy groupoids, we have $\CH(\cF)=\RR\times P$, the transformation groupoid of the action of the Lie group $\RR$ on $P$ by the flow of $X$, which reads $\phi_t(\theta \text { mod }2\pi, y)=(\theta +t \text { mod }2\pi,e^ty)$. Further  $\CH(\cF_M)=\RR\times M$, the transformation groupoid of the action of the Lie group $\RR$ on $M$ by the flow of $y\de_{y}$, which reads $\phi_t(y)=e^ty$. This follows from \cite[Ex. 3.7 (ii)]{AZ1}. The canonical surjective morphism of Thm. \ref{thm:sur.hol} is
						$$\Xi \colon \RR\times P\to \RR\times M,\;\; (t,p)\mapsto(t,\pi(p)).$$
						This can be seen 
						from   {eq. \eqref{eq:Xidescr}, since the vector field $X$ $\pi$-projects to $y\de_{y}$.} 
						Notice that, at points $S^1\times \{0\}$, the isotropy groups of $\CH(\cF)$ are discrete, as for all regular foliations, while the  isotropy group of $\CH(\cF_M)$ at the point $0\in M$ is isomorphic to $\RR$. 
					}
				\end{ex}
				
				\begin{figure}[h]\label{fig:moebius}
					\centering
					\scalebox{.3}{\includegraphics{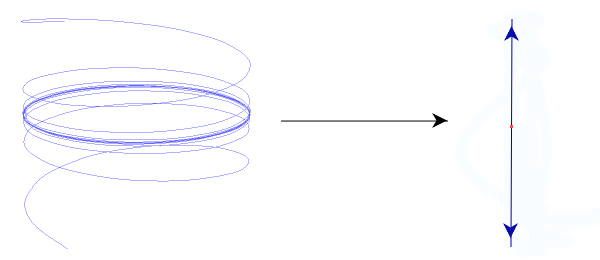}}
					\caption{The foliated manifold in Example \ref{ex:non.sing.q}}
				\end{figure}
				
				\subsection*{A characterization of the quotient map for pullback-foliations}
				
				We make the map $\Xi$ in Theorem \ref{thm:sur.hol} more explicit in the special case that {$\CF=\CF^{big}$ (i.e. when $\Gamma_c(\ker d\pi) \subset \CF$). In this case $\cF$ is the pullback of $\cF_M$ by $\pi$.}
				
				We will need theorem \ref{thm:pullbackgroid} from Section \ref{sec:pullb.grpd} which gives an isomorphism $\varphi\colon \CH(\CF) \xrightarrow{\sim} \pi^{-1}(\CH(\CF_M))$. Our alternative description of the map $\Xi$  is as follows:
				
				\begin{prop}\label{prop:normalq}
					Let $\pi \colon P \to M$ be a {surjective} submersion with connected fibers. Let $\cF$ be a singular foliation on $P$, such that $\Gamma_c(\ker d\pi) \subset \CF$. Denote $\CF_M$ the unique singular foliation on $M$ such that $\pi^{-1}(\CF_M)=\CF$.
					
					Under the canonical isomorphism $\varphi\colon \CH(\CF) \xrightarrow{\sim} \pi^{-1}(\CH(\CF_M)) $  given in theorem \ref{thm:pullbackgroid}, the following two morphisms coincide:
					\begin{itemize}
						\item the map $\Xi\colon \CH(\CF)\fto \CH(\CF_M)$    given by Theorem \ref{thm:sur.hol},
						\item the second projection $pr_2\colon \pi^{-1}(\CH(\CF_M))=P\times_M \CH(\CF_M) \times_MP\fto \CH(\CF_M)$.
					\end{itemize}
				\end{prop}
				
				\begin{rem}\label{rem:varphi} The isomorphism $\varphi$ from Theorem \ref{thm:pullbackgroid} is described as follows.
					Let $\CU$ be a path holonomy atlas for $\CF_M$. In theorem \ref{thm:pullbackgroid} it is proven that $\pi^{-1}\CU:=\{P{}_\pi \!\times_\bt U {}_\bs \!\times_\pi P: U\in \CU\}$ is a {source connected} atlas for $\pi^{-1}\CF_M$. Then 
					$$\varphi[(p,u,q)]= (p,[u],q).$$
				\end{rem}
				
				\begin{proof}
					Fix a path holonomy atlas $\CU$ for $\CF_M$. By Remark \ref{rem:varphi} the family $\pi^{-1}\CU:=\{P{}_\pi \!\times_\bt U {}_\bs \!\times_\pi P\}$ is a {source connected} atlas for $\CF=\pi^{-1}\CF_M$. By the characterization of $\Xi$ given in the proof of Thm. \ref{thm:sur.hol}  we get $\Xi([(p_0,u_0,q_0)])=[(p_0,u_0,q_0)]_M$.  It is sufficient to show that $[(p_0,u_0,q_0)]_M=[u_0]$.
					
					Note that $\hat{\pi}\colon (\pi^{-1}U,\pi\circ \bt,\pi\circ \bs)\fto U; (p,u,q)\mapsto u$ is a {morphism of} bisubmersions for $\CF_M$, therefore $(p_0,u_0,q_0)\in(\pi^{-1}U,\pi\circ \bt,\pi\circ \bs)$ is equivalent to $u_0\in (U,\bt,\bs)$. Then $\Xi([(p_0,u_0,q_0)])=[u_0]$.
				\end{proof}

				\subsection*{Proof of Lemma \ref{lem:all.sconbisub}}\label{app:thmXi}
				
				To prove Lemma \ref{lem:all.sconbisub} we give first the following statement.
				
				\begin{lem}\label{lem:projgen2}
					Let $(P,\CF)$ be a foliated manifold, $\pi : P \to M$ be a {surjective} submersion with connected fibers satisfying eq. (\ref{eq:bracketpres}).
					
					i) The set
					\begin{align*}
						& {\widehat{\cF}}^{proj}:=\{X\in {\widehat{\cF}}: X\text{ is $\pi$-projectable to a vector field on $M$}\}  
					\end{align*}
					generates $\cF$ as a ${C^{\infty}_c(M)}$-module. 
					
					ii) The singular foliation $\CF_M$  on $M$ satisfying eq. (\ref{eq:pullpres}) admits the following description:
					$$\cF_M=\pi_*(\cF):=Span_{{C_c^{\infty}(M)}}\{\pi_*X : X\in {\widehat{\cF}}^{proj}\}$$
				\end{lem}
				
				\begin{proof}
					We first make a claim.
					
					\noindent{\bf Claim:} \emph{{Lemma \ref{lem:projgen2} holds in the special case that $\Gamma_c(\ker d\pi)\subset \CF$.}}
					
					\noindent{Indeed, in this special case, by Prop. \ref{prop:submfol} there is  a unique singular foliation  $\cF_M$ on $M$ with $\pi^{-1}(\cF_M) = \cF$.
						Given this, i) is a consequence of Definition \ref{def:pullback}. For ii), note that $\pi^{-1}(\pi_*(\cF))=\CF$, as can be checked using i). Since} $\CF=\pi^{-1}\CF_M$, we obtain
					{$\cF_M=\pi_*(\cF)$ by the uniqueness statement in Proposition \ref{prop:submfol}. This proves the claim.} 
					
					\smallskip
					Take $\CF^{\text{big}}:=\Gamma_c(\ker d\pi)+\CF$, a singular foliation satisfying condition of the above claim.

					i)  {By the claim}, 	 ${\widehat{\cF^{\text{big}}}}^{proj}$
					generates $\cF^{\text{big}}$ as a ${C^{\infty}_c(M)}$-module.
					Take $X\in \CF\subset \CF^{\text{big}}$.
					There exist finitely many $Y_j\in {\widehat{\cF^{\text{big}}}}^{proj}$ and $f_j\in \CI_c(P)$ such that
					$X= \sum_j f_j Y_j$.
					By definition of $\CF^{\text{big}}$, we can write $Y_j= \widehat{Y}_j+Z_j$ with $\widehat{Y}_j\in {\widehat{\cF}}^{proj}$  and $Z_j\in \Gamma(Ker(d\pi))$. Then:
					\[X=\sum_j f^j_i Y_j= \sum_j f^j_i \widehat{Y}_j + \sum_j f^j_i Z_j. \]
					The last term $\sum_j f^j_i Z_j=X-\sum_j f^j_i \widehat{Y}_j$   lies in {$\cF$ as the difference of two elements of $\cF$, and is $\pi$-projectable (to the zero vector field on $M$). Hence this last term lies} in ${\widehat{\cF}}^{proj}$, and we have proven i).
					
					ii) {We have $\pi^{-1} (\pi_* (\CF))= \pi^{-1}( \pi_* (\CF^\text{big}))=\CF^\text{big}$ by the claim, and $\CF^\text{big}=\pi^{-1}\CF_M$ by definition.}  Using the uniqueness in Proposition \ref{prop:submfol} we get $\CF_M =\pi_* (\CF)$.
				\end{proof}
				
				\begin{lem}\label{prop:carry.diffM} Let $\pi:P\fto M$, $\CF$ and $\CF_M$ be as in Lemma \ref{lem:all.sconbisub}.   Then there exists a family of path holonomy bisubmersions $\CS$ for $\CF$ such that 
					{\begin{itemize}
							\item[i)] for every $x\in P$ there is $u_x\in U_x\in \CS$ carrying the identity {diffeomorphism} nearby $x$, 
							\item[ii)]  for any $U\in \CS$ we have that $(U,\pi \circ\bt, \pi\circ\bs)$ is a source connected bisubmersion for $\CF_M$. Further it is adapted to a path holonomy bisubmersions for $\CF_M$.
					\end{itemize}} 
				\end{lem}
				
				\begin{proof}
					{We first show that $\cF$ is locally finitely generated by $\pi$-projectable vector-fields in $\widehat{\cF}$.} For any $p\in P$ there are a neighborhood $U\subset P$ and finitely many generators ${Y_1,\dots,Y_k}\in \CX(U)$  of $\iota_U^{-1}(\cF)$,  for $\iota_U$ the inclusion. Take any precompact open set  $V\subset U$ containing $p$ and $\rho_V\in \CI_c(U)$ such that $\rho_V=1$ on $V$. For each $Y_i$, since $\rho_V Y_i\in \cF$, condition (i) of Lemma \ref{lem:projgen2}  assures that there is a finite number of projectable elements $X^j_i\in {\widehat{\cF}}^{proj}$
					{and $f^j_i\in \CI_c(P)$}
					such that:
					\[\rho_V Y_i=\sum_j {f^j_i} X^j_i.\]
					{Therefore} every element of $\iota^{-1}_V(\cF)$ is a $\CI_c(V)$-linear combination of the $X^j_i$, which are $\pi$-projectable {and lie in $\widehat{\cF}$}.

					{Now, for every point of $P$, take   be a minimal set of $\pi$-projectable elements
						$\{X_1,\dots,X_n\}$ in $\widehat{\cF}$ that are local generators of $\cF$ nearby that point.} 
					Let $(U,\bt,\bs)$ be the corresponding  {path-holonomy}  bisubmersion, where $U\subset \RR^n\times P$. Then 
					$(U,\pi\circ\bt,\pi\circ\bs)$ is a bisubmersion for $\cF_M$,
					with source map $(\lambda,p)\mapsto \pi(p)$ and target map $(\lambda,p)\mapsto 
					\exp_{\pi(p)}(\sum \lambda_i \pi_*X_i).$
					A way to see this is to apply \cite[Lemma 2.3]{AndrSk} to the {path-holonomy}  bisubmersion $W   \subset {\RR^n\times M}$ for $\cF_M$ corresponding to the generators $\{\pi_*X_1,\dots,\pi_*X_n\}$ and to the submersion\footnote{{More precisely, to its  restriction to  $(Id_{\RR^n},\pi)^{-1}(W)\cap U$.}} $(Id_{\RR^n},\pi)\colon \RR^n\times P\to \RR^n\times M$.
					We observe that $(Id_{\RR^n},\pi)$ is a morphism of bisubmersions from $(U,\pi\circ\bt,\pi\circ\bs)$  to $W$. {This shows that the former bisubmersion is adapted 
						{(see Def. \ref{def:atlas.bi})} 
						to the latter}.
				\end{proof}

				Let $\CS$ the family given in Lemma \ref{prop:carry.diffM}. {Note that the atlas generated by the family $\pi\CS:=\{ (U,\pi\circ\bt,\pi\circ\bs) \st U\in \CS\}$ is a source connected atlas for $\CF_M$. We finally prove Lemma \ref{lem:all.sconbisub}.
					
					\begin{proof}[Proof of Lemma \ref{lem:all.sconbisub}:] 
						Let $\CS$ the family of path holonomy bisubmersions for $\CF$ given by Lemma \ref{prop:carry.diffM}, and $\CU$ the source connected atlas for $\CF$ generated by $\CS$. We will start showing that $\CU$ satisfies the condition in Lemma \ref{lem:all.sconbisub}, then we generalize to any source conneted atlas. 
						
						Take $U\in \CU$, without loss of generality consider $U= U_1\circ \dots \circ U_k$ for $U_i\in \CS$. Denote $\pi U_i :=(U_i, \pi\circ \bt,\pi\circ\bs)$, which are bisubmersions for $\CF_M$. Note that the inclusion map $\omega\colon	U_1\circ \dots \circ U_k\fto \pi U_1\circ \dots \circ \pi U_k$ makes the following diagram commute:
						\[\begin{tikzcd}
							U \ar[d, shift right=.2em, swap, "\bt"]\ar[d,shift left=.2em,"\bs"] \ar[r,"\omega"] & \pi U_1\circ \dots \circ \pi U_k \ar[d, shift right=.2em, swap,"\bt"]\ar[d,shift left=.2em,"\bs"] \\
							P \ar[r,"\pi"] & M
						\end{tikzcd}\]	
						Because of the commutative diagram and since $\pi U_1\circ \dots \circ \pi U_k$ is a bisubmersion for $\CF_M$ one gets that
						
						\begin{equation}\label{eq:igual1}
							A:= (\pi\circ \bt)^{-1}\CF_M=(\pi\circ \bs)^{-1}\CF_M.
						\end{equation}
						
						Note also that $\pi^{-1}\CF_M=\CF+\ker_c(d\pi)$\footnote{We do not distinguish between the map of sections $d\pi\colon \CX(P)\fto \Gamma(\pi^* TM)$ and the fiber wise map $d\pi$. Moreover $\ker_c(d\pi)=\Gamma_c(\ker d\pi)$}, therefore by equation (\ref{eq:igual1}) we get the following:
						
						\[A=\bt^{-1}(\CF+\ker_c(d\pi))=\bs^{-1}(\CF+\ker_c(d\pi)),\]	
						now using that $\bs$ and $\bt$ are submersions we get the following:
						\[A=\bt^{-1}(\CF)+\bt^{-1}(\ker_c(d\pi))=\bs^{-1}(\CF)+\bs^{-1}(\ker_c(d\pi)),\]
						using that $U$ is a bisubmersion for $\CF$ one gets
						\[A=\ker_c(d\bs)+\ker_c(d\bt)+\bt^{-1}(\ker_c(d\pi))=\ker_c(d\bs)+\ker_c(d\bt)+\bs^{-1}(\ker_c(d\pi)).\]	
						Then we get:
						\[A= \bt^{-1}(\ker_c(d\pi))+\bs^{-1}(\ker_c(d\pi))= \ker_c(d(\pi\circ\bs))+\ker_c(d(\pi\circ\bt)).\]
						
						Hence $(U,\pi\circ\bt,\pi\circ\bs)$ is a bisubmersion for $\CF_M$. 
						
						The family $\pi\CU$ generates a source connected atlas because $\CU$ is a source connected atlas. Using the map $\omega$ one sees clearly that $\pi\CU$ is adapted to the source connected atlas for $\CF_M$ generated by $\pi\CS$.
						
						Now let $\CU'$ be any source connected atlas for $\CF$. Then $\CU'$ is adapted to $\CU$ (See Prop. \ref{prop:sconn.grpd} and Lemma \ref{lem:adapt}). This means that for any element $x\in U'\in \CU'$ there exists a neighborhood $U'_x\subset U'$ and a bisubmersion $U\in \CU$ with a morphism of bisubmersions $\omega_x$ making the following diagram commute:
						\[\begin{tikzcd}
							U'_x \ar[d, shift right=.2em, swap, "\bt"]\ar[d,shift left=.2em,"\bs"] \ar[r,"\omega_x"] & U \ar[d, shift right=.2em, swap,"\pi\circ\bt"]\ar[d,shift left=.2em,"\pi\circ\bs"] \\
							P \ar[r,"\pi"] & M
						\end{tikzcd}\]
						
						The triple $(U'_x,\bt\circ\pi,\bs\circ\pi)$ is a bisubmersion for $\CF_M$, since the argument for $U$ can be applied identically to the diagram above. Moreover, since being a bisubmersion is a local property (see Prop. \ref{prop:fol.loc.pro}) we have that $(U',\bt\circ\pi,\bs\circ\pi)$ is a bisubmersion for $\CF_M$.
						
						The family $\pi\CU'$ generates a source connected atlas because of the same reason $\pi\CU$ does.
					\end{proof}
					
					A consequence of Lemma \ref{lem:all.sconbisub} is as follows:
					
					\begin{cor} If $(U,\bt,\bs)$ is any path holonomy bisubmersion of $\CF$, then $\pi U:=(U,\pi\circ\bt,\pi\circ\bs)$ is a bisubmersion for $\CF_M$. 
						
						If $\CG\soutar P$ is a source connected groupoid whose foliation is $\CF$, then for any Hausdorff open set $U\subset \CG$ we get that $\pi U$ is a bisubmersion for $\CF_M$.
					\end{cor}
					
					\begin{rem}
						Not every bisubmersion $(U,\bt,\bs)$ for $\cF$ satisfies that $(U,\pi\circ\bt,\pi\circ\bs)$ is a bisubmersion for $\cF_M$. For instance take $P:=\RR^2$, $\cF=0$ and $M=\RR$ with map $\pi\colon P\fto M$ given by the first projection. Then $\cF_M=0$. Now take any diffeomorphism $\phi\colon P\fto P$  that does not preserve the foliation $\pi^{-1}\cF_M$ 
						by the fibers of $\pi$. Then $(P,Id, \phi)$ is a bisubmersion for $\cF$ but $(P,\pi,\pi\circ \phi)$ is not a bisubmersion for $\cF_M$.
						
						This is not a counterexample to Lemma \ref{lem:all.sconbisub}. One can show that $(P,Id, \phi)$ is not adapted to a source connected atlas. Indeed, if $U$ is a bisubmersion for $\CF=0$ adapted to a source connected atlas, the only diffeomorphism carried by $U$ is the identity diffeomorphism. The bisubmersion $(P,Id, \phi)$ carries $\phi$ and $\phi$ is any diffeomorphism of $P$.
					\end{rem}
					
					\section{Quotients of Lie groupoids}\label{sec:quo.grpd}
					
					In this subsection we will introduce to the reader the structure behind quotients of Lie groupoids. As a reference we use the book \cite{MK2} by Mackenzie.
					
					Remember that Lie groupoids can be seen as group-like and manifold-like structures. Let us recall quotients for groups and quotients for manifolds.
					
					In the category of groups, a quotient map is defined to be a surjective homomorphism. Let $G$, $H$ two groups and $\Xi\colon G\fto H$ a surjective homomorphism. Then the subgroup $K:=\ker(\Xi)\subset G$ is normal, it acts canonically on $G$ and its orbits are the fibers of $\Xi$.
					
					In the category of manifolds, a quotient map is defined as a surjective submersion. Let $P,M$ be manifolds and $\pi\colon P\fto M$ a surjective submersion. Then the equivalence relation $R:= P\times_M P$ on $P$ is an embedded, wide Lie subgroupoid of the pair groupoid $P\times P$.
					
					\begin{defi} An equivalence relation $R$ on a manifold $P$ is called \textbf{smooth} if it is an embedded, wide Lie subgroupoid of the pair groupoid $P\times P$.	
					\end{defi}
					
					It is well known, by the Godement criterion, that there is a bijection between quotients maps for manifolds and smooth equivalence relations.
					
					In the category of Lie groupoids, the quotient maps are defined as fibrations.
					
					\begin{defi} Let $\CG\soutar P$ and $\CH\soutar M$ be Lie groupoids. A groupoid morphism of Lie groupoids $\Xi\colon \CG\fto \CH$ covering the smooth map $\pi\colon P\fto M$ is called a \textbf{fibration} if and only if $\Xi$ and $\pi$ are surjective submersions and the map $\CG\fto \CH {}_\bs\!\times_\pi P ;\phi\mapsto (\Xi(\phi), \bs(\phi))$ is also a surjective submersion.
						
						For open topological groupoids in \cite{RM2016} fibrations are defined replacing in the above definition surjective submersions by surjective open maps.
					\end{defi}
					
					The condition for the map $\CG\fto \CH {}_\bs\!\times_\pi P ;\phi\mapsto (\Xi(\phi), \bs(\phi))$ to be a surjective submersion assures that the composition on $\CH$ is entirely given by the composition in $\CG$, more clearly that for two composable elements in $\CH$ there exists composable preimages in $\CG$.
					
					In \cite{MK2} there are other two notions to describe a quotient map of Lie groupoids: smooth congruences and normal subgroupoid systems. They correspond to smooth equivalences on manifolds and to normal subgroups respectively. 
					
					\begin{defi}\label{def:smt.cong} Let $\CG\soutar P$ be a Lie groupoid. A \textbf{smooth congruence} on $\CG$ consists of two smooth equivalences $\CR$ on $\CG$ and $R$ on $P$, such that:
						\begin{itemize}
							\item $\CR\soutar R$ is a Lie subgroupoid of the Cartesian product $\CG\times \CG \soutar P\times P$.
							
							\item The map $\CR\fto \CG {}_\bs \!\times_{\mathrm{Pr}_1} R;(g_2,g_1)\mapsto (g_2, \bs(g_2),\bs(g_1))$ is a surjective submersion.\footnote{In \cite{MK2} the author described this condition as certain square diagram being "versal".}
						\end{itemize}	
					\end{defi}
					
					On the other hand, normal subgroupoid systems can be thought of as the group counterpart to represent fibrations. Note that, if $\CK$ is a closed embeded wide Lie subgroupoid of $\CG$ then, by the Godement criterion, the set 
					$$ \CK\backslash \CG=\{ \CK\circ g \st g\in \CG\}$$
					has an unique manifold structure making the quotient map $q\colon \CG \fto \CK\backslash \CG; g\mapsto \CK g$ a surjective submersion. Note also that, the source $\bs\colon \CG\fto P$ quotients to a well-defined surjective submersion that we also denote as $\bs\colon(\CK\backslash \CG)\fto P$.
					
					\begin{defi}
						A normal subgroupoid system in $\CG\soutar P$ is a triple $(\CK,R,\theta)$ where $\CK$ is a closed, embeded, wide Lie subgroupoid of $\CG$; $R$ is a smooth equivalence on $P$; and $\theta$ is an action of $R$ on the map $\bs\colon(\CK\backslash \CG)\fto P$ such that For all $(p,q)\in R$ the following conditions hold:
						\begin{enumerate}
							\item Let $g\in \CG$ with $\bs(g)=q$ and $g_1\in \CG$ such that $\theta(p,q)(\CK g)=\CK g_1$ then $(\bt(g_1),\bt(g))\in R$.
							\item $\theta(p,q)(\CK e(q))=\CK e(p)$.
							\item Let $h$ and $g$ composable elements in $\CG$ such that $\bs(g)=q$. Consider $g_1$ and $h_1$ such that $\theta(p,q)(\CK g)=\CK g_1$ and $\theta(\bt(g_1),\bt(g))(\CK h)=\CK h_1$, then:
							$$\theta(p,q)(\CK hg)=\CK h_1g_1.$$
						\end{enumerate}
					\end{defi}

					The proof of the following theorem can be found in \cite{MK2}.
					
					\begin{thm}\label{thm:norm.sub.sys}\label{thm:smt.cong}
						\begin{enumerate}
							\item If $\Xi$ is a fibration covering $\pi$ then the pais $(\CR,R)$ is a congruence, where $\CR:=\CG\times_\Xi \CG$ and $R:= P\times_\pi P$. Conversely given a congruence the quotient map is a fibration.
							\item If $(\CK,R,\theta)$ is a normal sub-groupoid system then $R$ together with the following relation in $\CG$ given by:
							$$\CR=\{(h,g)\in \CG^2 \st (\bs(h),\bs(g))\in R \y \theta(\bs(h),\bs(g))\CK g=\CK h\},$$
							generate a smooth congruence on $\CG$.
							
							\item Let $(\CR,R)$ be a smooth congruence on $\CG$. Then the class of the identity $\CK$ is a closed embedded wide Lie subgroupoid of $\CG$.
							
							Moreover, for every $(p,q)\in R$ and $g\in \CG$ with source $q$ we define
							$$\theta(p,q)(\CK g)=\CK h$$
							where $h$ is any element related to $g$ and with source $p$. Then $\theta$ is a well defined action on $\CK\backslash \CG$.
							
							Finally $(\CK,R,\theta)$ is a normal subgroupoid system.
						\end{enumerate}
					\end{thm}
					
					The theorem above says that smooth congruences, fibrations and normal subgroupoids systems are equivalent descriptions for the quotients of Lie groupoids.
					
					On an equivalent way, replacing submersions with open surjective maps, we get the same descriptions for open topological groupoids.
					
					Now we want to relate this notion of fibrations with the map $\Xi$ of theorem $\ref{thm:sur.hol}$. We get that in a special case the map $\Xi$ is indeed a fibration.
					
					\begin{prop}
						As in Prop. \ref{prop:submfol}, let $\pi\colon P\fto M$ be a surjective submersion with connected fibers. Let $\CF$ be a singular foliation on $P$, such that $\Gamma_c(\ker d\pi)\subset \CF$. Denote $\CF_M$ the unique singular foliation in $M$ such that $\pi^{-1}\CF_M=\CF$.
						
						The map $\Xi\colon \CH(\CF)\fto \CH(\CF_M)$ of Thm. \ref{thm:sur.hol} is a fibration.
					\end{prop}
					
					\begin{proof} This is a direct consequence of Prop. \ref{prop:submfol}.  The map $\varphi\colon \CH(\CF)\fto \pi^{-1}(\CH(\CF_M))$ is an homeomorphism of topological groupoids and the projection $\text{pr}_2\colon \pi^{-1}(\CH(\CF_M))\fto \CH(\CF_M)$ is clearly a fibration.
					\end{proof}
					
					It is important no notice that the map $\Xi$ of theorem \ref{thm:sur.hol} is not allways a fibration.
					
					\begin{ex} Let $P=\RR^2-(\{0\}\times \RR^+)$, $M=\RR$ and $\pi\colon P\fto M$ the first projection. Take $\CF$ the foliation given by the horizontal lines, then $\CF_M$ is the full foliation on $M$.
						
						In this case, the map $\CH\fto \CH' {}_\bs\!\times_\pi P ;\phi\mapsto (\Xi(\phi), \bs(\phi))$ is not surjective. Take $\theta\in \CH(\CF_M)$ such that $\bs(\theta)=1$ and $\bt(\theta)=-1$. The element $(\theta,(-1,1))\in \CH {}_\bs\!\times_\pi P$, nevertheless there is no $\gamma$ in $\CH(\CF)$ satisfying $\bs(\gamma)=(-1,1)$ and $\Xi(\gamma)=\theta$.
					\end{ex}
					
					We will show a different case when $\Xi$ is a fibration. To do so we will assume that $\pi$ is given as a quotient by a group action.
					
					\subsection*{{Induced groupoid actions}}\label{subsec:groidaction}
					
					Motivated by the case of holonomy groupoids  we will address in \S\ref{sec:general},
					in this section we consider the following abstract setting:
					\begin{enumerate}
						\item a free and proper action of a Lie group $G$ on a manifold $P$, with quotient map $\pi\colon P\fto M:=P/G$,
						\item a surjective open morphism of topological groupoids (or a surjective submersion morphism of Lie groupoids when $\CH$ and $\CH'$ are smooth) $\Xi\colon \CH\fto \CH'$ covering $\pi$
						\item a group action $\vec{\star}$ of  $G$ on $\CH$ by groupoid automorphisms covering the $G$-action on $P$ and preserving each fiber of $\Xi$.
					\end{enumerate}

					\[\begin{tikzcd}
						\CH \ar[d, shift right=.2em, swap,]\ar[d,shift left=.2em,] \ar[r,"\Xi"] & \CH' \ar[d, shift right=.2em, swap]\ar[d,shift left=.2em] \\
						P \ar[r,"\pi"] & M
					\end{tikzcd}\]
					
					\begin{prop}
						Assuming 1. 2. 3. above, the map the map $\aleph\colon\CH\fto \CH' {}_\bs\!\times_\pi P ;\phi\mapsto (\Xi(\phi), \bs(\phi))$ is surjective.
					\end{prop}
					
					\begin{proof}  	
						Let $(\theta,p)\in \CH(\CF_M) {}_\bs\!\times_\pi P$, there exists a $\phi\in \CH(\CF)$ such that $\Xi(\phi)=\theta$. Then $\pi(\bs(\phi))=\pi(p)$ what means that there exists a unique $g\in G$ such that $g\bs(\phi)=p$. Then $\aleph(g\vec{\star}\phi)=(\theta,p)$ i.e. $\aleph$ is surjective.
					\end{proof}
					\begin{rem} If $\aleph$ is open, then $\Xi$ is a (topological) fibration. We believe it is open, but we still need to get a full proof.
					\end{rem}
					
					Denote $\CK:=\ker(\Xi)$, which is a topological subgroupoid of $\CH$ with space of objects $P$.  Note that since the  action of $G$ on $\CH$ preserves each fiber of $\Xi$, we obtain by restriction a group action of $G$ on $\CK$, also by groupoid automorphisms.
					
					For the groupoid $R:=P\times_\pi P$, there is an action $\theta$ on $\bs\colon \CK\backslash \CH\fto P$ given by the following equation:
					$$\theta(p,q)\CK\phi=\CK([p/q]\vec{\star}\phi),$$
					where $[p/q]\in G$ is the unique element such that $p=[p/q]q$.
					
					\begin{theorem}
						Assuming 1. 2. 3. above, if $\CH$ and $\CH'$ are Lie groupoids then $(\CK,R,\theta)$ is a normal subgroupoid system.
						
						We have also that the map $\Xi$ is a (smooth) fibration.
					\end{theorem}
					\begin{proof} By theorems \ref{thm:smt.cong} and \ref{thm:norm.sub.sys} it is enough to show that $(\CK,R,\theta)$ is a normal subgroupoid system.
						
						Because $\Xi$ and $\pi$ are surjective submersion we have that $\CK$ is an embeded wide Lie subgroupoid of $\CH$ and $R$ is a smooth equivalence.
						
						Because $\vec{\star}$ is a smooth action by groupoid automorphisms covering the $G$-action on $P$ and 
						preserving each $\Xi$-fiber, we have that $\theta$ is indeed a smooth and well defined action of $R$ on $\bs\colon \CK\backslash \CH\fto P$.
					\end{proof}
					
					\begin{lem}\label{prop:fib.xi0}
						The fibers of $\Xi$ are given by the orbits of the action of $G$ on $\CH$ composed\footnote{Recall that the composition (multiplication) of the groupoid $\CH$ is denoted by $\circ$.}
						with elements in $\CK$. More precisely, the fiber through $\xi\in \CH$ is 
						$$\CK\circ (G\vec{\star} \xi) :=\{\chi\circ(g\vec{\star} \xi)  :\chi\in \CK \y g\in G\}.$$	
					\end{lem}
					\begin{proof}
						{Fix $\xi_1\in \CH$. The above subset $\CK\circ (G\vec{\star} \xi_1)$ is certainly contained in the $\Xi$-fiber through $\xi_1$.}
						
						{ To show the converse,} let $\xi_2$ lie in the same $\Xi$-fiber as $\xi_1$, then $\bs(\xi_2)$ and $\bs(\xi_1)$ lie in the same $\pi$-fiber. 
						Let $g\in G$ such that $g\bs(\xi_1)=\bs(\xi_2)$. {As this equals $\bs(g\vec{\star} \xi_1)$, the groupoid composition} 
						{$\xi_2\circ (g\vec{\star} \xi_1)^{-1}$} is well-defined, and
						$$\Xi\left(\xi_2\circ (g\vec{\star} \xi_1)^{-1}\right)=\Xi(\xi_2)\circ\Xi(g\vec{\star} \xi_1)^{-1}=\Xi (\xi_2)\circ \Xi(\xi_1)^{-1}= 1_{\pi(\bt(\xi_2))},$$
						where we used that $\Xi$ is a groupoid morphism and the action of $G$ preserves the $\Xi$ fibers, respectively in the {first and second equality}. As a consequence, {$\xi_2\circ (g\vec{\star} \xi_1)^{-1}\in \ker \Xi=\CK$}.
					\end{proof}
					
					\begin{rem}
						While the fibers of a group morphisms are just translates of the kernel, for morphisms of groupoids over different bases this is no longer true. This explains why the description of the fibers in Lemma \ref{prop:fib.xi0} is slightly involved.
					\end{rem}
					
					We will now describe the the fibers of $\Xi\colon \CH\fto \CH'$, as follows:
					
					\begin{prop}\label{cor:grpd.act.hol} 
						There is a Lie groupoid structure on $\CK\times G$ and a groupoid action of $\CK\times G$ on $\bt\colon \CH\fto P$, whose orbits coincide with the fibers of $\Xi$.	
					\end{prop}
					
					To prove Prop.  \ref{cor:grpd.act.hol} we first need the following construction.
					
					Since the Lie group $G$ acts by groupoid automorphisms on the groupoid $\CK$, we can form the semidirect product groupoid (see \cite[\S 2]{BrownCoefficients}  and \cite[\S 11.4]{TopGrpds}). We obtain a groupoid structure on $\CK\times G$ with space of objects $P$, as follows:
					\begin{itemize}
						\item[a)] the source and target maps are respectively
						
						$\CK\times G \to P;(\xi,g)\mapsto g^{-1}\bs(\xi)$ and $\CK\times G \to P;(\xi,g)\mapsto \bt(\xi)$,
						\item[b)] the composition is
						$$\circ\colon (\CK\times G)\times_P (\CK\times G)\fto (\CK\times G); \;\;(\xi_2,g_2)\circ (\xi_1,g_1)\mapsto (\xi_2\circ (g_2\vec{\star} \xi_1),g_2 g_1),$$
						\item[c)] the {unit map is} $1\colon P\fto \CK\times G; \; p\mapsto(1_p,e_G)$, where $e_G$ denotes the unit element of the group $G$,
						\item[d)] the   inverse map is  $(-)^{-1}\colon \CK\times G\fto \CK\times G; \;(\xi,g)\mapsto (g^{-1}\vec{\star} \xi^{-1},g^{-1})$.
					\end{itemize}
					One checks 
					that the groupoid $\CK\times G$ acts on the map $\bt\colon \CH\fto P$ via
					\begin{equation}\label{eq:stargroid}
						\smallwhitestar\colon (\CK\times G) \times_P \CH {\to \CH};\;\; \left(( {\chi},g),{\xi}\right) \mapsto ( {\chi},g)\smallwhitestar {\xi}:={\chi}\circ(g\vec{\star} {\xi}).
					\end{equation}
					
					\begin{proof}[Proof of Prop. \ref{cor:grpd.act.hol}] The orbits of the groupoid action ${\smallwhitestar}$ of $\CK\times G$ on $\bt\colon \CH\fto P$ are precisely the fibers of $\Xi$, by Lemma \ref{prop:fib.xi0}. 
					\end{proof}

					\section{{Lie $2$-group actions on holonomy groupoids}}
					\label{sec:quotgrouppull}
					
					{We start reviewing Lie 2-groups and Lie 2-group actions.
						In \S\ref{subsec:ex.lie.2.grp} we present an important special case of Thm. \ref{thm:sur.hol} in which the map $\Xi$   is the quotient map of a Lie 2-group action on $\CH(\CF)$ (see Thm. \ref{thm:act.q.fol} and Prop. \ref{prop:afterthm:act.q.fol}). We will revisit this special case later on, in Prop. \ref{prop:sameaction}.}

					\subsection*{{Background on Lie 2-groups}}\label{subsec:Lie.2.grp}
					
					In the sequel will need the notion of Lie 2-group, which we recall here.
					
					\begin{defi} 
						{	A \textbf{Lie 2-group} is a group in the category of Lie groupoids.}
					\end{defi} 
					{In other words,  a Lie 2-group is a Lie groupoid $\CG\soutar G$ such that $\CG$ and $G$ are Lie groups, so that the group multiplication and group inverse are Lie groupoid morphisms,
						and the inclusion of the neutral elements is a Lie groupoid morphism.}

					\begin{rem} Equivalently, a Lie 2-group is a groupoid in the category of Lie groups, i.e.  a Lie groupoid $\CG\soutar G$ such that $\CG$ and $G$ are Lie groups; and the source, the target, the composition and the inverse maps are homomorphisms.
					\end{rem}
					
					\begin{ex}\label{ex:GH}
						Let $G$ be a Lie group and $H\subset G$ a normal Lie subgroup. Then $H$ acts on $G$ by left multiplication, leading to the action Lie groupoid $H\times G\soutar G$. {In particular, the groupoid composition is $$(h_2,h_1  g)\circ(h_1,g)=(h_2 h_1, g).$$}
						Note that its  space of arrows has a group structure, namely the semidirect product by the conjugation action $C_g(h)=g h g^{-1}$ {of $G$ on $H$}. Explicitly, {the group multiplication} is given by
						$$(h_1,g_1)\cdot(h_2,g_2)=(h_1C_{g_1} (h_2), g_1g_2).$$
						{We write $H\rtimes G$ for $H\times G$ endowed with this group structure.}
						
						{One can check that $H\rtimes G\soutar G$ is a Lie 2-group.}
					\end{ex}
					
					\begin{rem}\label{rem:xm}
						For the sake of completeness, we provide the description of a Lie 2-group in full generality.
						A crossed module of Lie groups consists of  Lie groups   $H$ and $G$, Lie group morphisms $C\colon G\fto \mathrm{Aut}(H); g\mapsto C_g$ and $\bt\colon H\fto G$ such that $\bt(C_g(h))=g\bt(h)g^{-1}$ and $C_{\bt(h)}(j)=hjh^{-1}$ for all $g\in G$ and $h,j\in H$.
						There is a bijection between Lie 2-groups and crossed module of Lie groups \cite{CroMod}.
						Given a  Lie 2-group $\CG\soutar G$, the associated crossed module is given by $G$, by $H:=\ker(\bs)$ (a normal subgroup of  $\CG$), by the restriction $\bt\colon H\fto G$  of the target map, and the Lie group morphisms $C\colon G\fto \mathrm{Aut}(H); g\mapsto C_g(h):=g h {g^{-1}}$. Then $\CG$ as a Lie group is isomorphic to the semidirect product of $G$ and $H$ by the action $C$, and as a Lie groupoid it is isomorphic to the transformation groupoid of the $H$-action on $G$ by left multiplication with $\bt(\cdot)$.
					\end{rem}

					\begin{defi}
						A \textbf{Lie 2-group action} is a  {group} action in the category of Lie groupoids. 
					\end{defi}
					
					Hence an action of a Lie 2-group $\CG\soutar G$ on a Lie groupoid $\CH\soutar P$ consists of    {group} actions of $\CG$ on $\CH$ and of $G$ on $P$ such that the action map
					\[\begin{tikzcd}
						\CG \times \CH \ar[d, shift right=.2em, swap,"\bt\times \bt"]\ar[d,shift left=.2em,"\bs\times \bs"] \ar[r] & \CH \ar[d, shift right=.2em, swap,"\bt"]\ar[d,shift left=.2em,"\bs"] \\
						G\times P \ar[r] & P
					\end{tikzcd}\]
					is a Lie groupoid map.
					{Notice that such an action is not by Lie groupoid automorphisms of $\CH$. Nevertheless the following result holds:}
					
					\begin{prop}\label{prop:lie2grp.q}
						{Consider a free and proper action of the  Lie 2-group} $ H\rtimes G \soutar G$ on a Lie groupoid $\CH\soutar P$. 
						
						Define 
						$$R:=\{(x,gx)\in P\times P \st x\in p \y g\in G\},$$  $$\CR:=\{(\theta,(h,g)\theta)\in\CH \st \theta\in \CH \y (h,g)\in H\rtimes G\}.$$
						Then $(\CR,R)$ is a smooth congruence for $\CH$.
						
						This implies that $\CH':= \CH/(H\rtimes G)$ and $M:= P/G$ are manifolds, and that $\CH'\soutar M$ acquires a canonical  Lie groupoid structure. Further the projection $\CH\to \CH'$ is  a fibration.
					\end{prop} 
					
					\begin{proof}\footnote{For a proof in a special case, see also \cite[Prop. 2.10]{CZ2}.} By Thm. \ref{thm:smt.cong}, we only need to prove that $(\CR,R)$ is a smooth congruence. Following the definition of smooth congruences \ref{def:smt.cong} we need to prove three statements:
						\begin{enumerate}
							\item The sets $\CR$ and $R$ are smooth equivalence relations:
							
							Because $H\rtimes G$ and $G$ act freely and properly we get that $\CR$ and $R$ are smooth equivalence relations. So $\CH':=\CH/\CR= \CH/(H\rtimes G)$ and $M:= P/R=P/G$ are manifolds.
							
							\item $\CR\soutar R$ is a Lie subgroupoid of the Cartesian product $\CH\times \CH \soutar P\times P$:
							
							Because $\CR\subset \CH\times \CH$ and $R\subset P\times P$ are embedded submanifolds one only needs to prove the following.
							\begin{itemize}
								\item {\bf The manifold $R$ is the base for $\CR$:} For $(p,gp)\in R$ there is $(e_p,(e,g)e_p)\in \CR$ as its identity. Also, for any $(\theta,(h,g)\theta)\in \CR$ we get: $$(\bs\times\bs)((\theta,(h,g)\theta))=(\bs(\theta),g\bs(\theta))\in R$$
								$$(\bt\times\bt)((\theta,(h,g)\theta))=(\bt(\theta),(\bt(h)g)\bt(\theta))\in R$$

								\item {\bf Closeness under composition:} Take composable elements
								$(\theta_1,(h_1,g_1) \theta_1),(\theta_2,(h_2,g_2) \theta_2)\in \CR$, then:
								\begin{align*}
									\bs(\theta_1)&=\bt(\theta_2)\\
									g_1\bs(\theta_1)=\bs((h_1,g_1) \theta_1)&=\bt((h_2,g_2) \theta_2)=(\bt(h_2) g_2)\bt( \theta_2)
								\end{align*}
								
								By the second equality above and because the $G$-action is free on $P$ then $g_1=\bt(h_2) g_2$, which implies that $(h_1,g_1),(h_2,g_2)\in H\rtimes G$ are also composable. By axioms of 2-Lie group actions we get:
								$$(\theta_1\circ \theta_2,(h_1,g_1) \theta_1\circ (h_2,g_2) \theta_2)=(\theta_1\circ \theta_2, ((h_1,g_1)\circ (h_2,g_2)) (\theta_1\circ \theta_2))\in \CR.$$
								
								\item {\bf Closeness under inverse:}
								$$(\theta,(h,g)\theta)^{-1}=(\theta^{-1},(h^{-1},hg)\theta^{-1})\in \CR.$$
								
							\end{itemize}
							
							\item The map $\tau\colon \CR\fto \CH {}_\bs \!\times_{\mathrm{Pr}_1} R;(\theta_2,\theta_1)\mapsto (\theta_2, \bs(\theta_2),\bs(\theta_1))$ is a surjective submersion:
							
							Note that $\CR$ is diffeomorphic to $\CH\times H\times G$. Also $R$ is diffeomorphic to $P\times G$, therefore $\CH\times_M R$ is diffeomorphic to $\CH\times G$.
							
							The map $\tau$ under the diffeomorphisms given above, is the following map 
							$$\tau\colon \CH\times H\times G\fto \CH\times G; (\theta,h,g)\mapsto(\theta,g),$$ 
							which is clearly a surjective submersion.	
						\end{enumerate}	
					\end{proof}
					
					Although we do not need the above result, we mention it here because it puts in perspective Theorem \ref{thm:act.q.fol} below.

					\subsection*{A Lie 2-group action on the holonomy groupoid of a pullback foliation}\label{subsec:ex.lie.2.grp}

					Fix  a foliated manifold $(P,\cF)$ and a free and proper action of a connected Lie group $G$ on $P$ preserving $\CF$. We assume that the infinitesimal generators of the $G$-action lie in the global hull $\widehat{\CF}$, i.e.  $\ker_c(d\pi)\subset \CF$. This occurs exactly when $\cF$ is the pullback of $\cF_M$ by $\pi$, as in Prop. \ref{prop:normalq}.
					
					\begin{thm}\label{thm:act.q.fol}
						Let $G$ be a connected Lie group acting freely and properly on a foliated manifold $(P,\CF)$. 
						Assume {that $\Gamma_c(\ker d\pi)\subset \CF$.}  
						
						Then there is a canonical Lie 2-group action of $G\rtimes G\soutar G$  on the holonomy groupoid $\CH(\CF)$.
					\end{thm}
					
					Here $G\rtimes G\soutar G$ is endowed with the Lie 2-group structure of Ex. \ref{ex:GH}.
					
					\begin{proof} We make use  the canonical isomorphism $\CH(\CF) \cong \pi^{-1}(\CH(\CF_{M}))$ given in  Theorem \ref{thm:pullbackgroid}.
						
						There is a canonical Lie $2$-group action of $G\rtimes G$ on   $\pi^{-1}(\CH(\CF_{M}))$, extending the given action of $G$ on the base $P$,
						given by
						\begin{equation}\label{eq:star}
							(h,g)* (p,[v],q)=(hgp,[v],gq).
						\end{equation}
						It can be checked by computations that this defines a group action and groupoid morphism. Alternatively, as mentioned in \cite[\S3]{CZ2}, we can use the isomorphism of Lie 2-groups to    $G \times G \soutar G$ (the pair groupoid, with product group structure) given by  $G\rtimes G\cong G\times G, (h,g)\mapsto (hg,g)$. Under this isomorphism,   
						\eqref{eq:star} becomes 
						$$ (G\times G)\times {\pi^{-1}(\CH(\CF_{M}))}\to {\pi^{-1}(\CH(\CF_{M}))},\;\; \left((h,g),(p,[v],q)\right) \mapsto (hp,[v],gq),$$
						which is easily checked to   be a Lie $2$-group action.
					\end{proof}
					
					\begin{prop}\label{prop:afterthm:act.q.fol}
						Assume the set-up of Theorem \ref{thm:act.q.fol}.
						
						The orbits of the Lie 2-group action of $G\rtimes G\soutar G$ on $\CH(\CF)$ are exactly the fibers of the canonical map $\Xi \colon \CH(\CF)\fto \CH(\CF_M)$. In particular, the quotient of $\CH(\CF)$ by the action is canonically isomorphic to $\CH(\CF_M)$.
						
					\end{prop}
					\begin{proof}
						{The formula   \eqref{eq:star} makes clear what the orbits are, and Prop. \ref{prop:normalq} shows that they agree with the $\Xi$-fibers.}
					\end{proof}

					\section{{Quotients of foliations by group actions: the general case}}
					\label{sec:general}
					In this section we consider the following set-up:
					\begin{center}
						\fbox{
							\parbox[c]{12.6cm}{\begin{center}
									a foliated manifold $(P,\cF)$,\\
									a free and proper action of a connected Lie group $G$ on $P$ preserving $\cF$.\end{center}
						}}
					\end{center}
					Note that this implies the set-up of \S \ref{sec:Qfolman}, then the condition of equation \eqref{eq:bracketpres} in Thm. \ref{thm:sur.hol} is satisfied. Hence we obtain a surjective groupoid morphism 
					\begin{equation}\label{eq:Xi}
						\Xi\colon \CH(\CF)\fto \CH(\CF_M)  
					\end{equation}
					covering the projection   $\pi \colon P\to M:=P/G$, where the latter is endowed with the foliation $\CF_M$ specified there.
					
					Unlike the special case considered in \S \ref{subsec:ex.lie.2.grp}, the $\Xi$-fibers are not the orbits of a Lie 2-group action in general. In this section we make two general statements about the $\Xi$-fibers.
					First, after lifting the $G$-action on $P$ to an action on $\CH(\CF)$ by groupoid automorphisms, we characterize the fibers of $\Xi$ as the orbits of a \emph{groupoid}  action   (see Proposition \ref{prop:fib.xi}). 
					Second we establish the existence of a canonical  
					\emph{Lie 2-group action} on $\CH(\CF)$ whose orbits lie inside the $\Xi$-fibers, but which might fail to be the whole fiber (see Prop. \ref{prop:Lie.2.grpACTION} and Cor. \ref{prop:Lie.2.grpFIBER}).

					\subsection*{Groupoid actions on the holonomy groupoid}\label{sec:groidorbits}

					Here we characterize the fibers of $\Xi$ as the orbits of a groupoid action. We start by showing that the $G$ action on $P$ admits a canonical lift to $\CH(\CF)$.
					
					\begin{lem}\label{def:gW}
						Let $\widehat{g}\colon P\fto P$ be the diffeomorphism given by the action of $g\in G$. Take a path holonomy atlas $\CU$ and a bisubmersion $W\in \CU$. The  triple $$gW:=(W,\bs_g:=\widehat{g}\circ\bs,\bt_g:=\widehat{g}\circ\bt)$$ is a bisubmersion. Moreover $gW$ is adapted to $\CU$.
					\end{lem}  
					\begin{proof}
						{Because the $G$-action preserves $\cF$, the pullback foliation $\widehat{g}^{-1}\CF$ equals $\CF$, implying that $gW$ is a bisubmersion.}
						
						{We prove that $gW$ is adapted to $\CU$. Notice that   $g(W_1\circ W_2)=gW_1\circ gW_2$ for any $W_1,W_2\in \CU$. Hence it is sufficient to assume that $W$ is a path holonomy bisubmersion, as any element in the path-holonomy atlas is a composition of such elements.
							
							Denote by $v_1,\cdots,v_n\in \CF$ the vector fields that give rise to the path holonomy bisubmersion $W$ (hence $W\subset \RR^n\times P$). Consider the push-forward vector fields $\widehat{g}_*v_1,\cdots,\widehat{g}_*v_n \in \CF$. The associated path-holonomy bisubmersion is defined on
							$$W':=\{(v,gp)\st (v,p)\in W\}\subset \RR^n\times P.$$
							Since $gW\fto W', (v,p)\mapsto (v,gp)$ is an isomorphism of bisubmersions and since $W'$ is adapted to $\CU$ (being a path holonomy bisubmersion), we conclude that $gW$ is adapted to $\CU$.}
					\end{proof}
					
					Now {we introduce} the {\bf lifted action} 
					
					\begin{equation}\label{eq:vecstar}
						\vec{\star}\colon{G} \times \CH(\cF)\fto \CH(\cF), \hspace{1cm}
						g\vec{\star} [v]:=[v]_{gW}  
					\end{equation}
					where, for any $v$ in a path holonomy bisubmersion $W$, we denote by 
					{$[v]_{gW}$}  the class of $v$ regarded as an element of $gW$.
					{This is clearly well-defined and indeed a Lie group action.}
					{Further, this action is by groupoid automorphisms:}

					{\begin{lem}\label{lem:cover}
							For all $g\in G$ the map $g\vec{\star} (-)\colon \CH(\cF)\fto \CH(\cF)$ is a groupoid morphism covering the diffeomorphism $\widehat{g}\colon P\fto P$. 
					\end{lem}}
					
					\begin{proof}
						Using the construction of $g\vec{\star}(-)$ it is clear that 
						the source and target map commute with the map $\widehat{g}$. {Further
							$g\vec{\star}(-)$ preserves the groupoid composition since
							$g(W_1\circ W_2)=gW_1\circ gW_2$ for any path-holonomy bisubmersions $W_1,W_2$.} 
					\end{proof}

					{\begin{lem}\label{lem:orbits}
							The  orbits of the lifted action $\vec{\star}$ lie in the fibers of $\Xi:\CH(\CF)\fto \CH(\CF_M)$.
					\end{lem}}
					\begin{proof} 
						The morphism $\Xi$ is induced by the identity map from any source connected atlas $\CU$ for $\CF$ to the atlas $\pi\CU:=\{(U, \pi\circ\bt, \pi\circ\bs): U\in \CU\}$ for $\cF_M$,  see  the characterization of $\Xi$ given in the proof of Thm \ref{thm:sur.hol}.
						
						Fix $g\in G$, and $u\in U\in \CU$. {By the above and since $\pi \circ \widehat{g}=\pi$,} the images under $\Xi$ of both $[u]$ and $g\vec{\star} [u]=[u]_{gU}$
						are the class of the element $u\in (U, \pi\circ\bt, \pi\circ\bs)$. In particular,
						$\Xi([u])=\Xi(g\vec{\star} [u])$, showing the desired statement.
					\end{proof}

					\begin{ex}\label{ex:ri.regfol}({\bf Regular foliations}) As discussed in Rem. \ref{rem:hol.regfol}, if $\CF$ is a regular foliation there is a groupoid morphism $Q\colon \Pi(\CF)\fto \CH(\CF)$ where $\Pi(\CF)$ is the {monodromy groupoid, consisting of homotopy classes of paths lying in the leaves of $\CF$.} The Lie group $G$ acts canonically on $\Pi(\CF)$ 
						{by translating paths}: $G\times \Pi(\CF)\fto \Pi(\CF); (g,[\gamma])\mapsto [g\gamma]$. It is easy to see that {$Q$ is $G$-equivariant, i.e.} 
						$Q(g[\gamma])=g\vec{\star}(Q[\gamma])$.
					\end{ex}

					{Notice that the Lie group action of $G$ on $P$, the groupoid morphism 
						$\Xi\colon \CH(\CF)\fto \CH(\CF_M)$ (see eq. \eqref{eq:Xi}) and the lifted Lie group action
						$\vec{\star}$ of $G$ on  $\CH(\cF)$  (see eq. \eqref{eq:vecstar}) fit in the abstract setting described at in Prop. \ref{cor:grpd.act.hol}, thanks to Lemma \ref{lem:cover} and Lemma \ref{lem:orbits}.} 
					Therefore we can apply Prop. \ref{cor:grpd.act.hol} to obtain a complete description of the $\Xi$-fibers:
					
					{\begin{prop}\label{prop:fib.xi}
							Denote $\CK:= \ker \Xi$. The topological groupoid $\CK\times G$ defined in Prop. \ref{cor:grpd.act.hol}  acts on  $\CH(\CF)$ by	
							$$((\chi,g),\xi)\mapsto \chi\circ(g\vec{\star}\xi).$$
							The orbits of this action are precisely the fibers of $\Xi$.	
					\end{prop}}
					
					An instance of Prop. \ref{prop:fib.xi} is Example \ref{ex:non.sing.q}, where we have $G=S^1$ and $\ker\Xi =1_P$.

					\subsection*{A canonical Lie 2-group action on the holonomy groupoid}\label{subsec:2gract}

					Here we prove that there always is a Lie 2-group action on $\CH(\CF)$ whose orbits lie inside the fibers of the morphism $\Xi$. In general however the orbits do not coincide with the (connected components of) the $\Xi$-fibers.
					The formulae for this Lie 2-group action are suggested by the special case we will spell out in \S\ref{subsec:alternative}.
					
					Denote the Lie algebra of $G$ by $\g$, and by $v_x\in \vX(P)$ the generator of the action corresponding to $x\in \g$.
					
					\begin{lem}\label{lem:ideal}
						The subspace $\h:=\{x\in \g: v_x\in \widehat{\cF}\}$ is a Lie ideal of $\g$.
					\end{lem}
					\begin{proof}
						Since $\cF$ is $G$-invariant, for all $y\in \g$ we have $[v_y,\cF]\subset \cF$, or equivalently $[v_y,\widehat{\cF}]\subset \widehat{\cF}$.
						Let $x\in \h$. Then for all $y\in \g$ we have $v_{[y,x]}=[v_y,v_x]\in \widehat{\cF}$, that is,
						$[y,x]\in \h$. 
					\end{proof}
					
					Denote by $H$ the unique connected Lie subgroup of $G$ with Lie algebra $\h$. 
					Lemma \ref{lem:ideal} implies that $H$ is a normal subgroup, hence as in Example \ref{ex:GH} we obtain a Lie 2-group $H\rtimes G\soutar G$.

					We  define {a Lie group action of $H$} of $\CH(\CF)$. It is not by groupoid automorphisms, unlike the  lifted $G$-action $\vec{\star} $ introduced in \S \ref{sec:groidorbits}, {but rather it preserves every source fiber}. In order to do so, we need a lemma.
					
					\begin{lem}\label{lem:phi}
						{There is a canonical groupoid morphism   $$\phi\colon {H}\times P\to \CH(\CF),$$ where ${H}\times P$ denotes the transformation groupoid of the ${H}$-action on $P$  {obtained restring the action of $G$.}
							
							{The morphism $\phi$ can be described as follows:} take $({h},p)\in {H}\times P$ and denote by $\widehat{{h}}\colon P\fto P$ the diffeomorphism corresponding to ${h}$ under the $G$-action. Then
							$\phi({h},p)$ is the unique element of $\CH(\CF)$ carrying the diffeomorphism $\widehat{{h}}$ near $p$.} 
					\end{lem}

					\begin{proof}[Proof of Lemma \ref{lem:phi}]
						{Denote by $\cF_{H}$ the regular foliation on $P$ by orbits of the $H$-action.
							Its holonomy groupoid is exactly $H\times P$, as follows from \cite[Ex. 3.4(4)]{AndrSk} (use that the Lie groupoid $H\times P$ gives rise to the foliation $\cF_{H}$ and is effective, i.e. the identity diffeomorphism on $M$ is carried only by identity elements of the Lie groupoid,  
							due to the freeness of the action).} 
						
						Since $\cF_{H}\subset \cF$, we are done applying \cite[Lemma 4.4]{SingSub} in the special case of the pair groupoid over $P$. 
						
						The description of $\phi$ given in the statement holds since $\phi$ is a groupoid morphism covering $Id_P$.
					\end{proof}

					\begin{lem}\label{lem:Lie.2.grpFIBER}
						{   We have $\phi(H\times P)\subset \CK:=\ker(\Xi)$. }
					\end{lem}
					\begin{proof}
						We use Lemma \ref{lem:phi}.
						A point $\phi(h,p)$ of the l.h.s.  carries near $p$ the diffeomorphism $\widehat{{h}}$ 
						(the diffeomorphism corresponding to ${h}$ under the $G$-action). The element $\phi(h,p)$ admits a representative $u$ in a source connected atlas 
						$(U,\bt, \bs)$ for $\cF$. By Lemma \ref{lem:all.sconbisub} the trio $(U,\pi \circ\bt, \pi\circ\bs)$ is a bisubmersion for $\CF_M$. The point $u$, viewed as a point in $(U,\pi \circ\bt, \pi\circ\bs)$, carries $Id_M$, {since $\widehat{{h}}$ preserves each $\pi$-fiber}. This implies that  $\Xi([u])=1_{\pi(q)}$, by   the characterization of $\Xi$ given in the proof of Thm. \ref{thm:sur.hol}.
					\end{proof}
					
					{\begin{ex}\label{ex:le.regfol}({\bf Regular foliations}) As discussed in Rem. \ref{rem:hol.regfol}, if $\CF$ is a regular foliation there is a groupoid morphism  $Q\colon \Pi(\CF)\fto \CH(\CF)$. The orbits of the $H$-action lie inside the leaves of $\CF$, {hence} for every path $h(t)$ {in $H$} and $p\in P$ the homotopy class $[h(t)p]$ {is an element of} $\Pi(\CF)$. Moreover, the freeness of the $G$-action implies that the elements $[h(t)p],[\widetilde{h}(t)p]\in \Pi(\CF)$ have the same holonomy if and only if $h(1)=\widetilde{h}(1)$ and $h(0)=\widetilde{h}(0)$. Therefore there is a well-defined injective groupoid morphism 
							$$H\times P\fto \CH(\CF);\;\; (h,p)\mapsto Q[h(t)p],$$
							where $h(t)$ is any path in $H$ with $h(0)=e$ and $h(1)=h$. This morphism is precisely $\phi$. It is clear using Prop. \ref{prop:regXi} that its image lies inside $\CK=\ker(\Xi)$.
					\end{ex}}

					{Consider now the following {map}, obtained applying the morphism $\phi$ of Lemma \ref{lem:phi} and left-multiplying:}

					\begin{equation}\label{eq:left}
						\cev{\star}\colon {H} \times \CH(\cF)\fto \CH(\cF), \hspace{1cm}
						h\cev{\star} {\xi}:={\phi(h,\bt({\xi}))\circ {\xi}}.
					\end{equation}
					Notice that $\phi$ being a groupoid morphism implies that $\cev{\star}$ is group action. {We now assemble the group action $\cev{\star}$ and the lifted action $\vec{\star}$:}

					\begin{prop}\label{prop:Lie.2.grpACTION}
						The map $$\star\colon (H\rtimes G)\times \CH(\cF)\fto \CH(\cF), \;\;(h,g)\star \xi:=h\cev{\star} \left(g\vec{\star} \xi\right)$$
						is a Lie 2-group action.
					\end{prop}
					\begin{proof}
						We first observe that if $\xi\in \CH(\cF)$ carries a diffeomorphism $\psi$, then $g\,\vec{\star}\,\xi$ carries the diffeomorphism  $\widehat{g}\psi\widehat{g}^{-1}$.
						
						We also observe the following two facts, which hold because both the left and the right side carry  the same diffeomorphism  (as can be seen using the above observation), the definition of the holonomy groupoid and because of  Lemma \ref{prop:equiv2}:
						
						i) the map $\phi\colon H\times P\to \CH(\CF)$ satisfies the following equivariance property: $$g\,\vec{\star} \,\phi(h,p)=\phi(c_gh,gp),$$
						where $c_g$ denotes conjugation by $g$.

						ii) For all $h\in H$ and $\xi \in\CH(\CF)$ we have $$h\,\vec{\star}\,\xi=\phi(h,\bt(\xi))\circ \xi \circ \phi(h^{-1},h\bs(\xi)).$$
						
						To show that $\star$ is a group action, the main requirement   is to show that
						$((h_1,g_1)(h_2,g_2))\star\xi$ equals $(h_1,g_1)\star \big((h_2,g_2)\star\xi\big)$ for all $(h_i,g_i)\in H\rtimes G$ and $\xi \in\CH(\CF)$.
						This holds by a straightforward computation, in which the second term is re-written using the fact that $\phi$ is a groupoid morphism and is $G$-equivariant (fact i) above).
						
						To show that $\star$ is a groupoid morphism,  the main requirement is to show that $(h_1h_2,g_2)\star(\xi_1\circ \xi_2)$ and $\big((h_1,g_1)\star\xi_1\big)\circ \big((h_2,g_2)\star\xi_2\big)$ agree, where $g_1=h_2g_2$ and $\bs(\xi_1)=\bt(\xi_2)$. Upon using that  $\phi$ is a groupoid morphism and the action $\vec{\star}$ is by groupoid automorphisms, this boils down to applying\footnote{with $h:=h_2$ and $\xi:=g_2\vec{\star}\xi_1$} fact ii) above.
					\end{proof}

					{Lemma   \ref{lem:orbits} and Lemma \ref{lem:Lie.2.grpFIBER}  imply:}
					\begin{cor}\label{prop:Lie.2.grpFIBER}
						The orbits of the Lie 2-group action $\star$ of Prop. \ref{prop:Lie.2.grpACTION} lie inside the $\Xi$-fibers.
					\end{cor}

					\subsection*{An alternative description for the Lie 2-group action of Thm \ref{thm:act.q.fol}}\label{subsec:alternative}
					
					We obtained the formula for the Lie 2-group action of $H\rtimes G$ in \S\ref{subsec:2gract} by considering a special case, as we now explain. Assume the set-up of Thm. \ref{thm:act.q.fol}, in particular that  $\Gamma_c(\ker d\pi)\subset \CF$ (i.e., $\widehat{\CF}$ contains the infinitesimal generators of the  $G$-action). There  we  defined an action  of $G\rtimes G$ on $\CH(\CF)$ by means of the canonical isomorphism 
					$\varphi\colon\CH(\CF)\xrightarrow{\sim}\pi^{-1}\CH(\CF_M)$. 
					
					The goal of this subsection is to prove the following proposition:
					
					\begin{prop}\label{prop:sameaction}
						{Consider these Lie 2-group actions of $G\rtimes G$ on $\CH(\CF)$:
							\begin{itemize}
								\item the action\footnote{Note that, under our assumptions on $\CF$, in the setting of Prop. \ref{prop:Lie.2.grpACTION} we have $H=G$, because the Lie subalgebra $\h$ introduced in Lemma \ref{lem:ideal} equals  the whole of $\g$.}
								$\star$ described in Prop. \ref{prop:Lie.2.grpACTION},
								\item  the action described in Thm. \ref{thm:act.q.fol}. 
							\end{itemize}
							These two actions coincide. In other words, under the canonical isomorphism $\varphi\colon \CH(\CF)\xrightarrow{\sim} \pi^{-1}\CH(\CF_{M})$ given in Thm. \ref{thm:pullbackgroid} we have}
						$$\varphi((h,g)\star \xi)=(h,g)* \varphi(\xi)$$ 
						for all $(g,h)\in G\rtimes G\y \xi\in\CH(\CF)$, where  $*$ is the action given in eq. \eqref{eq:star}.	
					\end{prop}
					
					We will prove this statement by analyzing the restrictions of the actions to 
					$\{e\}\times G$ and $G\times\{e\}$. (Notice that these two subgroups generate $G\rtimes G$ as a group, since every element $(h,g)$ can be written as $(h,e)(e,g)$).
					{We denote by ${\varphi}\colon \CH(\CF)\xrightarrow{\sim} \pi^{-1}\CH(\CF_{M})$  the canonical isomorphism  given in Thm. \ref{thm:pullbackgroid}.}
					\begin{lem}\label{lem:raction}
						Under the  isomorphism ${\varphi}$,  
						the lifted action $\vec{\star} $ of $G$ introduced in \S\ref{sec:groidorbits} 
						and the restriction of the Lie $2$-group action $*$ of {eq. \eqref{eq:star}}
						agree:  $${\varphi}(g\vec{\star} \xi)=(e,g)* {\varphi}(\xi)$$
						for all $g\in G$ and $\xi\in \CH(\CF)$.
					\end{lem}
					
					\begin{proof}
						Let $\CU$ be a path holonomy atlas for $\CF_M$. Recall that by Rem. \ref{rem:varphi}, the pullback-atlas $\pi^{-1}\CU$ is an atlas for $\pi^{-1}(\CF_M)$  equivalent to a path-holonomy atlas.  
						Fix $\xi\in \CH(\CF)$. Take a representative $w$ in a bisubmersions $W$ in the
						path holonomy atlas of $\pi^{-1}(\CF_M)$.  
						By the above, there is a path holonomy bisubmersion $U$ in $\CU$ and a locally defined morphism of bisubmersions $$\tau\colon (W,\bt,\bs)\fto   (\pi^{-1} U,\bt,\bs)$$ 
						mapping $w$ to some point $(p,v,q)$.
						By definition,  $\varphi([w])=(p,[v],q)$.
						
						Now fix $g\in G$. Recall that the bisubmersion $gW:=(W, \widehat{g}\circ\bs, \widehat{g}\circ\bt)$ was defined in Lemma \ref{def:gW}.	
						The same map $\tau$ is also a morphism of bisubmersions $$gW \fto   (\pi^{-1} U,\widehat{g}\circ\bt,\widehat{g}\circ\bs).$$ 
						The latter bisubmersion is isomorphic to  $(\pi^{-1} U, \bt,\bs)$ via
						$(p',v',q')\mapsto (gp',v',gq')$. By composition we obtain a morphism of bisubmersions
						$gW \to (\pi^{-1} U, \bt,\bs)$ mapping $w$ to $(gp,v,gq)$. Hence $g\vec{\star} [w]:=[w]_{gW}$ agrees with $[(gp,v,gq)]\in \CH(\CF)$, and therefore under $\varphi$ it is mapped to 
						$(gp,[v],gq)=(e,g)* \varphi([w])$.
					\end{proof}
					
					\begin{lem}\label{lem:laction}
						{Under the  isomorphism ${\varphi}$,  
							the  action $\cev{\star} $ of $H$ introduced in eq. \eqref{eq:left} 
							and the restriction of the Lie $2$-group action $*$ of {eq. \eqref{eq:star}}
							agree:} $$\varphi(h\cev{\star} \xi)=(h,e)* \varphi(\xi)$$  for all $h\in G$ and $\xi\in \CH(\CF)$.
					\end{lem}
					
					\begin{proof} Take {an arbitraty element $\xi\in \CH(\CF)$ and} a path holonomy atlas $\CU$ for $\CF_M$. Let $(p,u,q)\in \pi^{-1}U\in \pi^{-1}\CU$ { be a representative of $\xi$}, and let $f$ be a local diffeomorphism carried at $(p,u,q)$.
						
						Fix $h\in G$, and  denote by {$\widehat{h}\colon P\fto P$} the diffeomorphism {corresponding to $h$ under the $G$-action}. Note that {the transformation groupoid $G\times P$ carries}  the diffeomorphism $\widehat{h}$ at $(h,p)$. Hence any representative  of ${\phi(h,p)}\in \CH(\CF)$ in $\pi^{-1}\CU$ carries this diffeomorphism, {where $\phi$ is the groupoid morphism of Lemma \ref{lem:phi}}.
						In turn, this implies that any representative   of $\phi(h,p)\circ[(p,u,q)]\in \CH(\CF)$ will carry $\widehat{h}\circ f$. Note that $(hp,u,q)\in \pi^{-1}U$ also carries $\widehat{h}\circ f$. {By the definition of holonomy groupoid (using the definition of the holonomy groupoid and Lemma \ref{prop:equiv2}) it follows that} $h\cev{\star} [(p,u,q)]=[(hp,u,q)]$.
						
						We conclude that
						$$\varphi(h\cev{\star} [(p,u,q)])=\varphi([(hp,u,q)])=(hp,[u],q)=(h,e)* \varphi([(p,u,q)]),$$
						{using in the second equality the description of the isomorphism $\varphi$ given in Rem. \ref{rem:varphi}.}
					\end{proof}
					
					\begin{proof}[Proof of Prop. \ref{prop:sameaction}] The proposition follows from
						$$\varphi((h,g)\star \xi) = \varphi(h\cev{\star}\left(g\vec{\star} \xi\right))=(h,e)*((e,g)*\varphi( \xi)) = (h,g)*\varphi( \xi),$$
						{where we used Lemmas \ref{lem:raction} and \ref{lem:laction} in the second equality.}
					\end{proof}
					
					\begin{rem}\label{rem:imker}
						{The image of $\phi\colon G\times P\fto \CH(\CF)$ is the kernel of $\Xi\colon \CH(\CF)\fto \CH(\CF_M)$.} 
						{Indeed, for every $(h,q)\in G\times P$, we have $\phi(h,q)=h\cev{\star} 1_q\in \CH(\CF)$. Under the identification  $\varphi\colon \CH(\CF)\xrightarrow{\sim} \pi^{-1}\CH(\CF_{M})$,   this element corresponds to $(h,e)* \varphi(1_q)=(hq,1_{\pi(q)},q)\in \pi^{-1}\CH(\CF_{M})$ by Lemma \ref{lem:laction}. Under the same identification, $\ker(\Xi)$ corresponds to $\ker(pr_2)=P\times_{\pi} M \times_{\pi} P$, by Prop. \ref{prop:normalq}.}
					\end{rem}

					\cite{Garmendia2019}
					
					
					
					
					\backmatter
					
					\includebibliography
					\newcommand{\etalchar}[1]{$^{#1}$}

					
					

\begin{thebibliography}{BFF{\etalchar{+}}78b}
						
						\bibitem[AS09]{AndrSk}
						Iakovos Androulidakis and Georges Skandalis.
						\newblock The holonomy groupoid of a singular foliation.
						\newblock {\em J. Reine Angew. Math.}, 626:1--37, 2009.
						
						\bibitem[AZ13]{AZ1}
						Iakovos {Androulidakis} and Marco {Zambon}.
						\newblock {Smoothness of holonomy covers for singular foliations and essential
							isotropy.}
						\newblock {\em {Math. Z.}}, 275(3-4):921--951, 2013.
						
						\bibitem[AZ14]{AZ2}
						Iakovos {Androulidakis} and Marco {Zambon}.
						\newblock {Holonomy transformations for singular foliations.}
						\newblock {\em {Adv. Math.}}, 256:348--397, 2014.
						
						\bibitem[AZ16]{AMsheaf}
						Iakovos Androulidakis and Marco Zambon.
						\newblock Stefan–sussmann singular foliations, singular subalgebroids and
						their associated sheaves.
						\newblock {\em International Journal of Geometric Methods in Modern Physics},
						13(Supp. 1):1641001, 2016.
						
						\bibitem[AZ17]{AZ5}
						Iakovos Androulidakis and Marco Zambon.
						\newblock Almost regular {P}oisson manifolds and their holonomy groupoids.
						\newblock {\em Selecta Math. (N.S.)}, 23(3):2291--2330, 2017.
						
						\bibitem[BBLM]{HenT}
						Francis Bischoff, Henrique Bursztyn, Hudson Lima, and Eckhard Meinrenken.
						\newblock Deformation spaces and normal forms around transversals.
						\newblock {\em arXiv:1807.11153}.
						
						\bibitem[BFF{\etalchar{+}}78a]{DefQ2}
						F.~Bayen, M.~Flato, C.~Fronsdal, A.~Lichnerowicz, and D.~Sternheimer.
						\newblock Deformation theory and quantization. {I}. {D}eformations of
						symplectic structures.
						\newblock {\em Ann. Physics}, 111(1):61--110, 1978.
						
						\bibitem[BFF{\etalchar{+}}78b]{DefQ1}
						F.~Bayen, M.~Flato, C.~Fronsdal, A.~Lichnerowicz, and D.~Sternheimer.
						\newblock Deformation theory and quantization. {II}. {P}hysical applications.
						\newblock {\em Ann. Physics}, 111(1):111--151, 1978.
						
						\bibitem[BM16]{RM2016}
						Alcides Buss and Ralf Meyer.
						\newblock Iterated crossed products for groupoid fibrations.
						\newblock {\em arXiv:1604.02015}, 2016.
						
						\bibitem[Bro72]{BrownCoefficients}
						Ronald Brown.
						\newblock Groupoids as coefficients.
						\newblock {\em Proc. London Math. Soc. (3)}, 25:413--426, 1972.
						
						\bibitem[Bro06]{TopGrpds}
						Ronald Brown.
						\newblock {\em Topology and groupoids}.
						\newblock BookSurge, LLC, Charleston, SC, 2006.
						\newblock Third edition of {\it Elements of modern topology}. McGraw-Hill, New
						York, 1968.
						
						\bibitem[BS76]{CroMod}
						Ronald Brown and Christopher~B. Spencer.
						\newblock G-groupoids, crossed modules and the fundamental groupoid of a
						topological group.
						\newblock {\em Indag. Math.}, 38(4):296--302, 1976.
						
						\bibitem[CdS01]{AnaSympl}
						Ana Cannas~da Silva.
						\newblock {\em Lectures on symplectic geometry}, volume 1764 of {\em Lecture
							Notes in Mathematics}.
						\newblock Springer-Verlag, Berlin, 2001.
						
						\bibitem[CdSW99]{WAPois}
						Ana Cannas~da Silva and Alan Weinstein.
						\newblock {\em Geometric models for noncommutative algebras}, volume~10 of {\em
							Berkeley Mathematics Lecture Notes}.
						\newblock American Mathematical Society, Providence, RI; Berkeley Center for
						Pure and Applied Mathematics, Berkeley, CA, 1999.
						
						\bibitem[CF03]{CrFeLie}
						Marius Crainic and Rui~Loja Fernandes.
						\newblock Integrability of {L}ie brackets.
						\newblock {\em Ann. of Math. (2)}, 157(2):575--620, 2003.
						
						\bibitem[CF04]{CFPois}
						Marius Crainic and Rui~Loja Fernandes.
						\newblock Integrability of {P}oisson brackets.
						\newblock {\em J. Differential Geom.}, 66(1):71--137, 2004.
						
						\bibitem[CM17]{JoaoCrainic2017}
						Marius {Crainic} and Joao~N. {Mestre}.
						\newblock {Orbispaces as differentiable stratified spaces}.
						\newblock {\em Letters in Mathematical Physics}, November 2017.
						
						\bibitem[Con79]{Connes1}
						Alain Connes.
						\newblock Sur la th\'{e}orie non commutative de l'int\'{e}gration.
						\newblock In {\em Alg\`ebres d'op\'{e}rateurs ({S}\'{e}m., {L}es
							{P}lans-sur-{B}ex, 1978)}, volume 725 of {\em Lecture Notes in Math.}, pages
						19--143. Springer, Berlin, 1979.
						
						\bibitem[Con82]{Connes3}
						Alain Connes.
						\newblock A survey of foliations and operator algebras.
						\newblock In {\em Operator algebras and applications, {P}art {I} ({K}ingston,
							{O}nt., 1980)}, volume~38 of {\em Proc. Sympos. Pure Math.}, pages 521--628.
						Amer. Math. Soc., Providence, R.I., 1982.
						
						\bibitem[CS84]{Connes2}
						Alain Connes and Georeges. Skandalis.
						\newblock The longitudinal index theorem for foliations.
						\newblock {\em Publ. Res. Inst. Math. Sci.}, 20(6):1139--1183, 1984.
						
						\bibitem[CZ13]{CZ2}
						Alberto~S. Cattaneo and Marco Zambon.
						\newblock A supergeometric approach to {P}oisson reduction.
						\newblock {\em Comm. Math. Phys.}, 318(3):675--716, 2013.
						
						\bibitem[Deb01]{DebordJDG}
						Claire Debord.
						\newblock Holonomy groupoids of singular foliations.
						\newblock {\em J. Diff. Geom.}, 58(3):467--500, 2001.
						
						\bibitem[Deb13]{Debord2013}
						Claire Debord.
						\newblock Longitudinal smoothness of the holonomy groupoid.
						\newblock {\em C. R. Math. Acad. Sci. Paris}, 351(15-16):613--616, 2013.
						
						\bibitem[dH13]{MatiasME}
						Matias~L. del Hoyo.
						\newblock Lie groupoids and their orbispaces.
						\newblock {\em Port. Math.}, 70(2):161--209, 2013.
						
						\bibitem[Dir47]{Dirac1}
						P.~A.~M. Dirac.
						\newblock {\em The {P}rinciples of {Q}uantum {M}echanics}.
						\newblock Oxford, at the Clarendon Press, 1947.
						\newblock 3d ed.
						
						\bibitem[Ehr65]{EhrGrpd}
						Charles Ehresmann.
						\newblock {\em Cat\'{e}gories et structures}.
						\newblock Dunod, Paris, 1965.
						
						\bibitem[Gin01]{GinzburgGrot}
						Viktor~L. Ginzburg.
						\newblock Grothendieck groups of {P}oisson vector bundles.
						\newblock {\em J. Symplectic Geom.}, 1(1):121--169, 2001.
						
						\bibitem[GY18]{AutOri}
						Alfonso Garmendia and Ori Yudilevich.
						\newblock On the inner automorphisms of a singular foliation.
						\newblock {\em Mathematische Zeitschrift, online version:
							https://doi.org/10.1007/s00209-018-2212-0, 5 pages}, Nov 2018.
						
						\bibitem[GY19]{Garmendia2019}
						Alfonso Garmendia and Ori Yudilevich.
						\newblock On the inner automorphisms of a singular foliation.
						\newblock {\em Mathematische Zeitschrift}, 293(1):725--729, Oct 2019.
						
						\bibitem[GZ19]{ME2018}
						Alfonso Garmendia and Marco Zambon.
						\newblock Hausdorff morita equivalence of singular foliations.
						\newblock {\em Annals of Global Analysis and Geometry}, 55(1):99--132, Feb
						2019.
						
						\bibitem[Hae84]{HaefligerHolonomieClassifiants}
						Andr{\'e} Haefliger.
						\newblock Groupo\"\i des d'holonomie et classifiants.
						\newblock {\em Ast{\'e}risque}, (116):70--97, 1984.
						\newblock Transversal structure of foliations (Toulouse, 1982).
						
						\bibitem[Kos77]{kostant}
						Bertram Kostant.
						\newblock Graded manifolds, graded lie theory, and prequantization.
						\newblock pages 177--306. Lecture Notes in Math., Vol. 570, 1977.
						
						\bibitem[LGLS18]{SylvainArticle}
						Camille Laurent-Gengoux, Sylvain Lavau, and Thomas Strobl.
						\newblock The universal lie $\infty$-algebroid of a singular foliation.
						\newblock {\em arXiv:1806.00475}, 2018.
						
						\bibitem[LGSX09]{LauMatXu}
						Camille Laurent-Gengoux, Mathieu Sti\'enon, and Ping Xu.
						\newblock Non-abelian differentiable gerbes.
						\newblock {\em Adv. Math.}, 220(5):1357--1427, 2009.
						
						\bibitem[Mac05]{MK2}
						Kirill C.~H. Mackenzie.
						\newblock {\em General theory of {L}ie groupoids and {L}ie algebroids}, volume
						213 of {\em London Mathematical Society Lecture Note Series}.
						\newblock Cambridge University Press, Cambridge, 2005.
						
						\bibitem[MM03]{MRIntr}
						I.~Moerdijk and J.~Mr\v{c}un.
						\newblock {\em Introduction to foliations and {L}ie groupoids}, volume~91 of
						{\em Cambridge Studies in Advanced Mathematics}.
						\newblock Cambridge University Press, Cambridge, 2003.
						
						\bibitem[MM05]{MdMkGrd}
						I.~Moerdijk and J.~Mr{\v{c}}un.
						\newblock Lie groupoids, sheaves and cohomology.
						\newblock In {\em Poisson geometry, deformation quantisation and group
							representations}, volume 323 of {\em London Math. Soc. Lecture Note Ser.},
						pages 145--272. Cambridge Univ. Press, Cambridge, 2005.
						
						\bibitem[Mol88]{MolME}
						Pierre Molino.
						\newblock {\em Riemannian foliations}, volume~73 of {\em Progress in
							Mathematics}.
						\newblock Birkh\"{a}user Boston, Inc., Boston, MA, 1988.
						\newblock Translated from the French by Grant Cairns, With appendices by
						Cairns, Y. Carri\`ere, \'{E}. Ghys, E. Salem and V. Sergiescu.
						
						\bibitem[Mol94]{MolinoOrbit-LikeFol}
						Pierre Molino.
						\newblock Orbit-like foliations.
						\newblock In {\em Geometric study of foliations ({T}okyo, 1993)}, pages
						97--119. World Sci. Publ., River Edge, NJ, 1994.
						
						\bibitem[Pro96]{DPronk}
						Dorette~A. Pronk.
						\newblock Etendues and stacks as bicategories of fractions.
						\newblock {\em Compositio Math.}, 102(3):243--303, 1996.
						
						\bibitem[Ren80]{GrpdCAlg}
						Jean Renault.
						\newblock {\em A groupoid approach to {$C^{\ast} $}-algebras}, volume 793 of
						{\em Lecture Notes in Mathematics}.
						\newblock Springer, Berlin, 1980.
						
						\bibitem[RS13]{GrudMSch}
						Gerd Rudolph and Matthias Schmidt.
						\newblock {\em Differential geometry and mathematical physics. {P}art {I}}.
						\newblock Theoretical and Mathematical Physics. Springer, Dordrecht, 2013.
						\newblock Manifolds, Lie groups and Hamiltonian systems.
						
						\bibitem[Wan17]{RoyThesis}
						Roy Wang.
						\newblock {\em {On Integrable Systems Rigidity for PDEs with Symmetry}}.
						\newblock PhD thesis, Utrecht University, 12 2017.
						
						\bibitem[Wan18]{Kwang}
						Kirsten Wang.
						\newblock {\em {Proper Lie groupoids and their orbitspaces}}.
						\newblock PhD thesis, Universiteit van Amsterdam, 9 2018.
						
						\bibitem[Wei83]{We}
						Alan Weinstein.
						\newblock The local structure of {P}oisson manifolds.
						\newblock {\em J. Differential Geom.}, 18(3):523--557, 1983.
						
						\bibitem[Wei85]{WeinPois83}
						Alan Weinstein.
						\newblock Errata and addenda: ``{T}he local structure of {P}oisson manifolds''
						[{J}. {D}ifferential {G}eom. {\bf 18} (1983), no. 3, 523--557; {MR}0723816
						(86i:58059)].
						\newblock {\em J. Differential Geom.}, 22(2):255, 1985.
						
						\bibitem[Win83]{WinkGrpd}
						Horst~E. Winkelnkemper.
						\newblock The graph of a foliation.
						\newblock {\em Ann. Global Anal. Geom.}, 1(3):51--75, 1983.
						
						\bibitem[Xu91]{xuME}
						Ping Xu.
						\newblock Morita equivalence of {P}oisson manifolds.
						\newblock {\em Comm. Math. Phys.}, 142(3):493--509, 1991.
						
						\bibitem[Zam18]{SingSub}
						Marco Zambon.
						\newblock Singular subalgebroids.
						\newblock {\em arXiv:1805.02480}, 05 2018.
						
					\end{thebibliography}
				\end{document}